\documentclass[reqno,11pt]{amsart}
\usepackage{a4wide}
\usepackage{amssymb}
\usepackage{mathtools}

\newcommand{\cR}{\mathcal R}
\newcommand{\ph}{\varphi}
\newcommand{\Cinfini}{\mathcal C^\infty}
\newcommand{\F}{\mathcal F}
\newcommand{\zz}{z}
\newcommand{\lamc}{\check\lambda}
\newcommand{\sigc}{\check\sigma}
\newcommand{\varc}{\check\varepsilon}
\newcommand{\bc}{\check{b}}
\newcommand{\loc}{\textnormal{loc}}
\newcommand{\astar}{\alpha^\star}
\newcommand{\mm}{A}

\DeclareMathOperator{\sech}{sech}
\DeclareMathOperator{\MM}{M}

\newtheorem{theorem}{Theorem}[section]
\newtheorem{lemma}[theorem]{Lemma}
\newtheorem{proposition}[theorem]{Proposition}

\theoremstyle{definition}
\newtheorem{definition}[theorem]{Definition}
\theoremstyle{remark}\newtheorem{remark}{Remark}[section]
\numberwithin{equation}{section}

\begin{document}
\title[Flattening solitary waves for critical gKdV]{Full family of flattening solitary waves for the mass critical generalized KdV equation}

\author[Y. Martel]{Yvan Martel}
\address{CMLS, \'Ecole Polytechnique, CNRS, 91128 Palaiseau, France}
\email{yvan.martel@polytechnique.edu}

\author[D. Pilod]{Didier Pilod}
\address{Department of Mathematics, University of Bergen, Postbox 7800, 5020 Bergen, Norway }
\email{Didier.Pilod@uib.no}

\thanks{D. P. was supported by a grant from the Trond Mohn Foundation.}

\begin{abstract}
For the mass critical generalized KdV equation 
$\partial_t u + \partial_x (\partial_x^2 u + u^5)=0$ on $\mathbb R$,
we construct a full family of flattening solitary wave solutions.
Let $Q$ be the unique even positive solution of $Q''+Q^5=Q$.
For any $\nu\in (0,\frac 13)$, there exist 
global (for $t\geq 0$)  solutions of the equation with the asymptotic behavior
\begin{equation*}
u(t,x)= t^{-\frac{\nu}2} Q\left(t^{-\nu} (x-x(t))\right)+w(t,x)
\end{equation*}
where, for some $c>0$,
\begin{equation*}
x(t)\sim c t^{1-2\nu} \quad \mbox{and}\quad 
\|w(t)\|_{H^1(x>\frac 12 x(t))} \to 0\quad 
\mbox{as $t\to +\infty$.}
\end{equation*}
Moreover, the initial data for such solutions can be taken arbitrarily close to a solitary wave in the energy space. The long-time flattening of the solitary wave is forced by  a slowly decaying tail in the initial data.

This result and its proof are inspired and complement recent blow-up results for the critical generalized KdV equation.
This article is also motivated by previous constructions of exotic behaviors close to solitons for 
other nonlinear dispersive equations such as the
energy-critical wave equation.
\end{abstract}

\maketitle

\section{Introduction}
\subsection{Motivation and main result}
We consider the $L^2$-critical generalized Korteweg-de Vries equation (gKdV)
\begin{equation} \label{gkdv}
\partial_tu+\partial_x\big(\partial_x^2u+u^5\big)=0  ,\quad (t,x) \in \mathbb R\times \mathbb R,
\end{equation}
where $u(t,x)$ is a real-valued function. 

The mass $M(u)$ and the energy $E(u)$ are (formally) conserved by the flow of \eqref{gkdv} where 
\begin{equation} \label{mass_energy}
M(u)=\int_{\mathbb R} u^2 \, dx \quad \text{and} \quad E(u)=\frac12 \int_{\mathbb R} (\partial_xu)^2 \, dx-\frac16 \int_{\mathbb R} u^6 \, dx  .
\end{equation}
We recall the scaling invariance: if $u$ is a solution to \eqref{gkdv}, then for any $\lambda>0$
\begin{equation*}u_{\lambda}(t,x):=\lambda^{\frac12} u(\lambda^3 t, \lambda x)\end{equation*} is also a solution to \eqref{gkdv}.

Recall that the Cauchy problem for \eqref{gkdv} is locally well-posed in the energy space $H^1(\mathbb R)$ by the work of Kenig, Ponce and Vega \cite{KPV,KPV2}: for any $u_0 \in H^1(\mathbb R)$, there exists a unique (in a certain sense) \emph{maximal solution} of \eqref{gkdv} in $\mathcal{C}\big([0,T^{\star}): H^1(\mathbb R)\big)$ satisfying $u(0,\cdot)=u_0$. Moreover, we have the \emph{blow-up alternative}: 
\begin{equation*} \text{if} \quad T^{\star}<+\infty, \quad \text{then} \quad \lim_{t \uparrow T^{\star}} \|\partial_xu(t)\|_{L^2}=+\infty  . 
\end{equation*}
 For such $H^1$ solutions, the quantities $M(u)(t)$ and $E(u)(t)$ are conserved on $[0,T^{\star})$.

We recall the family of solitary wave solutions of \eqref{gkdv}. Let
 $Q(x)=\big(3\sech^2(2x)\big)^{1/4}$ be the unique (up to translation) positive solution of the equation
\begin{equation} \label{eq:Q}
-Q''+Q-Q^5=0 \quad \mbox{on $\mathbb R$} .
\end{equation}
Then, the function
\begin{equation*} 
u(t,x)=\lambda_0^{-\frac12}Q\big(\lambda_0^{-1}(x-\lambda_0^{-2}t-x_0)\big), \quad 
\mbox{for any $(\lambda_0,x_0) \in
(0,+\infty) \times \mathbb R$} ,
\end{equation*}
is a solution of \eqref{gkdv}.
It is well-known that $E(Q)=0$ and that $Q$ is related to the following sharp Gagliardo-Nirenberg inequality (see \cite{Wei}) 
\begin{equation} \label{sharpGN}
\frac 13 \int_{\mathbb R} \phi^6 \le \left(\frac{\int_{\mathbb R} \phi^2}{\int_{\mathbb R} Q^2} \right)^2 \int_{\mathbb R} (\partial_x\phi)^2, \quad \forall \, \phi \in H^1(\mathbb R)  .
\end{equation}
It follows from \eqref{sharpGN} and the conservation of the mass and the energy that 
any initial data $u_0 \in H^1(\mathbb R)$ satisfying $\|u_0\|_{L^2} < \|Q\|_{L^2}$ generates a global in time solution of \eqref{gkdv} that is also bounded in $H^1(\mathbb R)$.
 
Now, we summarize available results on blow-up solutions for \eqref{gkdv} in the case of
 initial data with mass equal or slightly above the threshold mass, \emph{i.e.} satisfying
\begin{equation*}
\|Q\|_{L^2}\leq \|u_0\|_{L^2}\leq (1+\delta_0) \|Q\|_{L^2} \quad \hbox{where}\quad 0<\delta_0\ll 1.
\end{equation*}
\begin{itemize}
\item At the threshold mass $\|u_0\|_{L^2}=\|Q\|_{L^2}$, there exists a unique (up to the invariances of the equation) blow-up solution $S(t)$
of the equation, which blows up in finite time (denoted by $T>0$) with the rate $\|S(t)\|_{H^1} \sim C(T-t)^{-1}$
as $t\to T$. See \cite{CoMa1,MaMeRa2}.
\item For mass slightly above the threshold, there exists a large set (including negative and zero energy solutions, and open in some topology) of blow-up solutions, with the blow-up rate $\|u(t)\|_{H^1} \sim C(T-t)^{-1}$
as $t\to T$. See \cite{MaMeRa1,Mjams} and other references therein.
\item In the neighborhood of the soliton for the same topology ($H^1$ solutions with suitable decay on the right),
there exists a $\mathcal C^1$ co-dimension one threshold manifold which separates the above stable blow-up behavior from 
solutions that eventually exit the soliton neighborhood by vanishing.
Solutions on the manifold are global and locally converge to the ground state $Q$ up to the invariances of the equation.
In this class of initial data, one thus obtains the following trichotomy: stable finite time blowup, soliton behavior or exit. See~\cite{MaMeNaRa,MaMeRa1,MaMeRa2}.
\item There also exists a large class of exotic finite time blow-up solutions, close to the family of solitons, enjoying blow-up rates of the form $\|u(t)\|_{H^1} \sim C (T-t)^{-\nu}$ for any $\nu>\frac {11}{13}$.
Note that the exponent $\frac{11}{13}$ does not seem sharp 
and it is an open question to determine the lowest finite time blow-up exponent for $H^1$ initial data.
Global solutions blowing up in infinite time with $\|u(t)\|_{H^1} \sim C t^{\nu}$ as $t\to \infty$,
were also constructed for any positive power $\nu>0$. See~\cite{MaMeRa3}.

Such exotic behaviors are generated by the interaction of the soliton with explicit slowly decaying tails added to the initial data.
Because of the tail, these $H^1$ solutions do not belong to the class where the trichotomy (blowup, soliton, exit) occurs.
\end{itemize}
We refer to the above mentioned articles and to the references therein for detailed results and previous references on the subject.

Recall that for the mass critical nonlinear Schr\"odinger equation (NLS), there exists a large class (stable in $H^1$) of blow-up solutions enjoying the so-called
$\log\log$ blow-up rate (see \cite{MR} and references therein), whereas (unstable) blow-up solutions with the conformal blow-up rate
$\|u(t)\|_{H^1}\sim C (T-t)^{-1}$ were also constructed by perturbation of the explicit minimal mass blow-up solution (\cite{BW,KS2,MRS}). Moreover, in the vicinity of the soliton, 
it is proved in \cite{Ra} that solutions cannot have a blow-up rate
strictly between the $\log\log$ rate and the conformal rate.
It is an open question to build solutions with a blow-up rate higher than the conformal one (see however \cite{MaRa} in the case of several solitons). 
The only available results concerning flattening solitons are deduced from the pseudo-conformal transformation applied to the solutions discussed above.
For the mass critical (NLS), the question of the existence of exotic behaviors is thus widely open. 

The systematic study of exotic blow-up behaviors was initiated by the articles \cite{KST1,KST2} for energy critical dispersive models, followed by \cite{DK,HR,JJ,KS}.
(We also refer to \cite{GNT} for the construction of exotic solutions in other contexts.)
The article \cite{DK}, where a class of flattening bubbles is constructed for the energy critical wave equation on $\mathbb R^3$, is particularly related to our work.
More precisely, $W$ being the unique radial positive solution of $\Delta W+W^5=0$ on $\mathbb R^3$,
it is proved in \cite{DK} that for any $|\nu|\ll 1$, there exist global (for positive time) solutions of $\partial_t^2 u =\Delta u + |u|^4 u$ such that $u(t,x)\sim t^{\nu/2} W (t^{\nu} x)$ as $t\to +\infty$; the case $0<\nu\ll 1$ corresponds to blow-up in infinite time, while
$0<-\nu\ll 1$ corresponds to flattening solitons.

Such construction is especially motivated by the \emph{soliton resolution conjecture}, which 
states that any global solution should decompose for large time into a certain number of decoupled solitons plus a dispersive part. 
We refer to \cite{DKM} and references therein for the proof of the soliton resolution conjecture for the $3$D critical wave equation 
in the radial case. It follows from \cite{DK} that some flexibility on the geometric parameters is necessary
 in the statement of the conjecture.

The above mentioned works are a strong motivation for investigating exotic behaviors related to flattening solitons in the context of mass critical dispersive models.
Our main result is the existence of such solutions for the critical generalized KdV equation.

\begin{theorem}\label{th:1}
Let any $\beta \in (\frac 13 ,1)$.
For any $\delta>0$, there exist $T_\delta>0$ and $u_0\in H^1(\mathbb R)$ with 
$\|u_0-Q\|_{H^1}\leq \delta$ such that the solution $u$ of \eqref{gkdv} with initial data $u_0$ is global 
for $t\geq 0$ and decomposes for all $t\geq 0$ as
\begin{equation*}
u(t,x)= \frac 1{\ell^{\frac 12}(t)} Q\left( \frac{x-x(t)}{\ell(t)}\right)+w(t,x)
\end{equation*}
where the functions $\ell(t)$, $x(t)$ and $w(t,x)$ satisfy
\begin{equation}\label{def:ell}
\ell(t)\sim\left(\frac t{T_\delta}\right)^{\frac {1-\beta}2},
\quad
x(t)\sim \frac{T_\delta}{\beta} \left(\frac t{T_\delta}\right)^{\beta}\quad \mbox{as $t\to +\infty$,}
\end{equation}
and
\begin{equation}\label{def:eta}
\sup_{t\geq 0}\|w(t)\|_{H^1}\leq \delta,\quad \ \lim_{t\to +\infty} \|w(t)\|_{H^1(x>\frac 12 x(t))}=0.
\end{equation} 
\end{theorem}
Theorem~\ref{th:1} states the existence of solutions arbitrarily close to the soliton $Q$ which eventually defocus in large time
with scaling $\ell(t)\sim t^{-\nu}$ where $\nu=(1-\beta)/2$ is any value in~$(0,\frac 13)$.
The values of the exponents and multiplicative constants in \eqref{def:ell} are consistent with the formal equation
$x'(t)= \ell^{-2}(t)$ relating the two geometrical parameters $x(t)$ and $\ell(t)$.

Note that by continuous dependence of the solution of \eqref{gkdv} with respect to the initial data, the constant $T_\delta$ in Theorem~\ref{th:1} satisfies $T_\delta\to \infty$ as $\delta\to 0$ .
The estimates in \eqref{def:ell} make sense only for $t\gg T_\delta$ when the flattening regime appears.
Of course, one can use the scaling invariance of the equation to generate solutions with different multiplicative constants
in \eqref{def:ell}. In the statement of Theorem~\ref{th:1}, the scaling is adjusted so that one can compare the initial data
with the soliton $Q$. We refer to Remark~\ref{rk:T0} for details.

We also notice that $w(t)$ does not converge to $0$ in $H^1(\mathbb R)$ as $t\to +\infty$; 
otherwise, it would hold $E(u(t))=0$ and $\int u^2(t)=\int Q^2$ and by variational arguments, $u(t)$ would be exactly a soliton.
However, the residue $w$ is arbitrarily small in $H^1$ and converges strongly to $0$ as $t\to \infty$ in the space-time region $x>\frac 12 x(t)\gg \ell(t)$ which largely includes the soliton.

To complement Theorem~\ref{th:1},
we prove in Section~\ref{S:6.5} that the solutions do not behave as solutions of the linear Airy equation $\partial_t v+\partial_x^3 v=0$ as $t\to \infty$ (non-scattering solutions).

We claim that the restriction $\beta\in(\frac 13,1)$ in Theorem~\ref{th:1} corresponds to the full range of relevant exponents.
Indeed, the exponent $\beta=\frac 13$ is related to self-similarity, and
in the region $x<t^{1/3}$, the question of existence or non-existence of coherent nonlinear structures is of different nature. See~\cite{MP} for several results in this direction.

As mentioned above, infinite time blow-up solutions with any positive power rate were constructed in \cite{MaMeRa3}. Thus 
 Theorem~\ref{th:1} essentially settles the question of all possible single soliton behaviors as $t\to+\infty$.
It also sheds some light on the classification of all possible behaviors in $H^1$, while the results in \cite{MaMeNaRa,MaMeRa1,MaMeRa2} hold in a stronger topology.

\begin{remark}
We note from the proof that all initial data in Theorem~\ref{th:1} have a tail on the right of the soliton of the form $c_0 x^{-\theta}$, for $c_0>0$ and $\theta=\frac {5\beta-1}{4\beta}\in (\frac 12,1)$.
Observe that for such value of $\theta$, this tail does not belong to $L^1(\mathbb R)$.

Recall from \cite{MaMeRa3} that $\theta\in (1,\frac 54]$ corresponds
to blowup in infinite time and $\theta \in (\frac 54,\frac {29}{18})$ to exotic blowup in finite time
(for negative values of the multiplicative constant $c_0$). This means that, except the remaining question of the largest value of $\theta$ leading to exotic blowup, the influence of such tails on the soliton is now well-understood.
\end{remark}

\begin{remark}
The more general statement Theorem~\ref{th:2} given in Section~\ref{S:6} provides a large set of initial data,
related to a one-parameter condition to control the scaling instability direction (in particular responsible for blowup
in finite time).
As in the classification given by \cite{MaMeRa1}, a strong topology related to $L^2$ weighted norm is necessary to avoid destroying the tail leading to the soliton flattening. 
Therefore, though the phenomenon of flattening solitons may seem exotic, it is rather robust by perturbation in weighted norms,
its only instability in such spaces being related to the scaling direction.
Moreover, it follows from formal arguments that any small perturbation in that direction should lead to blowup with the blow-up rate $C(T-t)^{-1}$ or to exit of the soliton neighborhood.
This is analogous to the situation described by the construction of the $\mathcal C^1$ threshold manifold in~\cite{MaMeNaRa}. Here, because of
weaker decay estimates on the residue, we do not address the question of the regularity of this set.
\end{remark}

\begin{remark}
Flattening solitary waves were constructed in Theorem 1.5 of \cite{Lan2} for the following double power (gKdV) equations with
saturated nonlinearities
\begin{equation*}
\partial_t u + \partial_x(\partial_x^2 u + u^5 - \gamma |u|^{q-1} u) = 0 \quad \mbox{where $q>5$ and $0<\gamma\ll 1$.}
\end{equation*}
The blow-down rate and the position of the soliton are fixed 
\begin{equation*}
\ell(t)\sim c_1 t^{\frac 2{q+1}},\quad x(t)\sim c_2 t^{\frac {q-3}{q+1}} \quad \mbox{as $t\to +\infty$.}
\end{equation*}
Observe that $q>5$ corresponds to $\frac 2{q+1}\in (0,\frac 13)$, \emph{i.e.} the same range of decay rates as 
in Theorem~\ref{th:1} for equation \eqref{gkdv}.

Analogous results (construction of minimal mass solutions with exotic blow-up rates) were also established for a double power nonlinear Schr\"odinger equation in \cite{LeMaRa}.
\end{remark}

\subsection*{Notation}
For $x\in \mathbb R$, we denote $x_+=\max(0,x)$.

For a given small positive constant $0<\astar \ll 1$, $\delta(\astar)$ will denote a small constant with 
\begin{equation*} \delta(\astar) \to 0 \quad \text{as} \quad \astar \to 0  .\end{equation*}
We will denote by $c$ a positive constant that may change from line to line. The notation $a \lesssim b$ (respectively, $a \gtrsim b$) means that $a \le c b$ (respectively, $a \ge c b$) for some positive constant $c$.

For $1 \le p \le +\infty$, $L^p(\mathbb R)$ denote the classical Lebesgue spaces. 
We define the weighted spaces $L^ 2_\loc=L^2(\mathbb R; e^{-\frac{|y|}{10}}dy)$ and $L^2_B(\mathbb R)=L^2(\mathbb R; e^{\frac{y}{B}}dy)$, for $B \ge 100$ to be fixed later in the proof, through the norms
\begin{equation} \label{def:L2B}
\|f\|_{L^2_\loc}=\left(\int_{\mathbb R} f^2(y) e^{-\frac{|y|}{10}} dy \right)^ {\frac12} \quad \text{and} \quad \|f\|_{L^2_B}=\left( \int_{\mathbb R} f^2(y)e^{\frac{y}B} dy\right)^{\frac12}  .
\end{equation}
It is clear from the definition that $\|f\|_{L^2_\loc} \lesssim \|f\|_{L^ 2_B}$.

For $f$, $g \in L^2(\mathbb R)$ two real-valued functions, we denote the scalar product 
\begin{equation*} (f,g)=\int_{\mathbb R} f(x)g(x) dx  .\end{equation*} 
We introduce the generator of the scaling symmetry 
\begin{equation} \label{def:lambda}
 \Lambda f=\frac12 f +yf'  .
 \end{equation}
We also define the linearized operator $\mathcal{L}$ around the ground state by 
\begin{equation} \label{def:L}
 \mathcal{L}f=-f''+f-5Q^4f  .
 \end{equation}

From now on, for simplicity of notation, we write $\int$ instead of $\int_{\mathbb R}$ and omit~$dx$ in integrals. 

\subsection{Strategy of the proof}
 
The overall strategy of the proof, based on the construction of a suitable ansatz and energy estimates,
follows the one developed in \cite{Ma,MaMeRa1,MaMeRa2,MaMeRa3,Me,RaSz} in similar contexts.
The originality of the present work lies mainly in the
prior preparation of suitable tails and the rigorous justification of all relevant flattening regimes.

\smallskip

\noindent (i) \textit{Definition of the slowly decaying tail.}
Given $c_0 >0$, $x_0 \gg 1$ and $\frac12<\theta<1$, we introduce a smooth function $f_0$ corresponding to a slowly decaying tail on the right:
\begin{equation*} 
f_0(x)=c_0x^{-\theta} \ \text{for} \ x>\frac{x_0}2, \quad f_0(x)=0, \ \text{for} \ x<\frac{x_0}4  .
\end{equation*}
In the present case, a special care has to be taken in the preparatory step of understanding the evolution of such slowly decaying tails under the (gKdV) flow. Not only the decay rate is slower than the one in \cite{MaMeRa3}
 but also the control of the solution 
is needed close to the larger space-time region $x\gtrsim t^{\beta}$, for $\beta>\frac 13$.
Note that the proof uses the mass criticality of the 
exponent (it extends to super-critical exponents). See Section~\ref{S:2}.

\smallskip

\noindent (ii) \textit{Emergence of the flattening regime.}
For $t_0\gg 1$, we consider the rescaled time variable 
\begin{equation} \label{eq:s}
\frac{ds}{dt}=\frac1{\lambda^3} \iff s(t)=s_0+\int_{t_0}^t\frac{d\tau}{\lambda^3(\tau)} \,d\tau  .
\end{equation}
In the variable $s$, the equations governing the parameters $(\lambda,\sigma,b) \in (0,+\infty)\times \mathbb R^2$ write
\begin{equation} \label{eq:la:si:b}
\frac{\lambda_s}{\lambda}+b=0, \quad \sigma_s=\lambda, \quad \frac{d}{ds}\left(\frac{b}{\lambda^2}+\frac{4}{\int Q}c_0 \lambda^{-\frac 32}\sigma^{-\theta}\right)=0  ,
\end{equation}
where the term $c_0 \lambda^{-\frac 32}\sigma^{-\theta}$ comes from the tail. See computations in Lemmas~\ref{lemma:est:rF}-\ref{lemma:gh}.

We integrate these equations following the formal argument in~\cite{MaMeRa3}. First, we observe integrating the last equation in \eqref{eq:la:si:b} that 
\begin{equation} \label{eq:b}
\frac{b}{\lambda^2}+\frac{4}{\int Q}c_0 \lambda^{-\frac 32}\sigma^{-\theta}=l_0  ,
\end{equation}
where $l_0$ is a constant. As in \cite{MaMeRa3}, we focus on the regime $l_0=0$, which corresponds formally to
avoid the instability by scaling.
By combining \eqref{eq:b} with the first two equations in \eqref{eq:la:si:b}, this leads to 
\begin{equation*} 
\lambda^{-\frac12}\lambda_s=\frac{4}{\int Q}c_0\sigma^{-\theta}\sigma_s  ,
\end{equation*}
which yields after integration 
\begin{equation*} 
\lambda^{\frac12} -\frac{2}{\int Q}\frac{c_0}{1-\theta}\sigma^{-\theta+1} =l_1  .
\end{equation*}
Since we expect $\lambda(s)\to +\infty$ as $s\to +\infty$, we can neglect the constant $l_1$, which leads us to
\begin{equation*}
\lambda^{\frac12} =\frac{2}{\int Q}\frac{c_0}{1-\theta}\sigma^{-\theta+1}  .
\end{equation*}
This imposes the conditions $\theta<1$ and $c_0>0$. Now, we use the second equation in~\eqref{eq:la:si:b} to obtain
(using the condition $\frac12<\theta$ which  also ensures that the tail belongs to the space $L^2$) 
\begin{equation*} 
\sigma_s=\lambda=\left(\frac2{\int Q}\frac{c_0}{1-\theta}\right)^2\sigma^{2-2\theta} \implies \sigma^{2\theta-1}(s)
=(2\theta-1)\left(\frac2{\int Q}\frac{c_0}{1-\theta}\right)^2s  ,
\end{equation*}
after integrating over $[s_0,s]$ and choosing $\sigma^{2\theta-1}(s_0)=(2\theta-1) \big(\frac2{\int Q}\frac{c_0}{1-\theta} \big)^2s_0$.
Hence,
\begin{equation*} 
\lambda(s)=(2\theta-1)^{\frac{2(1-\theta)}{2\theta-1}}\left(\frac2{\int Q}\frac{c_0}{1-\theta}\right)^{\frac2{2\theta-1}}s^{\frac{2(1-\theta)}{2\theta-1}}  .
\end{equation*}
By using the first equation \eqref{eq:la:si:b}, we also compute 
\begin{equation*} 
b(s)=-\frac{2(1-\theta)}{2\theta-1}s^{-1}  .
\end{equation*}
To simplify constants, we choose 
\begin{equation} \label{def:c0}
c_0=\frac{\int Q}{2}(1-\theta)(2\theta-1)^{-(1-\theta)}>0  ,
\end{equation}
so that 
\begin{equation} \label{resol:la:si:b}
\lambda(s)=s^{\frac{2(1-\theta)}{2\theta-1}}, \quad \sigma(s)=(2\theta-1)s^{\frac1{2\theta-1}} \quad \text{and} \quad b(s)=-\frac{2(1-\theta)}{2\theta-1}s^{-1}  .
\end{equation}
To come back to the original time variable, we first need to solve \eqref{eq:s}. We set 
\begin{equation*}
\beta=\frac{1}{5-4\theta}\in\left(\frac 13,1\right) \iff
\theta=\frac{5\beta-1}{4\beta}  .
\end{equation*}
Then, 
by choosing 
\begin{equation*} 
t_0=\frac{2\theta-1}{5-4\theta}s_0^{\frac{5-4\theta}{2\theta-1}} \quad \text{and} \quad c_s=\left(\frac2{3\beta-1}\right)^{\frac{3\beta-1}{2}}  ,
\end{equation*}
we obtain
\begin{equation*}
t=\frac{2\theta-1}{5-4\theta}s^{\frac{5-4\theta}{2\theta-1}} \iff t= \frac{3\beta-1}{2}s^{\frac2{3\beta-1}} \iff s=c_st^{\frac{3\beta-1}{2}}.
\end{equation*}
Last, we deduce from \eqref{resol:la:si:b} that 
\begin{equation} \label{resol:t:la:si:b}
\lambda(t)=c_{\lambda}t^{\frac{1-\beta}{2}} \quad \text{and} \quad \sigma(t)=c_{\sigma}t^{\beta}  ,
\end{equation}
for some positive constants $c_{\lambda}$ and $c_{\sigma}$
(see \eqref{def:cts}).

\smallskip

\noindent (iii) \emph{Energy estimates.} 
In order to construct an exact solution of \eqref{gkdv} satisfying the formal regime~\eqref{resol:t:la:si:b}, we use a variant
of the mixed energy-virial functional first introduced for (gKdV) in \cite{MaMeRa1} (the introduction of the 
virial argument in the neighborhood of the soliton for critical (gKdV) goes back to \cite{MaMejmpa}).
Considering a defocusing regime induces a simplification (see also the energy estimates in~\cite{CoMa1}) that allows us to treat the whole range
$\beta\in (\frac 13,1)$ in spite of a basic ansatz and relatively large error terms. See Section~\ref{section:energy}.

\section{Persistence properties of slowly decaying tails on the right}\label{S:2}

In this section, we present a general result concerning the persistence of a class of slowly decaying tails for the critical gKdV equation in a suitable space-time region. 

Let $\theta\in (\frac 12,1]$ and define
\begin{equation}\label{def:gar}
\beta=\frac{1}{5-4\theta}\in\left(\frac 13,1\right],\quad
\theta=\frac{5\beta-1}{4\beta},\quad
\nu=\frac{1-\beta}2\in\left[0,\frac13\right).
\end{equation}
For $c_0>0$ and $x_0 \gg 1$, we consider $f_0$ any smooth nonnegative function such that 
\begin{equation} \label{def:f_0}
f_0(x)= \begin{cases} c_0 x^{-\theta} & \text{for} \ x>\frac{x_0}2  \\ 0 & \text{for} \ x<\frac{x_0}4 \end{cases} 
\quad \text{and} \quad \Big|f_0^{(k)}(x)\Big| \lesssim c_0|x|^{-\theta-k}, \ \forall \, k \in \mathbb N, \ \forall \, x \in \mathbb R  .
\end{equation}
Note that 
\begin{equation} \label{est:f_0}
 \|f_0\|_{L^2} \sim c_0\left(\int_{x>x_0/4}x^{-2\theta} \, dx \right)^{\frac12} \sim c_0 x_0^{-(\theta-\frac12 )}=\delta(x_0^{-1})  .
 \end{equation}
Now,  for   $t_0 \gg 1$ to be fixed, let $f$ be a solution of the IVP 
\begin{equation} \label{eq:q_0}
\left\{ \begin{aligned} & \partial_t f+\partial_x\big(\partial_x^2 f+f^5\big)=0, \\ &f(t_0,x)=f_0(x)  .\end{aligned}\right. 
\end{equation}

The main result of this section states that the special asymptotic behavior of $f_0(x)$ on the right persists for $f(t,x)$ in regions of the form $x\gtrsim t^{\beta}$.
 
\begin{proposition} \label{decay:q_0} 
Let $\theta \in \big(\frac12,1\big]$, $\beta=\frac{1}{5-4\theta}$ and $c_0>0$.
For $x_0>0$ large enough, for any $\kappa_0>0$, setting $t_0:= ({x_0}/{\kappa_0} )^{1/\beta}$,
the solution $f$ of \eqref{eq:q_0} is global, smooth and bounded in $H^1$.
Moreover, it holds
for all $t \ge t_0$ and $x>\kappa_0 t^{\beta}\ge x_0$,
 \begin{align} \label{decay:q_0.1}
\forall \, k \in \mathbb N, \quad & \big|\partial_x^k f(t,x)-f_0^{(k)}(x)\big| \lesssim |x|^{-(5\theta-2+k)}  ,
\\ \label{decay:q_0.2}
& \big|\partial_t f(t,x)\big|\lesssim |x|^{-(\theta+3)}.
\end{align}
 \end{proposition}
The rest of this section is devoted to the proof of Proposition~\ref{decay:q_0}, which
requires preparatory monotonicity lemmas based on variants of the so-called Kato identity (see \cite{Kato,MaMejmpa,MaMe}). 
This result is a substantial generalization of Lemma 2.3 in \cite{MaMeRa3}, where only the case $\theta=1$ is treated.
Our proof allows regions $x \gtrsim t^{\beta}$ for any $\beta>\frac 13$.
Complementary results are obtained in \cite{MP}, where large regions close to $x=0$ are investigated
by similar functionals.

\begin{remark} \label{rem:decay}
Without loss of generality and for simplicity of notation, we reduce ourselves to prove 
estimates \eqref{decay:q_0.1} and \eqref{decay:q_0.2} for the special value $\kappa_0=2$. Indeed, consider the function $\widetilde{f}(s,y)=\lambda^{\frac12}f(\lambda^3 s, \lambda y)$. Then $\widetilde{f}$ is a solution to \eqref{eq:q_0} where $\widetilde{f}_0=\widetilde f(0)$ satisfies \eqref{def:f_0} with $\widetilde{c}_0=\lambda^{\frac12-\theta}c_0$ instead of $c_0$. Moreover, the condition $x>2t^{\beta}$ rewrites $y>2\lambda^{3\beta-1}s^{\beta}>\kappa_0 s^{\beta}$ by choosing $\lambda=(2\kappa_0)^{-\frac1{3\beta-1}}$ (recall that $\beta>\frac13$).
 \end{remark}
 
First, note that if $x_0$ is chosen large enough, it follows directly from the Cauchy theory developed in \cite{KPV} (see Corollary 2.9) and \eqref{est:f_0} that $f \in C(\mathbb R : H^s(\mathbb R))$ for all $s \ge 0$ and 
\begin{equation} \label{bound:q_0}
\sup_{t \in \mathbb R}\|f(t)\|_{H^s} \lesssim \delta(x_0^{-1})  .
\end{equation}
Moreover, by using the sharp Gagliardo-Nirenberg inequality \eqref{sharpGN} and the conservations of the mass and the energy \eqref{mass_energy}, we deduce, for $x_0$ large enough, that
\begin{equation} \label{bound:df:dx} 
\sup_{t \in \mathbb R} \| \partial_xf(t)\|_{L^2} \lesssim |E(f_0)| \lesssim x_0^{-(\theta+\frac12)}  .
\end{equation}

 Define $q(t,x):=f(t,x)-f_0(x)$. Then, it follows from \eqref{eq:q_0} that 
 \begin{equation} \label{eq:q}
 \left\{ \begin{aligned} &\partial_tq+\partial_x\big(\partial_x^2q+(q+f_0)^5-f_0^5\big)=F_0, \\ &q(t_0)=0 , \end{aligned}\right. 
 \end{equation}
 where 
 \begin{equation*} F_0:=-\partial_x^3f_0-\partial_x(f_0^5)  . \end{equation*}
For any $\bar{r} \ge 0$, we define a smooth function $\omega_{\bar{r}}$ such that
\begin{equation} \label{def:omega}
\omega_{\bar{r}}(x)= x^{\bar{r}} \ \text{for} \ x \ge 2, \quad \omega_{\bar{r}}(x)=e^{\frac{x}8} \ \text{for} \ x \le 0, \quad \omega_{\bar{r}}' > 0 \ \text{on} \ \mathbb R.
\end{equation}
Observe that
\begin{equation}\label{def:omegabis}
|\omega_{\bar{r}}'' |+|\omega_{\bar{r}}'''|\le C\omega_{\bar{r}}' \ \text{on} \ \mathbb R,
\end{equation}
for some constant $C=C(\omega_{\bar r})>0$.

\begin{lemma} \label{decay_lemma.1}
Let $0<r<2\theta+4$, $r\neq 5$ and $0<\epsilon<\frac{3\beta-1}{20}|r-5|$.
Define 
\begin{equation*}
M_r(t):= \int q^2(t,x) \omega_r(\bar{x}) \, dx   \quad \text{where} \quad \bar{x}=\frac{x-t^{\beta}}{t^{\nu+\epsilon}}  .
\end{equation*}
Then, for $x_0>1$ large enough, and any $t \ge t_0=\left(\frac{x_0}2\right)^{\frac1{\beta}}$,
\begin{multline}\label{decay_claim.1} 
M_{r}(t)+\int_{t_0}^t\int\left[s^{-3\nu-\epsilon} q^2
+s^{-1} \bar x_+ q^2
+s^{-\nu-\epsilon}(\partial_xq)^2\right]\omega_{r}'(\bar{x})\, dx ds \\
\lesssim 
\begin{cases}
t^{\frac{3\beta-1}2 (r-5)-r\epsilon} & \mbox{if $r>5$}\\
t_0^{-\frac{3\beta-1}2 (5-r)-r\epsilon} & \mbox{if $r<5$.}
\end{cases} 
\end{multline}
\end{lemma}

\begin{proof}
To prove \eqref{decay_claim.1}, we differentiate $M_{r}$ with respect to time, use \eqref{eq:q} and integrate by parts in the $x$ variable to obtain 
\begin{align*}
M'_r &= -3t^{-\nu-\epsilon}\int (\partial_xq)^2\omega_r'(\bar{x})+t^{-3\nu-3\epsilon}\int q^2\omega_r'''(\bar{x})- \beta t^{-3\nu-\epsilon}\int q^2\omega_{{r}}'(\bar{x})\\ 
&\quad-(\nu+\epsilon) t^{-1} \int q^2 \, \bar{x}\omega_{{r}}'(\bar{x}) -2t^{-\nu-\epsilon} \int \left(\frac{(q+f_0)^6}{6}-(q+f_0)^5q-\frac{f_0^6}{6}\right)\omega_{{r}}'(\bar{x})
\\ &
\quad -2\int \left( (q+f_0)^5-5f_0^4q-f_0^5\right)f_0'\omega_{{r}}(\bar{x}) +2\int q F_0 \omega_{{r}}(\bar{x}) \\& =: M_1 +M_2+M_3+M_4+M_5+M_6+M_7 .
\end{align*}
By using \eqref{def:omegabis}, for $t_0$ large enough, we have
\begin{equation} \label{decay_claim1.1}
|M_2|=t^{-3\nu-3\epsilon}\left| \int q^2\omega_{{r}}'''(\bar{x})\right| \leq ct_0^{-2\epsilon} t^{-3\nu-\epsilon} \int q^2\omega_{{r}}'(\bar{x})\leq -\frac 12 M_3 ,
\end{equation}
and so
\begin{equation*}
M_1+M_2+M_3\leq
-3t^{-\nu-\epsilon}\int (\partial_xq)^2\omega_r'(\bar{x})
- \frac{\beta}2 t^{-3\nu-\epsilon}\int q^2\omega_{{r}}'(\bar{x}).
\end{equation*}

Next, we estimate $M_j$ for $j=4,\cdots,7$ separately. 
For future use, observe that by the assumption $0<\epsilon<\frac{3\beta-1}{20}|r-5|$ and $0<r<6$, we also have
$0<\epsilon<\frac{3\beta-1}{4}$.

We denote
\begin{equation}\label{notation}
M_j=\int_{\bar x<-t^{1-3\nu-2\epsilon}}+
\int_{-t^{1-3\nu-2\epsilon}<\bar x<0}+\int_{\bar x>0}=:M_j^-+M_j^0+M_j^+.
\end{equation}

\smallskip

\noindent \textit{Estimate for $M_4$.} It is clear that $M_4^+(t) \le 0$. Next, 
for $t_0$ large enough,
\begin{equation*} 
M_4^0 \leq(\nu+\epsilon) t^{-3\nu-2\epsilon}\int q^2 \omega_{{r}}'(\bar{x})
\leq (\nu+\epsilon) t_0^{-\epsilon} t^{-3\nu-\epsilon}\int q^2 \omega_{{r}}'(\bar{x})
\leq -\frac14M_3.
\end{equation*}
Then, it follows from the definition of $\omega_{{r}}$ in \eqref{def:omega} and \eqref{bound:q_0} that, 
for $t_0$ large enough, 
\begin{equation*} 
M_4^-
\lesssim t^{-1}\int_{\bar x<-t^{1-3\nu-2\epsilon}}q^2|\bar x|e^{\frac{\bar x}{8}} 
\lesssim t^{-1} e^{-\frac1{16}(t^{1-3\nu-2\epsilon)}}\int q^2 
\lesssim t^{-10} ,
\end{equation*}
since $0<\epsilon<\frac{3\beta-1}4$.
 
\smallskip

\noindent \textit{Estimate for $M_5$.} 
Using
\begin{equation*}
\left| \frac{(q+f_0)^6}{6}-(q+f_0)^5q-\frac{f_0^6}{6}\right| \lesssim q^2 f_0^4+q^6  ,
\end{equation*}
it holds 
\begin{equation*} 
\big| M_5 \big| \le c t^{-\nu-\epsilon}\int q^2 f_0^4 \omega_{{r}}'(\bar{x})
+ct^{-\nu-\epsilon}\int q^6 \omega_{{r}}'(\bar{x})
=:M_{5,1}+M_{5,2}  .
\end{equation*}
We observe that, for $t_0$ large enough,
\begin{equation}\label{zone}
-t^{1-3\nu-2\epsilon}=-t^{\beta-\nu-2\epsilon}<\bar x \implies t^{\beta}-t^{\beta-2\epsilon}<x
\implies \frac 12 t^\beta <x .
\end{equation}
Thus, we deduce from \eqref{def:f_0}, and then
$4\theta \beta > 2\beta>2\nu$ (since $\theta>\frac12$ and $\beta > \frac13>\nu$),
 that, for $t_0$ large enough,
\begin{align}
M_{5,1}^0+M_{5,1}^+
&\leq c t^{-\nu-\epsilon} \int_{x>\frac 12 t^\beta} q^2 x^{-4\theta} \omega_{{r}}'(\bar{x})
\leq c t^{-(\nu+4\theta \beta+\epsilon)} \int q^2 \omega_{{r}}'(\bar{x}) \nonumber
\\ &\leq c t_0^{-(\nu+4\theta \beta)+3\nu} t^{-3\nu-\epsilon} \int q^2 \omega_{{r}}'(\bar{x})
\leq -\frac 18 M_3. \label{M51}
\end{align}
As before for $M_4^-$, we have for $t_0$ large enough,
\begin{equation} \label{M51bis}
M_{5,1}^- \lesssim t^{-10}  .
\end{equation}

To deal with $M_{5,2}$, we follow an argument in Lemma 6 of \cite{Mjams}. We have by using the fundamental theorem of calculus 
\begin{equation*} 
-q^2(x,t)\sqrt{\omega_{{r}}'(\bar{x})}=2\int_x^{+\infty}q\partial_xq\sqrt{\omega_{{r}}'(\bar{x})}+\frac12t^{-\nu-\epsilon}\int_x^{+\infty}q^2 \frac{\omega_{{r}}''(\bar{x})}{\sqrt{\omega_{{r}}'(\bar{x})}}  ,
\end{equation*}
and so, by Cauchy Schwarz inequality and then \eqref{def:omegabis},
\begin{align}
\left\|q^2(x,t)\sqrt{\omega_{{r}}'(\bar{x})}\right\|_{L^\infty}^2
&\lesssim \|q\|_{L^2}^2\int (\partial_xq)^2 \omega_{{r}}'(\bar{x})
 +t^{-2\nu-2\epsilon}\|q\|_{L^2}^2\int q^2\frac{\big(\omega_{{r}}''(\bar{x})\big)^2}{\omega_{{r}}'(\bar{x})}\nonumber\\
&\lesssim \|q\|_{L^2}^2\left[\int (\partial_xq)^2 \omega_{{r}}'(\bar{x})
 +t^{-2\nu-2\epsilon}\int q^2 \omega_{{r}}'(\bar{x})\right].
\label{est:infty}
\end{align}
Therefore, using also \eqref{bound:q_0}, for $x_0$ large enough,
\begin{align}
M_{5,2} &\lesssim t^{-\nu-\epsilon} \|q\|_{L^2}^2 \left\| q^2\sqrt{\omega_{{r}}'(\bar{x})}\right\|_{L^\infty}^2 \label{M52} \\
& \le \delta(x_0^{-1})\left[ t^{-\nu-\epsilon}\int (\partial_x q)^2 \omega_{{r}}'(\bar{x})+ t^{-3\nu-\epsilon} \int q^2 \omega_{{r}}'(\bar{x})\right]
\leq -\frac 12 M_1-\frac 1{16}M_3.\nonumber
\end{align}

\smallskip
 
\noindent \textit{Estimate for $M_6$.} By using interpolation, \eqref{def:f_0} and then the inequality $|x|^{-\theta}q^5 \lesssim x^{-4\theta}q^2+q^6$, we observe that
\begin{equation*} 
\left| f_0'\left((q+f_0)^5-5f_0^4q-f_0^5\right)\right| \lesssim |f_0'||f_0|^3q^2+|f_0'||q|^5
\lesssim |x|^{-1} f_0^4q^2+|x|^{-1}q^6 .
\end{equation*}
It follows that
\begin{equation*} 
\big| M_6 \big| \le \int_{x\geq \frac 14 x_0}q^2 f_0^4 x^{-1} \omega_r(\bar{x})
+\int_{x\geq \frac 14 x_0} x^{-1}q^6 \omega_r(\bar{x})=:M_{6,1}+M_{6,2}  .
\end{equation*}
By \eqref{def:omega} and \eqref{zone}, and choosing $\epsilon>0$ such that 
$0<\epsilon<\beta-\nu=\frac{3\beta-1}2$,
\begin{equation} \label{omega:est}
\omega_{{r}}(\bar{x}) x^{-1}\lesssim
\begin{cases}
\omega'_{{r}}(\bar{x}) {\bar x}{x^{-1}}
\lesssim t^{-\nu-\epsilon}\omega_{{r}}'(\bar{x}) &\mbox{for $\bar x>2$,}
\\
t^{-\beta}\omega_{{r}}'(\bar{x})\lesssim t^{-\nu-\epsilon}\omega_{{r}}'(\bar{x}) & 
\mbox{for $-t^{1-3\nu-2\epsilon}<\bar x<2$,}
\end{cases}
\end{equation}
Thus, for $t_0$ and $x_0$ large enough,
\begin{equation*} 
M_{6,1}^{0,+}+M_{6,2}^{0,+} \leq c( M_{5,1}^{0,+}+M_{5,2}^{0,+})
\leq c (\delta(x_0^{-1})+\delta(t_0^{-1})) (M_1+M_3)
\leq -\frac 14 M_1-\frac 1{32}M_3  .
\end{equation*}

Last, $M_{6,1}^- +M_{6,2}^- \lesssim t^{-10}$ is proved as for $M_4^-$.

\smallskip
 
\noindent \textit{Estimate for $M_7$.} We get from the Cauchy-Schwarz inequality that 
\begin{equation*} 
\big| M_7 \big| \leq 2 \left( \int q^2 \omega_{{r}}'(\bar{x})\right)^{\frac12}\left(\int F_0^2 \frac{\omega_{{r}}^ 2(\bar{x})}{\omega_{{r}}'(\bar{x})} \right)^{\frac12} 
\le -\frac 1{64} M_3 + cM_8\quad \mbox{where}\quad M_8=t^{3\nu+\epsilon}\int F_0^2 \frac{\omega_{{r}}^ 2(\bar{x})}{\omega_{{r}}'(\bar{x})}. 
\end{equation*}
First, we see from \eqref{def:f_0} that 
for $x>\frac 14 x_0$, 
$|F_0| \lesssim |x|^{-\theta-3}$ ($\theta>\frac12$),
and for $x<\frac 14 x_0$, $F_0=0$. 

For $\bar{x} \ge 2$, it holds $ \frac{\omega_{{r}}^ 2(\bar{x})}{\omega_{{r}}'(\bar{x})} =r^{-1} |\bar{x}|^{{r}+1}
 \lesssim t^{-(\nu+\epsilon)({r}+1)}|x|^{{r}+1}$. Hence, 
\begin{align*}
t^{3\nu+\epsilon}\int_{\bar x>2} F_0^2 \frac{\omega_{{r}}^ 2(\bar{x})}{\omega_{{r}}'(\bar{x})}
&\lesssim t^{2\nu-r(\nu+\epsilon)}\int_{\bar x>2} |x|^{-2(\theta+3)}| {x}|^{{r}+1} 
\\ & \lesssim t^{2\nu-r(\nu+\epsilon)}\int_{x>t^\beta} |x|^{-2\theta+{r}-5}
\\ & \lesssim t^{2\nu-r(\nu+\epsilon)}t^{-\beta(2\theta-{r}+4)}
=t^{-1+\frac{3\beta-1}2 (r-5)-r\epsilon} ,
\end{align*}
since $2\theta-r+4>0$ by assumption, and
\begin{align}
1+2\nu-r(\nu+\epsilon) -\beta(2\theta-{r}+4)
&=r(\beta-\nu)+1+2\nu-2\beta\theta-4\beta-r\epsilon\nonumber\\
&=\frac{3\beta-1}2 (r-5)-r\epsilon.\label{pourtt}
\end{align}

For $-t^{-1-3\nu-2\epsilon}< \bar x<2$, it holds $ \frac{\omega_{{r}}^ 2(\bar{x})}{\omega_{{r}}'(\bar{x})} \lesssim 1$
and $x\geq \frac 12 t^\beta$ (from \eqref{zone}) so that 
\begin{align*} 
t^{3\nu+\epsilon}\int_{-t^{-1-3\nu-2\epsilon}<\bar x<2} F_0^2 \frac{\omega_{{r}}^ 2(\bar{x})}{\omega_{{r}}'(\bar{x})} &\lesssim t^{3\nu+\epsilon} \int_{x>\frac 12t^{\beta}}x^{-2(\theta+3)} 
\\
&\lesssim t^{3\nu+\epsilon}t^{-\beta(2\theta+5)}=t^{-9\beta+2+\epsilon}  .
\end{align*}
Last, for $\bar x<-t^{-1-3\nu-2\epsilon}$, then $\frac{\omega_{{r}}^ 2(\bar{x})}{\omega_{{r}}'(\bar{x})} =8e^{\frac{\bar{x}}8}$ so that as for $M_4^-$,
\begin{equation*}
t^{3\nu+\epsilon}\int_{\bar x<-t^{-1-3\nu-2\epsilon}} F_0^2 \frac{\omega_{{r}}^ 2(\bar{x})}{\omega_{{r}}'(\bar{x})} 
 \lesssim t^{-10}  .
\end{equation*}

Gathering all those estimates, we obtain in conclusion that, for some $c>0$, 
\begin{equation*}
M_r'+c\int\left[t^{-3\nu-\epsilon} q^2
+t^{-1} \bar x_+ q^2
+t^{-\nu-\epsilon}(\partial_xq)^2\right]\omega_{r}'(\bar{x})\, dx 
\lesssim 
t^{-1+\frac{3\beta-1}2 (r-5)-r\epsilon}.
\end{equation*}
Observe that by the assumption $0<\epsilon<\frac{3\beta-1}{20} |r-5|$,
\begin{equation} \label{assump:epsilon}
r\epsilon<6\epsilon<\frac{3}{10}(3\beta-1)|r-5|<\frac{3\beta-1}2|r-5|.
\end{equation} 
Thus, integrating this estimate on $[t_0,t]$, we obtain \eqref{decay_claim.1}.
\end{proof}

We prove a similar estimate for a quantity related to the energy.
\begin{lemma}
Let $0<r<2\theta+4$, $r\neq 5$ and $0<\epsilon<\frac{3\beta-1}{20} |r-5| .$
Define
\begin{equation*}
E_{{r}}(t):=t^{2\nu+2\epsilon}\int \left[ (\partial_xq)^2-\frac13\big((q+f_0)^6-f_0^6-6qf_0^5\big) \right](t,x) \omega_{{r}+2}(\bar{x}) \, dx . 
\end{equation*}
where $\bar{x}=\frac{x-t^{\beta}}{t^{\nu+\epsilon}}$. 
Then, for $x_0>1$ large enough, and any $t \ge t_0=\left(\frac{x_0}2\right)^{\frac1{\beta}}$,
\begin{multline} \label{decay_claim.2} 
E_{{r}}(t)+\int_{t_0}^t\int\left[s^{-\nu+\epsilon} (\partial_xq)^2 
+s^{2\nu+2\epsilon-1}   \bar x_+(\partial_x q)^2 
+s^{\nu+\epsilon} (\partial_x^2q)^2\right]\omega_{r+2}'(\bar{x})\, dxds  \\  \lesssim 
\begin{cases}
t^{\frac{3\beta-1}2 (r-5)-r\epsilon} & \mbox{if $r>5$}\\
t_0^{-\frac{3\beta-1}2 (5-r)-r\epsilon} & \mbox{if $r<5$.}
\end{cases}
\end{multline}
\end{lemma}
 
\begin{proof}
We differentiate $E_r$ with respect to time and integrate by parts to obtain
\begin{align*}
E_r'&= -t^{\nu+\epsilon}\int \left[\partial_x^2 q+(q+f_0)^5-f_0^5\right]^2 \omega_{r+2}'(\bar x)\\
&\quad -2t^{\nu+\epsilon}\int (\partial_x^2 q)^2 \omega_{r+2}'(\bar x)
+t^{-\nu-\epsilon}\int(\partial_x q)^2 \omega_{r+2}'''(\bar x)\\
&\quad -\beta t^{-\nu+\epsilon}
 \int \ (\partial_xq)^2 \omega_{r+2}'(\bar{x})
-(\nu+\epsilon) t^{2\nu+2\epsilon-1} 
\int (\partial_xq)^2 \bar x \omega_{r+2}'(\bar{x})\\
&\quad +\frac\beta3 t^{-\nu+\epsilon}\int \left[(q+f_0)^6-f_0^6-6qf_0^5\right]\omega_{r+2}'(\bar{x})\\
&\quad +\frac13 (\nu+\epsilon) t^{2\nu+2\epsilon-1} 
\int \left[(q+f_0)^6-f_0^6-6qf_0^5 \right] \bar x \omega_{r+2}'(\bar{x})\\
&\quad +2t^{\nu+\epsilon}\int\left[(q+f_0)^5-f_0^5\right]_x (\partial_x q) \omega_{r+2}'(\bar{x})\\
&\quad +2t^{2\nu+2\epsilon}\int (\partial_x q ) F_0' \omega_{r+2}(\bar x) -2t^{2\nu+2\epsilon}\int\left[(q+f_0)^5-f_0^5\right] F_0 \omega_{r+2}(\bar{x})\\
&\quad +2(\nu+\epsilon)t^{2\nu+2\epsilon-1}\int (\partial_xq)^2\omega_{r+2}(\bar x)\\
&\quad -\frac23(\nu+\epsilon)t^{2\nu+2\epsilon-1}\int \left[(q+f_0)^6-f_0^6-6qf_0^5 \right] \omega_{r+2}(\bar x)
\\& =: E_1 +E_2+E_3+E_4+E_5+E_6+E_7+E_8+E_9 +E_{10}+E_{11}+E_{12} .
\end{align*}
First, observe that $E_1\leq 0$, $E_2\leq 0$ and $E_4\leq 0$.
As in the proof of \eqref{decay_claim1.1}, we have for $t_0$ large, $|E_3|\leq -\frac 12 E_4$.

Next, we use the same notation as in \eqref{notation} for $E_j$, $j=5,\cdots,10$. We observe that $E_5^+\leq 0$.
Moreover, as for the estimate of $M_4$ in the proof of \eqref{decay_claim.1},
it holds $E_5^0+E_5^-\leq -\frac 14 E_4+Ct^{-10}$.

\smallskip

\noindent \emph{Estimate for $E_6$}. First, we note
\begin{equation*}
|E_6|\leq c t^{-\nu+\epsilon}\int q^2 f_0^4 \omega_{r+2}'(\bar x)
+ct^{-\nu+\epsilon}\int q^6 \omega_{r+2}'(\bar x)=:E_{6,1}+E_{6,2}.
\end{equation*}
We estimate $E_{6,1}^-\lesssim t^{-10}$ and
\begin{align*}
E_{6,1}^+
&\lesssim t^{-\nu+\epsilon} \int_{\bar x>0} q^2 x^{-4\theta} \omega_{r+2}'(\bar{x})
\lesssim t^{-\nu+\epsilon} \int_{\bar x>0} q^2 x^{-4\theta}(1+\bar x^2) \omega_{r}'(\bar{x})\\
&
\lesssim t^{-3\nu-\epsilon } \int_{x> t^\beta} q^2 x^{-4\theta+2 }\omega_{{r}}'(\bar{x})
\lesssim t^{-3\nu-(4\theta-2)\beta-\epsilon } \int q^2 \omega_{{r}}'(\bar{x})
\lesssim t^{-3\nu-\epsilon } \int q^2 \omega_{{r}}'(\bar{x});
\end{align*}
moreover, since for $\bar x<0$, $\omega_{r+2}'(\bar x)= \omega_{r}'(\bar x)$, we deduce from \eqref{zone} that
\begin{equation*}
E_{6,1}^0\lesssim t^{-\nu+\epsilon}\int_{x>\frac12t^{\beta}} x^{-4\theta}q^2\omega_{r}'(\bar x)
\lesssim t^{-\nu+\epsilon-4\beta\theta}\int q^2 \omega_{{r}}'(\bar{x})\lesssim t^{-3\nu-\epsilon } \int q^2 \omega_{{r}}'(\bar{x}),
\end{equation*}
since $\epsilon<\frac{3\beta-1}2$.
Arguing as for $M_{5,2}$ in the proof of \eqref{decay_lemma.1},
\begin{equation*}
E_{6,2} \lesssim t^{-\nu+\epsilon}\|q\|_{L^2}^4 \int (\partial_xq)^2 \omega_{r+2}'(\bar{x})
 +t^{-3\nu-\epsilon}\|q\|_{L^2}^4 \int q^2\frac{\big(\omega_{r+2}''(\bar{x})\big)^2}{\omega_{r+2}'(\bar{x})} ;
\end{equation*}
and thus, as before, for $x_0$ large,
\begin{equation*}
E_{6,2} \leq-\frac 14 E_4
+t^{-3\nu-\epsilon} \int q^2 \omega_{{r}}'(\bar{x})  .
\end{equation*}

\smallskip

\noindent \emph{Estimate for $E_7$}.
We estimate $E_7^-\leq t^{-10}$ and 
\begin{equation*}
|E_7^0|\lesssim 
t^{-\nu} \int_{-t^{1-3\nu-2\epsilon}<\bar x<0} (q^2 f_0^4+q^6) \omega_{r+2}'(\bar{x})
\lesssim t^{-\varepsilon} (E_{6,1}+E_{6,2}).
\end{equation*}
Now, we estimate
\begin{equation*}E_7^+\lesssim t^{-1+2\nu+2\epsilon} \int_{\bar x>0} q^2 f_0^4\bar x \omega_{r+2}'(\bar{x}) +t^{-1+2\nu+2\epsilon} \int_{\bar x>0} q^6 \bar x \omega_{r+2}'(\bar{x}) 
=:E_{7,1}^++E_{7,2}^+.
\end{equation*}
As before, we have
\begin{align*}
E_{7,1}^+
&\lesssim t^{-1+2\nu+2\epsilon} \int_{\bar x>0} q^2 x^{-4\theta} \bar x\omega_{r+2}'(\bar{x})
\lesssim t^{-1+2\nu+2\epsilon} \int_{\bar x>0} q^2 x^{-4\theta}(1+\bar x^2)\bar x \omega_{r}'(\bar{x})\\
&
\lesssim t^{-1} \int_{x> t^\beta} q^2 x^{-4\theta+2}\bar x_+\omega_{{r}}'(\bar{x})
\lesssim t^{-1} \int q^2 \bar x_+\omega_{{r}}'(\bar{x}).
\end{align*}
Last, by arguing similarly as in \eqref{M52},
\begin{align*}
E_{7,2}^+
&\lesssim t^{-1+2\nu+2\epsilon}\|q\|_{L^2}^2\left\|q^2\sqrt{\omega_{r+2}(\bar x)}\right\|_{L^\infty(\bar x>0)}^2\\
&\lesssim t^{-1+2\nu+2\epsilon}\|q\|_{L^2}^4\left(\int_{\bar x>0}(\partial_x q)^2\omega_{r+2}(\bar x)
+t^{-2\nu-2\epsilon}\int_{\bar x>0}q^2\frac{(\omega_{r+2}')^2}{\omega_{r+2}}\right)\\
&\lesssim
t^{-1+2\nu+2\epsilon}\delta(x_0^{-1})\left(\int_{\bar x>0}(\partial_x q)^2(1+\bar x)\omega_{r+2}'(\bar x)
+t^{-2\nu-2\epsilon}\int_{\bar x>0}q^2(1+\bar x)\omega_r'(\bar x)\right)\\
&\lesssim \delta(x_0^{-1}) (E_4+E_5^+) +
t^{-1}\delta(x_0^{-1})\left( \int_{\bar x>0}q^2 \omega_r'(\bar x)+ \int_{\bar x>0}q^2\bar x\omega_r'(\bar x)\right).
\end{align*}

\smallskip

\noindent \emph{Estimate for $E_8$}.
We compute
\begin{align*}
E_8&=
10t^{\nu+\epsilon}\int (q+f_0)^4(\partial_x q)^2 \omega_{r+2}'(\bar{x})
+10t^{\nu+\epsilon}\int \left[(q+f_0)^4-f_0^4\right](\partial_x q) f_0' \omega_{r+2}'(\bar{x}) \\ 
& \lesssim t^{\nu+\epsilon} \int f_0^4(\partial_x q)^2 \omega_{r+2}'(\bar{x})+ t^{\nu+\epsilon} \int q^4 (\partial_x q)^2\omega_{r+2}'(\bar{x}) \\ & \quad +t^{\nu+\epsilon} \int |q||\partial_xq||f_0|^3|f_0'|\omega_{r+2}'(\bar{x})
+t^{\nu+\epsilon} \int |q|^4|\partial_xq||f_0'|\omega_{r+2}'(\bar{x}) \\ 
& =: E_{8,1}+E_{8,2}+E_{8,3}+E_{8,4}  .
\end{align*}

First, we have as before $E_{8,1}^- \lesssim t^{-10}$. Now, arguing as for $M_{5,1}$, we get for $t_0$ large enough that 
\begin{equation*}
E_{8,1}^0+E_{8,1}^+ \le ct^{\nu+\epsilon} \int_{x>\frac12t^{\beta}}(\partial_xq)^2x^{-4\theta} \omega_{r+2}'(\bar{x}) \le ct^{-(-\nu+4\theta\beta-\epsilon)}\int(\partial_xq)^2 \omega_{r+2}'(\bar{x}) \le -\frac18E_4  .
\end{equation*}

To handle $E_{8,2}$, we use a similar argument as in \eqref{M52}. Observe from the fundamental theorem of calculus that 
\begin{align*} 
-[q\partial_xq](x,t)\sqrt{\omega_{{r+2}}'(\bar{x})}&=\int_x^{+\infty}q\partial_x^2q\sqrt{\omega_{{r+2}}'(\bar{x})}+\int_x^{+\infty}(\partial_xq)^2\sqrt{\omega_{{r+2}}'(\bar{x})}
\\ &\quad +\frac12t^{-\nu-\epsilon}\int_x^{+\infty}q\partial_xq \frac{\omega_{{r+2}}''(\bar{x})}{\sqrt{\omega_{{r+2}}'(\bar{x})}}  .
\end{align*}
Thus, by the Cauchy-Schwarz inequality,
\begin{align*}
\left\| q\partial_xq \sqrt{\omega_{{r+2}}'(\bar{x})}\right\|_{L^{\infty}} &\lesssim \|q\|_{L^2}\left(\int (\partial_x^2q)^2 \omega_{{r+2}}'(\bar{x})\right)^{\frac12} +\int (\partial_xq)^2\sqrt{\omega_{{r+2}}'(\bar{x})}\nonumber \\ &\quad +t^{-\nu-\epsilon}\|q\|_{L^2}\left(\int (\partial_xq)^2 \frac{\big(\omega_{{r+2}}''(\bar{x})\big)^2}{\omega_{{r+2}}'(\bar{x})}\right)^{\frac12} .
\end{align*}
To deal with the second term on the right-hand side of the above inequality, we integrate by parts to get
\begin{equation*}
\int (\partial_xq)^2\sqrt{\omega_{{r+2}}'(\bar{x})}=-\int q\partial_x^2q \sqrt{\omega_{{r+2}}'(\bar{x})}-\frac12t^{-\nu-\epsilon} \int q\partial_xq\frac{\omega_{{r+2}}''(\bar{x})}{\sqrt{\omega_{{r+2}}'(\bar{x})}}  .
\end{equation*}
Using again the Cauchy-Schwarz inequality and then \eqref{def:omegabis} and using \eqref{bound:q_0}, we deduce
\begin{align}
\left\| q\partial_xq \sqrt{\omega_{{r+2}}'(\bar{x})}\right\|_{L^{\infty}}^2 
&\lesssim \|q\|_{L^2}^2\int (\partial_x^2q)^2 \omega_{{r+2}}'(\bar{x})
+t^{-2\nu-2\epsilon}\|q\|_{L^2}^2 \int (\partial_xq)^2 \frac{\big(\omega_{{r+2}}''(\bar{x})\big)^2}{\omega_{{r+2}}'(\bar{x})} \nonumber \\
&\lesssim \|q\|_{L^2}^2\left[ \int (\partial_x^2q)^2 \omega_{{r+2}}'(\bar{x})
+t^{-2\nu-2\epsilon}\int (\partial_xq)^2 \omega_{{r+2}}'(\bar{x})\right]
 . \label{Sobo:firstderiv}
\end{align}
Therefore, for $x_0$ large enough, we obtain
\begin{align*}
E_{8,2} &\lesssim t^{\nu+\epsilon} \|q\|_{L^2}^2 \left\| q\partial_xq \sqrt{\omega_{{r+2}}'(\bar{x})}\right\|_{L^\infty}^2 \\
 & \lesssim \delta(x_0^{-1}) \left[t^{\nu+\epsilon} \int (\partial_x^2q)^2 \omega_{{r+2}}'(\bar{x})
 +t^{-\nu-\epsilon} \int (\partial_xq)^2\omega_{{r+2}}'(\bar{x}) \right]\\ 
 & \le -\frac14 E_2-\frac 1{16}E_4  .
 \end{align*}

Next we deal with $E_{8,3}$. It is clear that $E_{8,3}^- \lesssim t^{-10}$. Moreover, since $\omega_{r+2}'(\bar{x}) \sim \omega_r'(\bar{x})$ for $\bar{x}<2$, we deduce from \eqref{zone} and the Cauchy-Schwarz inequality that 
\begin{align*}
 t^{\nu+\epsilon}\int_{-t^{-1-3\nu-2\epsilon}<\bar{x}<2}&|q||\partial_xq||f_0^3||f_0'| \omega'_{r+2}(\bar{x})\\&\lesssim t^{\nu+\epsilon} \int_{x>\frac12t^{\beta}} |q| |\partial_xq| |x|^{-4\theta-1}\omega'_r(\bar{x})\\
 &\lesssim t^{-\beta(4\theta+1)+\epsilon}\int q^2 \omega'_r(\bar{x}) +t^{2\nu-\beta(4\theta+1)+\epsilon}\int (\partial_xq)^2 \omega'_{r+2}(\bar{x}) \\ 
 & \le ct^{-3\nu-\epsilon}\int q^2 \omega'_r(\bar{x})-\frac1{32}E_4  ,
\end{align*}
for $t_0$ large enough, since $\beta(4\theta+1)-3\nu=\frac52(3\beta-1)>0$. Finally, we use that $\omega'_{r+2}(\bar{x})=(r+2)\bar{x}\bar{x}^{\frac{r-1}2}\bar{x}^{\frac{r+1}2}$ when $\bar{x}>2$. Thus, we get by using again the Cauchy-Schwarz inequality that 
\begin{align*}
 t^{\nu+\epsilon}\int_{\bar{x}>2}|q||\partial_xq||f_0^3||f_0'| \omega'_{r+2}(\bar{x})&\lesssim \int_{x>t^{\beta}} |x|^{-4\theta} |q|\bar{x}^{\frac{r-1}2} |\partial_xq|\bar{x}^{\frac{r+1}2}\\
 &\lesssim t^{-\nu-4\beta\theta-\epsilon}\int q^2 \omega'_r(\bar{x}) +t^{\nu-4\beta\theta+\epsilon}\int (\partial_xq)^2 \omega'_{r+2}(\bar{x}) \\ 
 & \le ct^{-3\nu-\epsilon}\int q^2 \omega'_r(\bar{x})-\frac1{64}E_4  ,
\end{align*}
by taking $t_0$ large enough, since $4\beta\theta-2\nu=2(3\beta-1)$.

Finally we estimate $E_{8,4}$. First, we have from Young's inequality 
\begin{equation*}
E_{8,4} \lesssim 
t^{\nu+\epsilon} \int q^4 (\partial_x q)^2\omega_{r+2}'(\bar{x}) + t^{\nu+\epsilon} \int q^4 (f_0')^2\omega_{r+2}'(\bar{x})
=: E_{8,2}+E_{8,5}  .
\end{equation*}
As before, it is clear that $E_{8,5}^- \lesssim t^{-10}.$ By interpolation, we have $q^4|f_0'|^2 \lesssim q^2f_0^4|x|^{-2}+q^6|x|^{-2}.$ Moreover, we get arguing as in \eqref{omega:est} that $\omega'_{r+2}(\bar{x})x^{-2} \lesssim t^{-2\nu-2\epsilon} \omega'_r(\bar{x})$ for $\bar{x}>-t^{1-3\nu-2\epsilon}$. Hence, we deduce using the estimates for $M_{5,1}^0$, $M_{5,1}^+$ and $M_{5,2}$ that 
\begin{equation*}
E_{8,5}^0+E_{8,5}^+ \lesssim M_{5,1}^0+M_{5,1}^++M_{5,2} 
\lesssim t^{-3\nu-\epsilon} \int q^2 \omega'_r(\bar{x})+t^{-\nu-\epsilon}\int (\partial_xq)^2 \omega'_r(\bar{x}) .
\end{equation*}

\smallskip

\noindent \emph{Estimate for $E_9$}.
We have
\begin{equation*} 
\big| E_9 \big| \leq 2t^{2\nu+2\epsilon} \left( \int (\partial_xq)^2 \omega_{r+2}'(\bar{x})\right)^{\frac12}\left(\int (F_0')^2 \frac{\omega_{r+2}^2(\bar{x})}{\omega_{r+2}'(\bar{x})}\right)^{\frac12} 
\le -\frac 1{2^7} E_4 + c\widetilde{E}_{9}\end{equation*}
where $\widetilde{E}_{9}=t^{5\nu+3\epsilon}\int (F_0')^2 \frac{\omega_{r+2}^ 2(\bar{x})}{\omega_{r+2}'(\bar{x})}$.
From \eqref{def:f_0}, it follows that 
for $x>\frac 14 x_0$, 
$|F_0'| \lesssim |x|^{-\theta-4}$
and for $x<\frac 14 x_0$, $F_0'=0$. 

For $\bar{x} \ge 2$, it holds $ \frac{\omega_{r+2}^ 2(\bar{x})}{\omega_{r+2}'(\bar{x})} \lesssim |\bar{x}|^{r+3}
 \lesssim t^{-(\nu+\epsilon)(r+3)}|x|^{r+3}$. Hence, 
\begin{align*}
t^{5\nu+3\epsilon}\int_{\bar x>2} (F_0')^2 \frac{\omega_{r+2}^ 2(\bar{x})}{\omega_{r+2}'(\bar{x})}
&\lesssim t^{2\nu-r(\nu+\epsilon)}\int_{\bar x>2} |x|^{-2(\theta+4)}| {x}|^{{r}+3} 
\\ & \lesssim t^{2\nu-r(\nu+\epsilon)}\int_{x>t^\beta} |x|^{-2\theta+{r}-5}
\\ & \lesssim t^{2\nu-r(\nu+\epsilon)}t^{-\beta(2\theta-{r}+4)}
=t^{-1+\frac{3\beta-1}2 (r-5)-r\epsilon},
\end{align*}
since $2\theta-r+4>0$ by assumption, and using \eqref{pourtt}.

For $-t^{-1-3\nu-2\epsilon}< \bar x<2$, it holds $ \frac{\omega_{r+2}^ 2(\bar{x})}{\omega_{r+2}'(\bar{x})} \lesssim 1$
and $x\geq \frac 12 t^\beta$ (from \eqref{zone}) so that 
\begin{align*} 
t^{5\nu+3\epsilon}\int_{-t^{-1-3\nu-2\epsilon}<\bar x<2}(F_0')^2 \frac{\omega_{r+2}^ 2(\bar{x})}{\omega_{r+2}'(\bar{x})}&\lesssim t^{5\nu+3\epsilon} \int_{x>\frac 12t^{\beta}}x^{-2(\theta+4)} 
\\
&\lesssim t^{5\nu+3\epsilon}t^{-\beta(2\theta+7)}=t^{-12\beta+3+\epsilon}  .
\end{align*}
Last, as before,
\begin{equation*}
t^{5\nu+3\epsilon}\int_{\bar x<-t^{-1-3\nu-2\epsilon}} (F_0')^2 \frac{\omega_{r+2}^ 2(\bar{x})}{\omega_{r+2}'(\bar{x})} 
 \lesssim t^{-10}  .
\end{equation*}

\smallskip

\noindent \emph{Estimate for $E_{10}$}. We have
\begin{equation*}
|E_{10}|\lesssim 
t^{2\nu+2\epsilon}\int |q| f_0^4 |F_0| \omega_{r+2}(\bar{x}) + t^{2\nu+2\epsilon}\int |q|^5 |F_0| \omega_{r+2}(\bar{x})
=:E_{10,1}+E_{10,2}.
\end{equation*}

First,
\begin{equation*}
E_{10,1}\lesssim t^{-3\nu-\epsilon} \int q^2 \omega_{r}'(\bar{x}) + 
t^{7\nu+5\epsilon} \int f_0^8F_0^2\frac{\omega_{r+2}^2(\bar{x})}{\omega_{r}'(\bar{x})}  .
\end{equation*}
Then,
\begin{align*}
t^{7\nu+5\epsilon} \int_{\bar x>2} f_0^8F_0^2\frac{\omega_{r+2}^2(\bar{x})}{\omega_{r}'(\bar{x})}
&\lesssim t^{7\nu+5\epsilon} \int_{\bar x>2} x^{-10\theta-6}|\bar x|^{r+5} \\ &
\lesssim t^{2\nu-r\nu-r\epsilon} \int_{x>t^\beta} x^{-10\theta+r-1}\\
&\lesssim t^{2\nu-r(\nu+\epsilon)-\beta(10\theta-r)} \\ &\lesssim t^{2\nu-r(\nu+\epsilon)}t^{-\beta(2\theta-{r}+4)}
=t^{-1+\frac{3\beta-1}2 (r-5)-r\epsilon},
\end{align*}
for $0<\epsilon$ small enough since $\theta>\frac12$,
 \begin{equation*}
t^{7\nu+5\epsilon} \int_{-t^{-1-3\nu-2\epsilon}<\bar x<2} f_0^8F_0^2\frac{\omega_{r+2}^2(\bar{x})}{\omega_{r}'(\bar{x})}
\lesssim t^{7\nu-\beta(10\theta+5)+5\epsilon},
\end{equation*}
and
\begin{equation*}
t^{7\nu+5\epsilon} \int_{\bar x<-t^{-1-3\nu-2\epsilon}} f_0^8F_0^2\frac{\omega_{r+2}^2(\bar{x})}{\omega_{r}'(\bar{x})}
\lesssim t^{-10}  .
\end{equation*}

Now, we deal with $E_{10,2}$. On the one hand, we estimate as before $E_{10,2}^- \lesssim t^{-10}$. On the other hand, we deduce by using \eqref{zone} and \eqref{omega:est} that
\begin{equation*}
E_{10,2}^0+E_{10,2}^+\lesssim t^{\nu+\epsilon}
\int_{x>\frac12t^{\beta}} |q|^5 |x|^{-(\theta+2)} \omega_{r+2}'(\bar x) \lesssim t^{+\nu+\epsilon} \left\|q^2 \sqrt{\omega_{r+2}'(\bar x)}\right\|_{L^\infty}^2\int_{x>\frac12t^\beta} |q| |x|^{-(\theta+2)}.
\end{equation*}
Thus, it follows arguing as in \eqref{M52} that for $t_0$ large enough,
\begin{align*}
E_{10,2}^0+E_{10,2}^+
&\lesssim t^{\nu+\epsilon} \left[\int_{x>\frac12t^\beta} |x|^{-2(\theta+2)}\right]^{\frac 12} \|q\|_{L^2}^3
\left[\int (\partial_x q)^2 \omega_{r+2}'(\bar x)
+t^{-2\nu-2\epsilon}\int q^2 \frac{(\omega_{r+2}''(\bar x))^2}{\omega_{r+2}'(\bar x)}\right]\\
&\lesssim t^{\nu-(\theta+\frac32)\beta+\epsilon}
\left[\int (\partial_x q)^2 \omega_{r+2}'(\bar x)
+t^{-2\nu-2\epsilon}\int q^2 \omega_r'(\bar x)\right] \\ 
& \le-\frac1{2^8}E_4+t^{-3\nu-\epsilon} \int q^2 \omega_r'(\bar x)  ,
\end{align*}
since $(\theta+\frac32)\beta=\frac{11\beta}{4}-\frac14>2\nu=1-\beta$, thanks to \eqref{def:gar}.

\smallskip

\noindent \emph{Estimate for $E_{11}$}. As before, $E_{11}^-\lesssim t^{-10}$. Moreover, observe from \eqref{def:omega}, that 
\begin{equation} \label{omega:estbis}
\begin{cases}
 \omega_{{r+2}}(\bar{x})=\frac{\bar{x}}{r+2} \omega_{r+2}'(\bar{x})&\mbox{for $\bar x>2$,}
\\
\omega_{{r+2}}(\bar{x}) \lesssim \omega_{r+2}'(\bar{x}) & 
\mbox{for $-t^{1-3\nu-2\epsilon}<\bar x<2$}
.\end{cases}
\end{equation} 
Then, it follows that for $t_0$ large enough,
\begin{align*}
E_{11}^0+E_{11}^+
&\le ct_0^{-\frac{3\beta-1}{2}+\epsilon}t^{-\nu+\epsilon} \int (\partial_x q)^2 \omega_{r+2}'(\bar x)
+\frac{2}{r+2}(\nu+\epsilon)t^{-1+2\nu+2\epsilon}\int_{\bar{x}>2}(\partial_x q)^2 \bar{x}\omega_{r+2}'(\bar x)
\\
& \le-\frac1{2^9}E_4-\frac2{r+2}E_5^+  ,
\end{align*}
since $1-3\nu=\frac{3\beta-1}2$ and $0<\epsilon<\frac{3\beta-1}4$.

\smallskip

\noindent \emph{Estimate for $E_{12}$}. On the one hand, it holds $E^{-}_{12} \lesssim t^{-10}$. On the other, we observe arguing as for $E_7$ and using \eqref{omega:estbis} that 
\begin{equation*} 
|E_{12}^0|+|E_{12}^+| \lesssim t^{-\epsilon}\big(E_{6,1}+E_{6,2}\big)+E_{7,1}^++E_{7,2}^+  ,
\end{equation*}
so that those terms are estimated similarly. 

\smallskip 

Gathering all those estimates, we obtain in conclusion, for some $c>0$, 
\begin{align*}
&E_r'+c\int \left[t^{-\nu+\epsilon} (\partial_xq)^2 
+t^{2\nu+2\epsilon-1} \bar x_+(\partial_x q)^2 
+t^{\nu+\epsilon} (\partial_x^2q)^2\right]\omega_{r+2}'(\bar{x})\, dx
\\ & \quad \lesssim 
\int\left[t^{-3\nu-\epsilon} q^2
+t^{-1} \bar x_+ q^2
+t^{-\nu-\epsilon}(\partial_xq)^2\right]\omega_{r}'(\bar{x})\, dx +t^{-1+\frac{3\beta-1}2 (r-5)-r\epsilon}.
\end{align*}
Therefore, we conclude the proof of \eqref{decay_claim.2} by using \eqref{decay_claim.1}, integrating the previous estimate over $[t_0,t]$ and using \eqref{assump:epsilon}.
\end{proof}

\begin{proof}[Proof of Proposition \ref{decay:q_0} in the case $k=0$]
First, we look for an estimate on $\int (\partial_xq)^2 \omega_{r+2}$ from the energy estimate.
Arguing as in \eqref{M51}, \eqref{M51bis} and \eqref{M52}, we get that 
\begin{equation*}
t^{2\nu+2\epsilon}\int q^2f_0^4 \omega_{{r+2}}(\bar{x}) \lesssim 
 \int q^2 \omega_{{r}}(\bar{x}) +t^{-10}
 \end{equation*}
and
\begin{equation*}
t^{2\nu+2\epsilon}\int q^6 \omega_{{r+2}}(\bar{x}) \lesssim t^{2\nu+2\epsilon}\|q\|_{L^2}^4 \int (\partial_xq)^2 \omega_{{r+2}}(\bar{x})
 +\|q\|_{L^2}^4 \int q^2 \omega_{{r}}(\bar{x})  . \nonumber
 \end{equation*}
Thus, it follows that for $x_0$ large enough
\begin{align*} 
E_r(t) &\ge t^{2\nu+2\epsilon}\int (\partial_xq)^2 \omega_{{r+2}}(\bar{x}) -ct^{2\nu+2\epsilon}\int \big( q^6+q^2f_0^4 \big) \omega_{{r+2}}(\bar{x}) \\ &
\ge \frac12t^{2\nu+2\epsilon}\int (\partial_xq)^2 \omega_{{r+2}}(\bar{x})-cM_r(t)-ct^{-10}  .
\end{align*}
Hence, we deduce by using \eqref{decay_claim.1} and \eqref{decay_claim.2} that, for $0<r<2\theta+4$,
\begin{multline}\label{decay_claim.3} 
t^{2(\nu+\epsilon)}\int (\partial_xq)^2 \omega_{{r+2}}(\bar{x})+\int_{t_0}^t\int\left[s^{-\nu+\epsilon} (\partial_xq)^2 
+s^{2(\nu+\epsilon)-1} \bar x_+(\partial_x q)^2 
+s^{\nu+\epsilon} (\partial_x^2q)^2\right]\omega_{r+2}'(\bar{x}) \\ \lesssim 
\begin{cases}
t^{\frac{3\beta-1}2 (r-5)-r\epsilon} & \mbox{if $r>5$}\\
t_0^{-\frac{3\beta-1}2 (5-r)-5\epsilon} & \mbox{if $r<5$.}
\end{cases}
\end{multline}

\smallskip

Now, we give the proof of estimate \eqref{decay:q_0.1} in the case $k=0$. 
By the fundamental theorem of calculus and the properties of $\omega_r$
it holds, for any $x$, 
\begin{align*}
t^{\nu+\epsilon}q^2(t,x) \omega_{{r}+1}(\bar{x}) &\lesssim t^{\nu+\epsilon}\int_x^{+\infty} |q||\partial_xq|\omega_{{r}+1}(\bar{x})+\int_x^{+\infty} q^2 \omega_{{r}+1}'(\bar{x}) \\ & \lesssim t^{2(\nu+\epsilon)}\int (\partial_xq)^2\omega_{{r}+2}(\bar{x})+\int q^2\omega_{{r}}(\bar{x})  .
\end{align*}
Hence, we obtain from estimates \eqref{decay_claim.1} and \eqref{decay_claim.3} that 
\begin{equation} \label{decay_claim.4}
t^{(\nu+\epsilon)}q^2(t,x) \omega_{{r}+1}(\bar{x}) \lesssim 
\begin{cases}
t^{\frac{3\beta-1}2 (r-5)-r\epsilon} & \mbox{if $r>5$}\\
t_0^{-\frac{3\beta-1}2 (5-r)-r\epsilon} & \mbox{if $r<5$}
.\end{cases}
\end{equation}

For $x>2t^{\beta}$, we have that $\bar{x}>t^{\beta-\nu}>2$, for $t \ge t_0$ large enough.
Then, we deduce from 
the properties of $\omega_r$, estimate \eqref{decay_claim.4} and the identity \eqref{pourtt} that
\begin{equation} \label{decay_claim.4bis}
q^2(t,x) \left| \frac{x-t^{\beta}}{t^{\nu+\epsilon}}\right|^{{r}+1} \lesssim t^{2-\beta-\nu({r}+1)-\beta(2\theta-{r}+4)-(r+1)\epsilon}  ,
\end{equation}
and thus, for such $t\geq t_0$ and $x>2t^\beta$, we have
\begin{equation*}
q^2(t,x)\lesssim x^{-(r+1)}t^{2-\beta-\beta(2\theta-r+4)}.
\end{equation*}
Taking $r$ close enough to $2\theta+4$ so that $2-\beta-\beta(2\theta-r+4)>0$, 
we conclude the proof of \eqref{decay:q_0.1} in the case $k=0$ and $\kappa_0=2$ (see Remark~\ref{rem:decay})
using $t^\beta<x$.
\end{proof}

\begin{proof}[Proof of Proposition \ref{decay:q_0} in the case $k \ge 1$] We will prove estimate \eqref{decay:q_0.1} in the case where $k \ge 1$ by an induction on $k$.
\begin{definition} Let $l \in \mathbb N$, $0<r<2\theta+4$, $r\neq 5$ and $0<\epsilon<\frac{3\beta-1}{20}|r-5|$.
We say that the induction hypothesis $\mathcal{H}_l$ holds true if 
\begin{equation} \label{def:Hl}
t^{2l(\nu+\epsilon)}\int (\partial_x^lq)^2 \omega_{r+2l}(\bar{x}) \, dx
 \lesssim 
\begin{cases}
t^{\frac{3\beta-1}2 (r-5)-r\epsilon} & \mbox{if $r>5$}\\
t_0^{-\frac{3\beta-1}2 (5-r)-r\epsilon} & \mbox{if $r<5$}
.\end{cases} 
\end{equation}
\end{definition}

First, it is clear arguing as in \eqref{decay_claim.4} that if $\mathcal{H}_l$ and $\mathcal{H}_{l-1}$ hold true for some $l \in \mathbb N$, $l \ge 1$, then 
\begin{equation} \label{decay_claim.5}
t^{(2l-1)(\nu+\epsilon)}(\partial_x^{l-1}q)^2(t,x) \omega_{r+2l-1}(\bar{x}) \lesssim 
\begin{cases}
t^{\frac{3\beta-1}2 (r-5)-r\epsilon} & \mbox{if $r>5$}\\
t_0^{-\frac{3\beta-1}2 (5-r)-r\epsilon} & \mbox{if $r<5$}
.\end{cases}
\end{equation}
Notice in particular that \eqref{decay_claim.5} would imply \eqref{decay:q_0.1} in the case $k=l-1$ arguing as in \eqref{decay_claim.4bis}.

\smallskip

Thus, it suffices to prove that \eqref{def:Hl} hold for any $l \in \mathbb N$ to conclude the proof of Proposition \ref{decay:q_0}. Observe from \eqref{decay_claim.1} and \eqref{decay_claim.2} that $\mathcal{H}_0$ and $\mathcal{H}_1$ hold true. 

Assume that \eqref{def:Hl} holds true for $l=0,1,\cdots,k-1$. 
The next lemma will prove that \eqref{def:Hl} is true for $l=k$, which will conclude the proof of Proposition \ref{decay:q_0}.
\end{proof}

\begin{lemma} \label{decay_lemma.3}
Let $k \in \mathbb N$, $k \ge 2$, $0<r<2\theta+4$, $r\neq 5$ and $0<\epsilon<\frac{3\beta-1}{20}|r-5|$. Assume moreover that \eqref{def:Hl} holds true for $l=0,1,\cdots,k-1$.
Define 
\begin{equation*}
F_{r,k}(t):= t^{2k(\nu+\epsilon)}\int (\partial_x^kq)^2(t,x) \omega_{r+2k}(\bar{x}) \, dx , \quad \text{where} \quad \bar{x}=\frac{x-t^{\beta}}{t^{\nu+\epsilon}}  .
\end{equation*}
Then, for $x_0>1$ large enough, and any $t \ge t_0=\left(\frac{x_0}2\right)^{\frac1{\beta}}$,
\begin{multline*} 
F_{r,k}(t)+\int_{t_0}^t\int\left[s^{(2k-1)(\nu+\epsilon)-2\nu} (\partial_x^kq)^2
+s^{(2k-1)(\nu+\epsilon)}(\partial_x^{k+1}q)^2\right]\omega_{r+2k}'(\bar{x})\, dx ds \\
 \lesssim 
\begin{cases}
t^{\frac{3\beta-1}2 (r-5)-r\epsilon} & \mbox{if $r>5$}\\
t_0^{-\frac{3\beta-1}2 (5-r)-r\epsilon} & \mbox{if $r<5$.}
\end{cases} 
\end{multline*}
\end{lemma}

\begin{proof} 
We differentiate $F_{k,r}$ with respect to time and integrate by parts to obtain
\begin{align*}
F'_{r,k} &= -3t^{(2k-1)(\nu+\epsilon)}\int (\partial_x^{k+1}q)^2\omega_{r+2k}'(\bar{x})+t^{(2k-3)(\nu+\epsilon)}\int (\partial_x^{k}q)^2\omega_{r+2k}'''(\bar{x})
\\ 
&\quad- \beta t^{(2k-1)(\nu+\epsilon)-2\nu}\int (\partial_x^{k}q)^2\omega_{r+2k}'(\bar{x})-(\nu+\epsilon) t^{2k(\nu+\epsilon)-1} \int (\partial_x^{k}q)^2  \bar{x}\omega_{r+2k}'(\bar{x}) 
\\ &
\quad +2t^{2k(\nu+\epsilon)}\int \partial_x^{k}\left( (q+f_0)^5-f_0^5\right) \left((\partial_x^{k+1}q)\omega_{r+2k}(\bar{x})+t^{-(\nu+\epsilon)}(\partial_x^{k}q)\omega_{r+2k}'(\bar{x})\right)
\\ & \quad +2t^{2k(\nu+\epsilon)}\int (\partial_x^kq) F_0^{(k)} \omega_{r+2k}(\bar{x}) 
+2k(\nu+\epsilon)t^{2k(\nu+\epsilon)-1}\int (\partial_x^kq)^2 \omega_{r+2k}(\bar{x})
\\& =: F_1 +F_2+F_3+F_4+F_5+F_6+F_7  .
\end{align*}
First observe that $F_1 \le 0$ and $F_3 \le 0$. Moreover, arguing as in the proof of \eqref{decay_claim1.1}, we have that, for $t_0$ large enough, $|F_2|\leq -\frac 12 F_3$. 

Next, we use the same notation as in \eqref{notation} for $F_j$, $j=4,\cdots,6$. We have $F_4^+\leq 0$.
Moreover, as for the estimate of $M_4$ in the proof of \eqref{decay_claim.1},
it holds $F_4^0+F_4^-\leq -\frac 14 F_3+Ct^{-10}$.

\smallskip

\noindent \emph{Estimate for $F_5$.} We will only explain how to estimate the terms 
\begin{equation*} 
F_{5,1}:=2t^{2k(\nu+\epsilon)}\int \partial_x^{k}(q^5) \left((\partial_x^{k+1}q)\omega_{r+2k}(\bar{x})+t^{-(\nu+\epsilon)}(\partial_x^{k}q)\omega_{r+2k}'(\bar{x})\right)
\end{equation*}
and
\begin{equation*} 
F_{5,2}:=2t^{2k(\nu+\epsilon)}\int \partial_x^{k}(f_0^4q) \left((\partial_x^{k+1}q)\omega_{r+2k}(\bar{x})+t^{-(\nu+\epsilon)}(\partial_x^{k}q)\omega_{r+2k}'(\bar{x})\right)
\end{equation*}
since the other ones are estimated interpolating between these estimates. 

We deal first with $F_{5,1}$. We deduce from the Cauchy-Schwarz and Young inequalities that 
\begin{equation*}
\big| F_{5,1} \big| \lesssim -\frac12 F_1-\frac18 F_3+t^{(2k+1)(\nu+\epsilon)} \int \left(\partial_x^k(q^5) \right)^2 \frac{\omega^2_{r+2k}(\bar{x})}{\omega_{r+2k}'(\bar{x})}=:-\frac12 F_1-\frac18 F_3+\widetilde{F}_{5,1}  .
\end{equation*}
We observe by using the Leibniz rule, the Cauchy-Schwarz inequality and the properties of $\omega$ in \eqref{def:omega} that 
\begin{equation*} 
\widetilde{F}_{5,1} \lesssim \!\!\!\! \sum_{\genfrac{}{}{0pt}{2}{k_1+\cdots+k_5=k}{k_1 \le \cdots \le k_5}} t^{(2k_1+\cdots 2k_5+1)(\nu+\epsilon)}\int \prod_{j=1}^5 (\partial_x^{k_j}q)^2 \omega_{r+2k_1+\cdots+2k_5+1} =:\!\!\!\!
\sum_{\genfrac{}{}{0pt}{2}{k_1+\cdots+k_5=k}{k_1 \le \cdots \le k_5}} \widetilde{F}_{5,1}(k_1,\cdots,k_5) .
\end{equation*}

When $k_5=k$, we have $k_1=\cdots=k_4=0$. It follows then applying \eqref{decay_claim.4} with $r=\frac12$ that
\begin{equation*}
\widetilde{F}_{5,1}(0,0,0,0,k) \le \Big\| t^{(\nu+\epsilon)}q^2\omega_{\frac32} \Big\|_{L^{\infty}_{x}}^4t^{(2k-3)(\nu+\epsilon)} \int (\partial_x^ kq)^2 \omega_{r+2k}'(\bar{x}) \le \delta(t_0^{-1}) F_3  .
\end{equation*}

When $k_5 \le k-1$, we consider two different cases. First if $k=2$, then $k_5=k_4=1$ and $k_1=k_2=k_3=0$, we deduce from \eqref{decay_claim.4} with $r=1$ and then \eqref{decay_claim.3} that
\begin{align*}
\widetilde{F}_{5,1}(0,0,0,1,1) &\le \Big\| t^{(\nu+\epsilon)}q^2\omega_{2} \Big\|_{L^{\infty}_{x}}^3\|\partial_xq\|_{L^{\infty}_x}^2 \left(t^{2(\nu+\epsilon)} \int (\partial_xq)^2 \omega_{r+2}(\bar{x}) \right)
\\ &\lesssim 
\begin{cases}
t^{\frac{3\beta-1}2 (r-5)-r\epsilon} & \mbox{if $r>5$}\\
t_0^{-\frac{3\beta-1}2 (5-r)-5\epsilon} & \mbox{if $r<5$}.
\end{cases}
\end{align*}

In the case where $k_5 \le k-1$ and $k \ge 3$, observe that $k_4 \le k-2$. By using the induction hypothesis \eqref{def:Hl} for $l=0,1,\cdots,k-1$, we deduce that \eqref{decay_claim.5} hold for $l=1,\cdots,k-1$. Thus, it follows that
\begin{align*}
\widetilde{F}_{5,1}(k_1,k_2,k_3,k_4,k_5) &\le \prod_{j=1}^4\Big\| t^{(2k_j+1)(\nu+\epsilon)}(\partial_x^{k_j}q)^2\omega_{2k_j+\frac32} \Big\|_{L^{\infty}_{x}} \left[t^{(2k_5-3)(\nu+\epsilon)} \int (\partial_x^{k_5}q)^2 \omega_{r+2k_5}(\bar{x}) \right]
\\ &\lesssim 
\begin{cases}
t^{\frac{3\beta-1}2 (r-5)-r\epsilon} & \mbox{if $r>5$}\\
t_0^{-\frac{3\beta-1}2 (5-r)-5\epsilon} & \mbox{if $r<5$}
.\end{cases}
\end{align*}

Now we deal with $F_{5,2}$. By using the Leibniz rule and integration by parts, we decompose
\begin{align*} 
F_{5,2}&:=2\sum_{l=0}^{k-1}\begin{pmatrix} k \\ l \end{pmatrix}t^{2k(\nu+\epsilon)}\int \partial_x^{k-l}(f_0^4)(\partial_x^lq) (\partial_x^{k+1}q)\omega_{r+2k}(\bar{x})-t^{2k(\nu+\epsilon)}\int \partial_x(f_0^4)(\partial_x^{k}q)^2 \omega_{r+2k}(\bar{x}) \\ &
\quad +t^{(2k-1)(\nu+\epsilon)}\int f_0^4 (\partial_x^{k}q)^2\omega_{r+2k}'(\bar{x}) 
+2\sum_{l=0}^{k-1}\begin{pmatrix} k \\ l \end{pmatrix}t^{(2k-1)(\nu+\epsilon)}\int \partial_x^{(k-l)}(f_0^4) (\partial_x^{k}q)\omega_{r+2k}'(\bar{x}) \\ &=: F_{5,2,1}+F_{5,2,2}+F_{5,2,3}+F_{5,2,4} .
\end{align*}
First, it is clear that $ |F_{5,2,1}^-|+F_{5,2,2}^-+F_{5,2,3}^-+F_{5,2,4}^- \lesssim t^{-10}$. Next, we observe arguing as in \eqref{omega:est} that 
\begin{equation} \label{omega:est.2}
\omega_{r+2k}(\bar{x})|x|^{-(k-l)-1} \lesssim t^{-((k-l)+1)(\nu+\epsilon)}\omega_{r+k+l-1}(\bar{x}), \quad \text{for} \quad \bar{x}>-t^{1-3\nu-2\epsilon} .
\end{equation}
Hence, it follows from \eqref{def:f_0}, \eqref{zone},  $\theta>\frac12$ and the Cauchy-Schwarz inequality that 
\begin{align*} 
F_{5,2,1}^0+F_{5,2,1}^+ &\lesssim \sum_{l=0}^{k-1}t^{(k+l-1)(\nu+\epsilon)}\int_{x>\frac12t^{\beta}} |x|^{-(4\theta-1)}|\partial_x^lq| |\partial_x^{k+1}q|\omega_{r+k+l-1}(\bar{x}) 
\\ & \lesssim \sum_{l=0}^{k-1} \left( t^{2l(\nu+\epsilon)} \int (\partial_x^lq)^2\omega_{r+2l}(\bar{x})\right)^{\frac12} \left( t^{(2k-2)(\nu+\epsilon)} \int (\partial_x^{(k+1)}q)^2\omega_{r+2k}'(\bar{x})\right)^{\frac12} \\
& \lesssim -\frac14F_1+\sum_{l=0}^{k-1}t^{2l(\nu+\epsilon)} \int (\partial_x^lq)^2\omega_{r+2l}(\bar{x})  .
\end{align*}
Now, observe from \eqref{def:f_0}, \eqref{zone} and \eqref{omega:est} that, for $t_0$ large enough,
 \begin{align*} 
F_{5,2,2}^0+F_{5,2,2}^+ +F_{5,2,3}^0+F_{5,2,3}^+&\lesssim t^{(2k-1)(\nu+\epsilon)}\int_{x>\frac12t^{\beta}} |x|^{-4\theta}(\partial_x^kq)^2\omega_{r+2k}'(\bar{x}) \\ &\lesssim t^{(2k-1)(\nu+\epsilon)-4\theta \beta}\int (\partial_x^kq)^2\omega_{r+2k}'(\bar{x})
\le -\frac1{2^4}F_3  ,
\end{align*}
since $4\theta \beta =5\beta-1=4-10\nu>2\nu$, thanks to \eqref{def:gar}.

Finally, by using an estimate similar to \eqref{omega:est.2},
 \begin{align*} 
F_{5,2,4}^0+F_{5,2,4}^+ &\lesssim \sum_{l=0}^{k-1}t^{(k+l-1)(\nu+\epsilon)}\int_{x>\frac12t^{\beta}} |x|^{-4\theta}|\partial_x^lq| |\partial_x^{k}q|\omega_{r+k+l}'(\bar{x}) 
\\ & \lesssim \sum_{l=0}^{k-1} \left( t^{2l(\nu+\epsilon)} \int (\partial_x^lq)^2\omega_{r+2l}(\bar{x})\right)^{\frac12} \left( t^{(2k-2)(\nu+\epsilon)-8\theta\beta} \int (\partial_x^{k}q)^2\omega_{r+2k}'(\bar{x})\right)^{\frac12} \\
& \lesssim -\frac1{2^5}F_3+\sum_{l=0}^{k-1}t^{2l(\nu+\epsilon)} \int (\partial_x^lq)^2\omega_{r+2l}(\bar{x})  .
\end{align*}

\smallskip

\noindent \emph{Estimate for $F_6$.} From the Cauchy-Schwarz inequality,
\begin{align*} 
\big| F_6 \big| &\leq 2 t^{2k(\nu+\epsilon)}\left( \int (\partial_x^kq)^2 \omega_{{r+2k}}'(\bar{x})\right)^{\frac12}\left(\int \big(F_0^{(k)}\big)^2 \frac{\omega_{{r+2k}}^ 2(\bar{x})}{\omega_{{r+2k}}'(\bar{x})} \right)^{\frac12} \\ &
\le -\frac 1{2^6} F_3 + ct^{2k(\nu+\epsilon)+3\nu+\epsilon}\int \big(F_0^{(k)}\big)^2 \frac{\omega_{{r+2k}}^ 2(\bar{x})}{\omega_{{r+2k}}'(\bar{x})}  .
\end{align*}
First,  from \eqref{def:f_0}, 
for $x>\frac 14 x_0$, 
$\big|F_0^{(k)}\big| \lesssim |x|^{-(\theta+k+3)}$ ($\theta>\frac12$),
and for $x<\frac 14 x_0$, $F_0=0$. 

For $\bar{x} \ge 2$, it holds $ \frac{\omega_{{r+2k}}^ 2(\bar{x})}{\omega_{{r+2k}}'(\bar{x})} =(r+2)^{-1} |\bar{x}|^{r+2k+1}
 \lesssim t^{-(\nu+\epsilon)(r+2k+1)}|x|^{r+2k+1}$. Hence, by using \eqref{pourtt},
\begin{align*}
t^{2k(\nu+\epsilon)+3\nu+\epsilon}\int_{\bar x>2} \big(F_0^{(k)}\big)^2 \frac{\omega_{{r+2k}}^ 2(\bar{x})}{\omega_{{r+2k}}'(\bar{x})}
&\lesssim t^{2\nu-r(\nu+\epsilon)}\int_{\bar x>2} |x|^{-2(\theta+k+3)}| {x}|^{r+2k+1} 
\\ & \lesssim t^{2\nu-r(\nu+\epsilon)}\int_{x>t^\beta} |x|^{-2\theta+{r}-5}
\\ & \lesssim t^{2\nu-r(\nu+\epsilon)}t^{-\beta(2\theta-{r}+4)}
=t^{-1+\frac{3\beta-1}2 (r-5)-r\epsilon}  ,
\end{align*}
since $2\theta-r+4>0$ by assumption.

For $-t^{-1-3\nu-2\epsilon}< \bar x<2$, it holds $ \frac{\omega_{{r+2k}}^ 2(\bar{x})}{\omega_{{r+2k}}'(\bar{x})} \lesssim 1$
and $x\geq \frac 12 t^\beta$ (from \eqref{zone}) so that 
\begin{align*} 
t^{2k(\nu+\epsilon)+3\nu+\epsilon}\int_{-t^{-1-3\nu-2\epsilon}<\bar x<2} \big(F_0^{(k)}\big)^2 \frac{\omega_{{r+2k}}^ 2(\bar{x})}{\omega_{{r+2k}}'(\bar{x})} &\lesssim t^{2k(\nu+\epsilon)+3\nu+\epsilon} \int_{x>\frac 12t^{\beta}}x^{-2(\theta+k+3)} 
\\
&\lesssim t^{2k(\nu+\epsilon)+3\nu+\epsilon}t^{-\beta(2\theta+2k+5)}\\
&\lesssim t^{-\frac{k(3\beta-1)}{2}-9\beta+2+(2k+1)\epsilon}  .
\end{align*}
Last, for $\bar x<-t^{-1-3\nu-2\epsilon}$, then $\frac{\omega_{{r+2k}}^ 2(\bar{x})}{\omega_{{r+2k}}'(\bar{x})} =8e^{\frac{\bar{x}}8}$ so that as for $M_4^-$,
\begin{equation*}
t^{2k(\nu+\epsilon)+3\nu+\epsilon}\int_{\bar x<-t^{-1-3\nu-2\epsilon}} \big(F_0^{(k)}\big)^2 \frac{\omega_{{r+2k}}^ 2(\bar{x})}{\omega_{{r+2k}}'(\bar{x})} 
 \lesssim t^{-10}  .
\end{equation*}

\smallskip

\noindent \emph{Estimate for $F_7$.} As before, $F_{7}^-\lesssim t^{-10}$. Moreover, observe from \eqref{def:omega}, that 
\begin{equation*} 
\begin{cases}
 \omega_{{r+2k}}(\bar{x})=\frac{\bar{x}}{r+2k} \omega_{r+2k}'(\bar{x})&\mbox{for $\bar x>2$,}
\\
\omega_{{r+2k}}(\bar{x}) \lesssim \omega_{r+2k}'(\bar{x}) & 
\mbox{for $-t^{1-3\nu-2\epsilon}<\bar x<2$}
.\end{cases}
\end{equation*} 
Then, it follows that for $t_0$ large enough,
\begin{align*}
F_{7}^0+F_{7}^+
&\le ct_0^{-\frac{3\beta-1}{2}+\epsilon}t^{-\nu+\epsilon} \int (\partial_x^k q)^2 \omega_{r+2}'(\bar x)
+\frac{2k}{r+2k}(\nu+\epsilon)t^{-1+2\nu+2\epsilon}\int_{\bar{x}>2}(\partial_x^k q)^2 \bar{x}\omega_{r+2}'(\bar x)
\\
& \le-\frac1{2^7}F_3-\frac{2k}{r+2k}F_4^+  ,
\end{align*}
since $1-3\nu=\frac{3\beta-1}2$ and $0<\epsilon<\frac{3\beta-1}4$.

\smallskip

Therefore, we complete the proof of Lemma \ref{decay_lemma.3} combining all those estimates  with the induction hypothesis \eqref{def:Hl} for $l=0,1,\cdots,k-1$.
\end{proof}

\section{Decomposition around the soliton} \label{section:decomp_sol}

\subsection{Linearized operator}
Here, we recall some properties of the linearized operator $\mathcal{L}$ around the soliton $Q$ defined in \eqref{def:L}. 
We first introduce the function space $\mathcal{Y}$: 
\begin{equation*}
\mathcal{Y}:= \Big\{ \phi \in \mathcal{C}^{\infty}(\mathbb R : \mathbb R) : \forall \, k \in \mathbb N, \, \exists \, C_k, \, r_k>0 \ \text{s.t.} \ |\phi^{(k)}(y)| \le C_k(1+|y|)^{r_k}e^{-|y|}, \ \forall \, y \in \mathbb R\Big\} \, .
\end{equation*}

\begin{lemma}[Properties of the linearized operator $\mathcal{L}$]
The self-adjoint operator $\mathcal{L}:H^2(\mathbb R) \subseteq L^2(\mathbb R) \to L^2(\mathbb R)$ defined in \eqref{def:L} satisfies the following properties.
\begin{itemize} 
\item[(i)] \emph{Spectrum of $\mathcal{L}$:} the operator $\mathcal{L}$ has only one negative eigenvalue $-8$ associated to the eigenfunction $Q^3$; $\ker \mathcal{L}=\{aQ' : a \in \mathbb R\}$; and $\sigma_{ess}( \mathcal{L})=[1,+\infty)$.
\item[(ii)] \emph{Scaling:} $\mathcal{L}\Lambda Q=-2Q$ and $(Q,\Lambda Q)=0$, where $\Lambda$ is defined in \eqref{def:lambda}.
\item[(iii)] \emph{Coercivity of $\mathcal{L}$:} 
for all $\phi \in H^1(\mathbb R)$,
\begin{equation*}
(\phi,Q^3)=(\phi,Q')=0 \implies (\mathcal{L}\phi,\phi) \ge \|\phi\|_{H^1}^2  .
\end{equation*}
Moreover, there exists $\nu_0>0$ such that, for all $f \in H^1(\mathbb R)$,
\begin{equation} \label{coercivity.2}
(\mathcal{L}\phi,\phi) \ge \nu_0\|\phi\|_{H^1}^2-\frac1\nu_0\Big((\phi,Q)^2+(\phi,y\Lambda Q)^2 +(\phi,\Lambda Q)^2 \Big)  .
\end{equation}
\item[(iv)] \emph{Invertibility:} there exists a unique $R \in \mathcal{Y}$, even, such that 
\begin{equation} \label{def:R}
\mathcal{L}R=5Q^4; \quad \text{moreover} \quad (Q,R)=-\frac34 \int Q  .
\end{equation}
\item[(v)] \emph{Invertibility (bis):} there exists a unique function $P \in \mathcal{C}^{\infty}(\mathbb R) \cap L^{\infty}(\mathbb R)$ such that $P' \in \mathcal{Y}$ and 
\begin{equation*}
(\mathcal{L}P)'=\Lambda Q, \ \lim_{y \to -\infty}P(y)=\frac12 \int Q, \ \lim_{y \to +\infty} P(y)=0  .
\end{equation*}
Moreover, 
\begin{equation} \label{prop:P}
(P,Q)=\frac1{16}\left( \int Q\right)^2>0, \quad (P,Q')=0  .
\end{equation}
\end{itemize}
\end{lemma}

\begin{proof} 
The properties (i), (ii) and (iii) are standard and we refer to Lemma 2.1 of \cite{MaMeRa1} and the references therein for their proof. Property (iv) is proved in Lemma 2.1 in \cite{MaMeRa3}, while property (v) is proved in Proposition 2.2 in \cite{MaMeRa1}.
\end{proof}

\subsection{Refined profile} We follow \cite{MaMeRa1} to define the one parameter family of approximate self similar profiles $:b\mapsto Q_b$, $|b| \ll 1$ which will provide the leading order deformation of the ground state profile $Q=Q_{b=0}$ in our construction. 

More precisely, we need to localize $P$ on the left hand side. Let $\chi \in \mathcal{C}^{\infty}(\mathbb R)$ be such that 
\begin{equation} \label{def:chi}
0\le \chi \le 1,  \quad  0 \le (\chi'')^2 \lesssim \chi' \quad \text{on} \ \mathbb R,  \quad \chi_{|_{(-\infty,-2)}} \equiv 0 \quad \text{and} \quad \chi_{|_{(-1,+\infty)}}\equiv 1  .
\end{equation}

\begin{definition} 
Let $\gamma=\frac34$. The localized profile $Q_b$ is defined by
\begin{equation} \label{def:Qb} 
Q_b(y)=Q(y)+bP_b(y)  ,
\end{equation}
where 
\begin{equation} \label{def:Pb} 
P_b(y)=\chi_b(y)P(y) \quad \text{and} \quad \chi_b(y)=\chi(|b|^{\gamma}y)  .
\end{equation}
\end{definition}

The properties of $Q_b$ are stated in the next lemma. 
\begin{lemma}[Approximate self-similar profiles $Q_b$ , \cite{MaMeRa1}] \label{refined:profile}
There exists $b^{\star}>0$ such that for $|b|<b^{\star}$, the following properties hold. 
\begin{itemize} 
\item[(i)] \emph{Estimate of $Q_b$:} for all $y \in \mathbb R$, 
\begin{equation} \label{est:Qb} \begin{aligned} 
\big|Q_b(y) \big| &\lesssim e^{-|y|}+|b| \left(\boldsymbol{1}_{[-2,0]}(|b|^{\gamma}y)+e^{-\frac{|y|}2} \right)  ; \\
\big|Q_b^{(k)}(y) \big| &\lesssim e^{-|y|}+|b|e^{-\frac{|y|}2}+ |b|^{1+k\gamma}\boldsymbol{1}_{[-2,-1]}(|b|^{\gamma}y), \quad \forall \, k \ge 1  .  
\end{aligned}\end{equation}
\item[(ii)] \emph{Equation of $Q_b$:} the error term
\begin{equation*}
-\Psi_b=\big(Q_b''-Q_b+Q_b^5 \big)'+b\Lambda Q_b-2b^2\frac{\partial Q_b}{\partial b} 
\end{equation*}
satisfies, for all $y \in \mathbb R$,
\begin{align} 
\big|\Psi_b(y) \big| &\lesssim |b|^{1+\gamma}\boldsymbol{1}_{[-2,-1]}(|b|^{\gamma}y)+b^2 \left(e^{-\frac{|y|}2} +\boldsymbol{1}_{[-2,0]}(|b|^{\gamma}y)\right)  ; \label{est:Psib} \\
\big|\Psi_b^{(k)}(y) \big| &\lesssim |b|^{1+(k+1)\gamma}\boldsymbol{1}_{[-2,-1]}(|b|^{\gamma}y)+ b^2 e^{-\frac{|y|}2}, \quad \forall \, k \ge 1  . \label{est:derivPsib} 
\end{align}
Moreover, 
\begin{equation}\label{est:PsibY}
\left|\big(\Psi_b,\phi\big)\right| \lesssim b^2, \quad \forall \phi \in \mathcal{Y}  ,
\end{equation}
and $($recalling the definition of $L^2_B$ in \eqref{def:L2B}$)$,
\begin{equation} \label{est:PsibL2}
 \big\|\Psi_b^{(k)}(y) \big\|_{L^2_B} \lesssim b^2, \ \forall \, k \ge 0  .
\end{equation}
Note that the implicit constant in \eqref{est:PsibL2} depends on the constant $B \ge 100$.
\item[(iii)] \emph{Projection of $\Psi_b$ in the direction $Q$:} 
\begin{equation} \label{est:PsibQ}
\left| \big( \Psi_b,Q \big) \right| \lesssim |b|^3  .
\end{equation}
\item[(iv)] \emph{Mass and energy properties of $Q_b$:}
\begin{align} 
 \left| \int Q_b^2-\left(\int Q^2+2b\int PQ \right) \right| &\lesssim |b|^{2-\gamma} ; \label{est:massQb}\\
 \left|E(Q_b)+b\int P Q \right| &\lesssim b^2  .\label{est:energyQb}
\end{align}
\end{itemize}
\end{lemma}

\begin{proof} The proof of Lemma \ref{refined:profile} can be found in \cite{MaMeRa1}. Actually, the properties (i), (ii) and (iv) are proved in Lemma 2.4 of \cite{MaMeRa1}. The estimate \eqref{est:PsibY} follow directly from \eqref{est:Psib} and \eqref{est:derivPsib}. Now, we explain how to prove \eqref{est:PsibL2} in the case $k=0$. It follows from \eqref{def:L2B}, \eqref{est:Psib} and the fact that $B \ge 100$ that
\begin{equation*} 
\|\Psi_b\|_{L^2_B} \lesssim |b|^{1+\gamma} \left( \int_{-2|b|^{-\gamma}}^{-|b|^{-\gamma}} e^{\frac{y}B}\right)^{\frac12}+b^2\left(\int e^{\frac{y}B-\frac{|y|}2} \right)^{\frac12}+b^2\left( \int_{-2|b|^{-\gamma}}^0 e^{\frac{y}B}\right)^{\frac12} \lesssim b^2  ,
\end{equation*}
for $|b|$ small enough. The proof of \eqref{est:PsibL2} in the case $k\ge 1$ follows in a similar way by using \eqref{est:derivPsib} instead of \eqref{est:Psib}.

Note  that we have added the term $-2b^2\frac{\partial Q_b}{\partial b}$ to the definition of $\Psi_b$ compared to the definition in \cite{MaMeRa1} in order to get a better estimate for the projection of $\Psi_b$ on $Q$. The property (iii) follows from the computation in the first formula of page 80 in \cite{MaMeRa1} together with \eqref{prop:P}.
\end{proof}

\begin{remark} For future reference, we also observe that 
\begin{equation} \label{eq:dQb:db}
\frac{\partial Q_b}{\partial b}=\chi_b P+\gamma |b|^{\gamma}y\chi'(|b|^{\gamma}y)=P_b+\gamma y\chi_b'P  .
 \end{equation}
\end{remark}

\subsection{Definition of the approximate solution} \label{sec:AS}
Let any $\frac 12 < \theta <1$.
Following \eqref{def:c0}, set
\begin{equation} \label{def:c0bis}
c_0=\frac12 (1-\theta)(2\theta-1)^{-(1-\theta)} {\int Q}>0.
\end{equation}
For such $c_0$, for $x_0$ large enough and for 
\begin{equation}\label{def:t0}
t_0=(2x_0)^{1/\beta},
\end{equation}
(our intention is to use Proposition~\ref{decay:q_0} with the value $\kappa_0=\frac 12$),
we consider $f(t,x)$ the solution of \eqref{eq:q_0}.
Let
\begin{equation*}v(t,x)=U(t,x)-f(t,x).\end{equation*}
Note that $U$ is solution \eqref{gkdv} if and only if $v$ satisfies
\begin{equation*}
\partial_t v +\partial_x(\partial_x^2v+(v+f)^5-f^5)=0 .
\end{equation*}
We renormalize the flow using $\mathcal C^1$ functions $\lambda(t)$ and $\sigma(t)$, defining
$V$, $F_0$ and $F$ as follows
\begin{equation*}
v(t,x)=\lambda^{-\frac 12}(t)  V(t,y), \quad y=\frac{x-\sigma(t)}{\lambda(t)},
\end{equation*}
\begin{equation*}
f_0(x)=\lambda^{-\frac 12}(t) F_0\left(t,y\right),\quad f(t,x)=\lambda^{-\frac 12}(t) F\left(t,y\right).
\end{equation*}
We introduce the rescaled time variable
\begin{equation}\label{eq:sbis}
s(t)= s_0+ \int_{t_0}^t \frac {d\tau}{\lambda^3(\tau)}\quad \hbox{with} \quad t_0= \frac{2\theta-1}{5-4\theta} s_0^{\frac{5-4\theta}{2\theta-1}} 
=\frac{3\beta-1}{2}s_0^{\frac2{3\beta-1}}  .
\end{equation}
Note that from \eqref{def:t0} relating $t_0$ and $x_0$, 
$s_0$ can be taken arbitrarily large provided $x_0$ is large.

From now, any time-dependent function can be seen as a function of $t \in \mathcal{I}$ or $s \in \mathcal{J}$, where $\mathcal{I}$ is an interval of the form $[s_0,s^{\star}]$
and $\mathcal{J}=s(\mathcal{I})$.
In view of the resolution of the ODE system in \eqref{resol:la:si:b}, we will work under the following assumptions on the parameters $(\lambda,\sigma,b)$:
\begin{equation}\label{BS:param}
\left\{
\begin{aligned}
&\left| \lambda (s) - s^{\frac{2(1-\theta)}{2\theta-1}}\right|\leq s^{\frac{2(1-\theta)}{2\theta-1}-\rho}  ;\\
&\left| \sigma(s)- (2\theta-1) s^{\frac1{2\theta-1}}\right|\leq s^{\frac1{2\theta-1}-\rho}  ;\\
&\left| b(s)+\frac{2(1-\theta)}{2\theta-1} s^{-1}\right|\leq s^{-1-\rho} ;
\end{aligned}
\right.
\end{equation}
where $\rho$ is a positive number satisfying 
\begin{equation} \label{def:rho}
0<\rho<\min\left\{\frac1{12},\frac{1-\theta}{3(2\theta-1)} \right\}  .
\end{equation} 

We set
\begin{equation} \label{def:EV}
\mathcal E(V)=
V_s+\partial_y(\partial_y^2 V-V+(V+F)^5-F^5) 
-\frac{\lambda_s}{\lambda}\Lambda V - \left(\frac{\sigma_s}{\lambda}-1\right) \partial_y V .
\end{equation}
Note that $u$ is solution of \eqref{gkdv} if and only if $\mathcal E(V)=0$.
We look for an approximate solution $W$ of the form
\begin{equation*}
W(s,y) = Q_{b(s)}(y) + r(s) R(y) ,
\end{equation*}
where $b(s)$ is a $\mathcal C^1$ function to be determined and where we set
\begin{equation*}
r(s)=F(s,0)=\lambda^{\frac 12}(s)f(s,\sigma(s)).
\end{equation*}
We also define (see Lemma~\ref{le:3.6})
\begin{equation} \label{def:mM}
 \vec m =\begin{pmatrix} \frac{\lambda_s}{\lambda}+b \\ \frac{\sigma_s}{\lambda}-1 \end{pmatrix} \quad \text{and} \quad \vec \MM =\begin{pmatrix} \Lambda \\ \partial_y \end{pmatrix}  .
\end{equation}

First, we gather some useful estimates for $r$ and $F$. 
\begin{lemma} \label{lemma:est:rF}
Under the assumptions \eqref{BS:param} and for $s$ large enough, it holds
\begin{equation}\label{e:r}
\left|r-c_0\lambda^{\frac 12} \sigma^{-\theta}\right|
\lesssim \lambda^{\frac 12}\sigma^{-(5\theta-2)}\lesssim s^{-3},\quad
|r|\lesssim s^{-1}  ,
\end{equation}
\begin{equation}\label{e:dr}
\left| r_s - c_0\lambda^{\frac 12}\sigma^{-\theta} \left( \frac 12 \frac{\lambda_s}{\lambda} 
-\theta\frac{\sigma_s}{\sigma} \right)\right|\lesssim |\vec m|s^{-3} + s^{-4},\quad
|r_s|\lesssim |\vec m| s^{-1}+s^{-2}  ,
\end{equation} 
\begin{equation} \label{e:F}
\sup_{y \in \mathbb R} \left\{e^{-\frac{|y|}{10}}\big|F(s,y)\big|\right\} \lesssim s^{-1}, \quad \sup_{y \in \mathbb R} \left\{e^{-\frac{|y|}{10}}\big|\partial_yF(s,y)\big|\right\} \lesssim s^{-2}  ,
\end{equation}
\begin{equation} \label{e:FLinfty}
\|F\|_{L^{\infty}(y>-2|b|^{\gamma})} \lesssim s^ {-1}, \quad \|\partial_yF\|_{L^{\infty}(y>-2|b|^{\gamma})} \lesssim s^{-2}  ,
\end{equation}
\begin{equation} \label{e:r-F}
\sup_{y \in \mathbb R} \left\{e^{-\frac{3|y|}4}\big|r(s)-F(s,y)\big|\right\} \lesssim s^{-2}  
\end{equation}
and 
\begin{equation} \label{e2:r-F}
\left| \left(\partial_y \big( 5Q^4(r-F)\big),Q\right)-c_0\theta \left(\int Q\right) \lambda^{\frac32}\sigma^{-\theta-1} \right| \lesssim s^{-3}  .
\end{equation}
\end{lemma}

\begin{proof} 
Estimate \eqref{e:r} follows from \eqref{decay:q_0.1} and then \eqref{BS:param}.

Next, we compute $r_s$:
\begin{equation*}
r_s = \frac 12 \frac{\lambda_s}{\lambda} r +\lambda^{\frac 12} \sigma_s \partial_y f(s,\sigma)
+\lambda^{\frac 12}\partial_s f(s,\sigma).\end{equation*}
Note that by \eqref{e:r} and then \eqref{BS:param}
\begin{equation*}
\left| \frac 12 \frac{\lambda_s}{\lambda} r - \frac 12 c_0\frac{\lambda_s}{\lambda^{\frac 12}} \sigma^{-\theta}\right|
\lesssim \left|\frac{\lambda_s}{\lambda}\right| s^{-3}\lesssim |\vec m|s^{-3} + s^{-4}.
\end{equation*}
By \eqref{decay:q_0.1} for $k=1$, we have $|\partial_y f(s,\sigma)+c_0\theta\sigma^{-\theta-1}|\lesssim \sigma^{-5\theta+1}$
and so
\begin{equation*}
\left|\lambda^{\frac 12} \sigma_s \partial_y f(s,\sigma)+c_0\theta\lambda^{\frac 12}\sigma_s \sigma^{-\theta-1} \right|\lesssim 
\lambda^{\frac 12}|\sigma_s |\sigma^{-5\theta+1}
\lesssim (|\vec m|+1)\lambda^{\frac 32} \sigma^{-5\theta+1}
\lesssim (|\vec m|+1) s^{-4}.
\end{equation*}
Last, from \eqref{decay:q_0.2} and $ds = \lambda^{-3} dt$, we have
\begin{equation*}
|\lambda^{\frac 12}\partial_s f(s,\sigma)|
=\lambda^{\frac 72} |\partial_t f(s,\sigma)|
\lesssim \lambda^{\frac 72} \sigma^{-\theta-3}\lesssim s^{-4}.
\end{equation*}
We deduce the proof of \eqref{e:dr} gathering the above estimates.

We recall that $F(s,y)=\lambda^{\frac12}(s)f(s,\lambda(s)y+\sigma(s))$ and $\partial_yF(s,y)=\lambda^{\frac32}(s)\partial_yf(s,\lambda(s)y+\sigma(s))$. Thus, splitting the integration domain into the two cases $\lambda(s)y>-\frac14\sigma(s)\iff \lambda(s)y+\sigma(s)>\frac34 \sigma(s)$ and $\lambda(s)y<-\frac14\sigma(s) \implies y<-c s$, and then using \eqref{def:f_0}, \eqref{decay:q_0.1} (with $\kappa_0=\frac12$), we deduce that, for $s$ large enough,
\begin{align*} 
e^{-\frac{|y|}{10}}\big|F(s,y)\big| &\lesssim \lambda(s)^{\frac12}\left( e^{-\frac{|y|}{10}}\|f(s,\cdot)\|_{L^{\infty}(x>\frac34 \sigma(s))}+e^{-\frac{\sigma(s)}{40\lambda(s)}}\|f(s,\cdot)\|_{L^{\infty}}\right)  \\
 & \lesssim \lambda(s)^{\frac12} \left( \sigma(s)^{-\theta}+e^{-cs} \right),
\end{align*}
and respectively,
\begin{align*} 
e^{-\frac{|y|}{10}}\big|\partial_yF(s,y)\big| &\lesssim \lambda(s)^{\frac32}\left( e^{-\frac{|y|}{10}}\|\partial_yf(s,\cdot)\|_{L^{\infty}(x>\frac34 \sigma(s))}+e^{-\frac{\sigma(s)}{40\lambda(s)}}\|\partial_yf(s,\cdot)\|_{L^{\infty}}\right)\nonumber \\
 &   \lesssim \lambda(s)^{\frac32} \left( \sigma(s)^{-\theta-1}+e^{-cs} \right),
\end{align*}
which, together with \eqref{BS:param}, concludes the proof of \eqref{e:F}. Note that in the case where $y>-2|b|^{-\gamma}$, we get from \eqref{BS:param} and the choice $\gamma<1$ that $\lambda(s)y+\sigma(s)>\frac34 \sigma(s)$, so that \eqref{e:FLinfty} follows from \eqref{def:f_0} and \eqref{decay:q_0.1}.

From the definition of $F$ and $r$, we have
\begin{equation*}r(s)-F(s,y)=F(s,0)-F(s,y)=\lambda^{\frac12}(s)\big(f(s,\sigma(s))-f(s,\lambda(s)y+\sigma(s)\big)  .\end{equation*}
Thus, it follows applying the mean value theorem, splitting into the two cases $\lambda(s)y>-\frac14\sigma(s)$ and $\lambda(s)y<-\frac14\sigma(s)$ as above, and then using \eqref{def:f_0} and \eqref{decay:q_0.1} that, for $s$ large enough,
\begin{align*} 
e^{-\frac{3|y|}4}\big|r(s)-F(s,y)\big| &\lesssim e^{-\frac{|y|}2} |y| \lambda(s)^{\frac32}\left( e^{-\frac{|y|}4}\|\partial_yf(s,\cdot)\|_{L^{\infty}(x>\frac34 \sigma(s))}+e^{-\frac{\sigma(s)}{16\lambda(s)}}\|\partial_yf(s,\cdot)\|_{L^{\infty}}\right) \\
& \lesssim e^{-\frac{|y|}4} \lambda(s)^{\frac32} \left( \sigma(s)^{-\theta-1}+e^{-cs} \right),
\end{align*}
which implies \eqref{e:r-F} by using \eqref{BS:param}.

Finally, by using the identities
\begin{equation*} \left(\partial_y\big( 5Q^4(r-F)\big),Q\right)=-5\int Q^4Q'(r-F)=-\int Q^5 \partial_yF(s,\cdot)=-\lambda^{\frac32}(s)\int Q^5\partial_yf(s,\lambda\cdot+\sigma) \end{equation*}
and $ \int Q^5=\int Q$, we deduce arguing as in \eqref{e:r-F} that
\begin{align*} 
&\left| \left(\partial_y \big( 5Q^4(r-F)\big),Q\right)-c_0\theta \left(\int Q\right) \lambda^{\frac32}\sigma^{-\theta-1} \right| \\ &\lesssim 
\lambda^{\frac32} \int Q^5 \left( \left| f_0'(\lambda y+\sigma)-f_0'(\sigma) \right|+\left| f_0'(\lambda y+\sigma)-\partial_y f(s,\lambda y+\sigma) \right|\right)
 \lesssim \lambda^{\frac52}\left( \sigma(s)^{-\theta-2}+e^{-cs} \right). 
\end{align*}
which, together with \eqref{BS:param}, concludes the proof of \eqref{e2:r-F}.
\end{proof}

In the next lemma, we derive an estimate for the mass and the energy of $W+F$.
\begin{lemma}[Mass and energy of $W+F$]
Under the assumptions \eqref{BS:param}, it holds for $s$ large enough
\begin{equation} \label{mass:W}
\left| \int (W+F)^2-\left(\int Q^2+2b\int PQ+\frac12r\int Q \right) \right| \lesssim s^{-(2-\gamma)}+x_0^{-(2\theta-1)}
\end{equation}
and
\begin{equation} \label{ener:W}
\left|\lambda^{-2}E(W+F)+c_1\left( \frac{b}{\lambda^2}+\frac{4c_0}{\int Q}\lambda^{-\frac32}\sigma^{-\theta}\right) \right| \lesssim \lambda^{-2}s^{-2}+x_0^{-(2\theta+1)}  ,
\end{equation}
where $c_1=\frac1{16}\left(\int Q \right)^2$.
\end{lemma}

\begin{proof} 
Observe by using the decomposition in \eqref{def:Qb} that
\begin{equation*} 
\int (W+F)^2=\int Q_b^2+2r\int Q_bR+r^2\int R^2+2\int Q_bF+2r\int RF+\int F^2  , 
\end{equation*}
From the definition of $F$ and the $L^2$ conservation for \eqref{eq:q_0}, we have that 
\begin{equation} \label{mass:F}
\int F^2 =\int f^2=\int f_0^2=c_0x_0^{-(2\theta-1)}  .
\end{equation}
Moreover, we get from \eqref{e:FLinfty} that 
\begin{equation*} 
\left| \int P_b F \right| \lesssim \int_{-2|b|^{\gamma}<y<0} |F(y)|+\int_{y>0} |F(y)|e^{-\frac34{|y|}} \lesssim s^{-1+\gamma}  .
\end{equation*}
Hence, it follows from \eqref{def:Qb} and \eqref{e:r-F} that 
\begin{equation*} 
2\int Q_bF=2r\int Q+2\int Q(F-r)+2b\int P_bF=2\int Qr +\mathcal{O}(s^{-2+\gamma})  ,
\end{equation*}
which together with \eqref{def:R}, \eqref{est:massQb}, \eqref{BS:param} and \eqref{e:r} imply \eqref{mass:W}.

Now, we compute the energy of $W$: 
\begin{align*}
E(W+F)&=E(Q_b)+E(F)+\int Q_b'\partial_y(rR+F)+r\int R'\partial_yF+\frac12 r^2 \int (R')^2 \\ 
&\quad-\int Q_b^5(rR+F) -\frac16 \int \left( (Q_b+rR+F)^6-Q_b^6-F^6-6Q_b^5(rR+F) \right) .
\end{align*}
Moreover, it follows from the definition of $F$, the conservation of the energy for \eqref{eq:q_0} and \eqref{est:f_0} that 
\begin{equation*}
E(F)=\lambda^2 E(f)=\lambda^2 E(f_0) \sim \lambda^2x_0^{-(2\theta+1)}  .
\end{equation*}
Thus, we deduce then from the definition of $Q_b$ in \eqref{def:Qb}, \eqref{BS:param}, \eqref{e:r}, \eqref{e:F}, \eqref{e:FLinfty}, \eqref{e:r-F} and then \eqref{est:energyQb}, \eqref{prop:P}, \eqref{eq:Q} and \eqref{def:R} that 
\begin{align*} 
E(W+F)&=E(Q_b)-\int (Q''+Q^5)(rR+F)+\mathcal{O}\left(s^{-2}+\lambda^2x_0^{-(2\theta-1)}\right)\\
&=-\frac{1}{16}\left(\int Q \right)^2b-\frac14 \left(\int Q\right) r+\mathcal{O}\left(s^{-2}+\lambda^2x_0^{-(2\theta-1)}\right).
\end{align*}
This last estimate combined with \eqref{e:r} implies \eqref{ener:W}.
\end{proof}

We compute $\mathcal E(W)$ in the next lemma.
\begin{lemma} \label{le:3.6}
Under the assumptions \eqref{BS:param}, it holds 
\begin{equation} \label{def:EW}
\mathcal E(W) = 
- \vec m\cdot \vec \MM Q + \cR,\quad 
\vec m =\begin{pmatrix} \frac{\lambda_s}{\lambda}+b \\ \frac{\sigma_s}{\lambda}-1 \end{pmatrix},\quad
\vec \MM =\begin{pmatrix} \Lambda \\ \partial_y \end{pmatrix},
\end{equation}
where, for $s$ large enough, 
\begin{align} 
&\left|\big(\mathcal{R},\phi\big)\right| \lesssim |\vec{m}|s^{-1}+|b_s| +s^{-2}, \quad \forall \, \phi \in \mathcal{Y}  , \label{est:R:Y} \\
&\|\cR\|_{L^2_B}+\|\partial_y\cR\|_{L^2_B} \lesssim |\vec{m}|s^{-1}+|b_s|+s^{-2}  , \label{est:R:L2B} 
\end{align}
where the norm $L^2_B$ is defined in \eqref{def:L2B} ,
and
\begin{equation} \label{est:RQ}
\left| (\cR,Q)- c_1 \left[ b_s+2b^2-\frac{4c_0}{\int Q}\lambda^{\frac 12}\sigma^{-\theta} \left(\frac 32 \frac {\lambda_s}{\lambda}+\theta \frac{\sigma_s}{\sigma} \right)\right]
\right|\lesssim |\vec m| s^{-1}+s^{-3}  ,
\end{equation}
with $c_1=\frac1{16}\left(\int Q \right)^2$.
\end{lemma}
\begin{proof}
We compute $\mathcal E(W)$ from the definition of $W$:
\begin{align*}
\mathcal E(W) &= b_s \frac{\partial Q_b}{\partial b} + r_s R 
-\frac{\lambda_s}{\lambda} \Lambda W -\left(\frac {\sigma_s}{\lambda}-1\right) \partial_y W\\
&\quad + ( Q_b'' -Q_b+Q_b^5)'
+r (R''-R)' + \partial_y((Q_b+rR+F)^5-Q_b^5-F^5).
\end{align*}
Using the definition of $Q_b$ and $\Psi_b$, the definitions of $\vec m$ and $\vec \MM$ and the equation of $R$, 
we rewrite the previous identity as follows
\begin{align}
\mathcal E(W) &= -\vec m \cdot \vec \MM Q-b\vec m \cdot \vec \MM P_b-r\vec m \cdot \vec \MM R
+ (b_s+2b^2) \frac{\partial Q_b}{\partial b} + r_s R +br\Lambda R-\Psi_b
\nonumber \\
&\quad + \partial_y\big((Q_b+rR+F)^5-Q_b^5-F^5-5 rQ^4 (R+1)\big) \nonumber \\ & =-\vec m \cdot \vec \MM Q + \cR, \label{decomp:EW}
\end{align}
where
\begin{align*}
 \cR&=\cR_1+\partial_y\cR_2+\cR_3-\Psi_b,
\\
\cR_1&= r_s R + br \Lambda R 
+5\partial_y\left[Q_b^4(rR+F)-Q^4 (rR+r)\right],
\\
\cR_2&=(Q_b+rR+F)^5-Q_b^5-F^5-5Q_b^4(rR+F),
\end{align*}
and 
\begin{equation*}
\cR_3=(b_s+2b^2) \frac {\partial Q_b}{\partial b}-b\, \vec m \cdot \vec \MM P_b-r\, \vec m \cdot \vec \MM R  .
\end{equation*}

\noindent \textit{Estimates for $\cR_1$.}
First we deduce from Lemma \ref{lemma:est:rF} that 
\begin{equation} \label{est:R1:L2}
\left| \left( \mathcal{R}_1,\phi\right) \right|+\|\cR_1\|_{L^2_B}+\|\partial_y\cR_1\|_{L^2_B} \lesssim |\vec{m}|s^{-1}+s^{-2}, \quad \forall \phi \in \mathcal{Y}  .
\end{equation}

Now, we estimate $(\cR_1,Q)$.
First, by \eqref{e:dr} and $(R,Q)=-\frac 34 \int Q$, 
\begin{equation*}
\left| (r_s R,Q) +\frac 34 c_0 \left(\int Q\right) \lambda^{\frac 12}\sigma^{-\theta} \left( \frac 12 \frac{\lambda_s}{\lambda} 
-\theta\frac{\sigma_s}{\sigma} \right)\right|
\lesssim |\vec m| s^{-3}+s^{-4}.
\end{equation*}
Second,
$(br\Lambda R,Q)=-br(R,\Lambda Q)=:I$.

Third, we write
\begin{align*}
5&\left(\partial_y\left[Q_b^4(rR+F)-Q^4 (rR+r)\right],Q\right)
\\&=-5 \left( Q^4(F-r),Q'\right)-5\left((Q_b^4-Q^4)(rR+F),Q'\right) \\
&=-5 \left( Q^4(F-r),Q'\right)-20rb\left(Q^3P(R+1),Q'\right)-20b\left(Q^3P(F-r),Q'\right) \\ & 
\quad -5\left((Q_b^4-Q^4-4bQ^3P)(rR+F),Q'\right) \\ & =: II_1+II_2+II_3+II_4  .
\end{align*}
Moreover, by using the identity (2.52) in \cite{MaMeRa3}
\begin{equation*}-\left(R,\Lambda Q \right)-20 \left( Q^3(R+1)P,Q'\right)=0,\end{equation*}
 we get that $I+II_2=0$. To deal with $II_1$, we deduce from \eqref{e2:r-F}, $|\sigma_s-\lambda| \le \lambda |\vec{m}|$ and \eqref{BS:param} that 
\begin{equation*}
\left| II_1+c_0\theta \left(\int Q\right) \lambda^{\frac12}\sigma^{-\theta-1}\sigma_s \right| \lesssim |\vec{m}|s^{-2}+s^{-3}  .
\end{equation*}
Next, it is clear from \eqref{BS:param} and \eqref{e:r-F} that $\big| II_3 \big| \lesssim s^{-3}$. 

Finally we also claim that $\big| II_4 \big| \lesssim s^{-3}$. Indeed, a direct computation gives 
\begin{equation*}
II_4=-5\left(\big(4Q^3bP(\chi_b-1)+6Q^2b^2P_b^2+4Qb^3P_b^3+b^4P_b^4\big)(rR+F),Q'\right)
\end{equation*}
which implies the claim by using the definition of $\chi_b$ in \eqref{def:Pb}, \eqref{BS:param}, \eqref{e:r} and \eqref{e:F}.

Therefore, we deduce gathering those estimates that 
\begin{equation} \label{est:R1}
\left| (\cR_1,Q) +\frac{c_0}4 \left(\int Q\right) \lambda^{\frac 12}\sigma^{-\theta} \left( \frac 32 \frac{\lambda_s}{\lambda} 
+\theta\frac{\sigma_s}{\sigma} \right)\right| \lesssim |\vec{m}|s^{-2}+s^{-3} .
\end{equation}

\smallskip

\noindent \textit{Estimates for $\partial_y\cR_2$.}
We claim that 
\begin{equation} \label{est:R2}
\big|\big( \partial_y\mathcal{R}_2, \phi\big)\big|+\big\| \partial_y\mathcal{R}_2\big\|_{L^2_B}+\big\| \partial_y^2\mathcal{R}_2\big\|_{L^2_B} \lesssim s^{-2}, \ \forall \, \phi \in \mathcal{Y}, \quad \big|\big( \partial_y\mathcal{R}_2, Q\big)\big| \lesssim s^{-3} .
\end{equation}
We first develop $\mathcal{R}_2$:
\begin{align*} 
\mathcal{R}_2&=10Q^3(rR+F)^2+10(Q_b^3-Q^3)(rR+F)^2+10Q_b^2(rR+F)^3\\
&\quad +5Q_b(rR+F)^4+ (rR+F)^5-F^5  .
\end{align*}
We deduce easily that $\big|\big( \partial_y\mathcal{R}_2, \phi\big)\big|+\big\| \partial_y\mathcal{R}_2\big\|_{L^2_B}+\big\| \partial_y^2\mathcal{R}_2\big\|_{L^2_B} \lesssim s^{-2}$, $\forall \, \phi \in \mathcal{Y}$, arguing as in the proof of Lemma \ref{lemma:est:rF}. 

Now, we prove the second estimate in \eqref{est:R2}.
On the one hand, we get easily from the definition of $Q_b$ in \eqref{def:Qb}, \eqref{e:r} and \eqref{e:F} that 
\begin{equation*} 
\big|\big( \partial_y\mathcal{R}_2, Q\big)+10(Q^3(rR+F)^2,Q')\big| \lesssim s^{-3} .
\end{equation*}
On the other hand, we see by integration by parts that 
\begin{align*}
10(Q^3(rR+F)^2,Q')&=-5r^2(Q^4(R+1)\partial_yR)
\\&\quad -5r(Q^4,(F-r) \partial_yR)-5r(Q^4,R\partial_yF)-5(Q^4,F\partial_yF)  .
\end{align*}
Observe from the definition of $R$ in \eqref{def:R} that $5Q^4(R+1)=-\partial_y^2R+R$, so that the first term on the right-hand side of the above identity cancels out by symmetry.
Hence, it follows from \eqref{e:r}, \eqref{e:F} and \eqref{e:r-F} that 
\begin{equation*} 
\big|(Q^3(rR+F)^2,Q')\big|\lesssim s^{-3}  ,
\end{equation*}
which yields the second estimate in \eqref{est:R2}.

\smallskip

\noindent \textit{Estimates for $\mathcal{R}_3$.} First, we deduce from \eqref{eq:dQb:db} and \eqref{BS:param} that 
\begin{equation*} 
|b_s+2b^2| \left| \left( \frac {\partial Q_b}{\partial b}, \phi\right) \right| \lesssim |b_s+2b^2| \lesssim |b_s|+s^{-2}, \quad \forall \phi \in \mathcal{Y}  ,
\end{equation*}
\begin{equation*} 
|b_s+2b^2| \left(\left\| \frac {\partial Q_b}{\partial b}\right\|_{L^2_B}+\left\| \partial_y\left(\frac {\partial Q_b}{\partial b}\right)\right\|_{L^2_B}\right) \lesssim |b_s+2b^2| \lesssim |b_s|+s^{-2}  .
\end{equation*}
Arguing similarly, we get from \eqref{def:Pb}, \eqref{BS:param} and \eqref{e:r} that 
\begin{equation*} 
\left|b \left(\vec m \cdot \vec \MM P_b, \phi \right) \right|+\left|r \left(\vec m \cdot \vec \MM R, \phi \right) \right| \lesssim s^{-1}|\vec{m}|\quad \forall \phi \in \mathcal{Y} ,
\end{equation*}
and
\begin{equation*} 
|b| \left\|\vec m \cdot \partial_y^k\vec \MM P_b\right\|_{L^2_B} +|r| \left\|\vec m \cdot \partial_y^k\vec \MM R\right\|_{L^2_B} \lesssim s^{-1}|\vec{m}|, \quad k=0,1  .
\end{equation*}
It follows combining those estimates that 
\begin{equation} \label{est:R3}
(\mathcal{R}_3,\phi)+\|\mathcal{R}_3 \|_{L^2_B}+\|\partial_y\mathcal{R}_3 \|_{L^2_B} \lesssim s^{-1}|\vec{m}|+|b_s|+s^{-2}, \quad \forall \phi \in \mathcal{Y}  .
\end{equation}

By using the identity \eqref{eq:dQb:db}, we get that 
\begin{equation*}
\left( \frac {\partial Q_b}{\partial b},Q\right)=\left(Q,P\right)+\left((\chi_b-1)P,Q\right)+
\gamma\left(yP\chi_b',Q\right).
\end{equation*}
Thus, it follows from the definition of $\chi_b$ in \eqref{def:Pb}, the properties of $Q$, \eqref{prop:P} and \eqref{BS:param}, we deduce that 
\begin{equation} \label{est:dQb:db}
\left|(b_s+2b^2)\left( \frac {\partial Q_b}{\partial b},Q\right)-\frac1{16}\left(\int Q \right)^2(b_s+2b^2)\right| \lesssim e^{-\frac{s^{\gamma}}2}  .
\end{equation}

\smallskip

Therefore, we conclude the proof of \eqref{est:R:Y} gathering \eqref{est:PsibY}, \eqref{est:R1:L2}, \eqref{est:R2} and \eqref{est:R3}, the proof of \eqref{est:R:L2B} gathering \eqref{est:PsibL2}, \eqref{est:R1:L2}, \eqref{est:R2} and \eqref{est:R3}, and the proof of \eqref{est:RQ} gathering \eqref{est:PsibQ}, \eqref{BS:param}, \eqref{est:R1}, \eqref{est:R2}, \eqref{est:R3} and \eqref{est:dQb:db}.
\end{proof}

We define
\begin{equation} \label{def:gh}
g(s)=\frac{b(s)}{\lambda^2(s)}+\frac{4}{\int Q}c_0 \lambda^{-\frac 32}(s)\sigma^{-\theta}(s) \quad \text{and} \quad
h(s)=\lambda^{\frac 12}(s)-\frac{1}{\int Q} \frac{2c_0}{1-\theta} \sigma^{-\theta+1}(s)  .
\end{equation}
\begin{lemma} \label{lemma:gh}
It holds
\begin{align}
\left| c_1 \lambda^{2} g_s - (\cR,Q)\right|
&\lesssim |\vec m| s^{-1}+s^{-3}  ; \label{est:g} \\
\left| \lambda^{-\frac12}h_s +\frac 12 \lambda^2g\right|&\lesssim |\vec m|  . \label{est:h}
\end{align}
\end{lemma}
\begin{proof}
First, observe by a direct computation that 
\begin{equation*} 
\lambda^2g_s=b_s+2b^2-2\left(\frac{\lambda_s}{\lambda}+b\right)b-\frac{4c_0}{\int Q}\lambda^{\frac 12}\sigma^{-\theta} \left(\frac 32 \frac {\lambda_s}{\lambda}+\theta \frac{\sigma_s}{\sigma} \right)
\end{equation*}
Thus, estimate \eqref{est:g} follows from \eqref{BS:param}, \eqref{def:EW}, and \eqref{est:RQ}.

Another direct computation yields 
\begin{equation*} 
\lambda^{-\frac12}h_s +\frac 12 \lambda^2g=\frac12 \left( b+\frac{\lambda_s}{\lambda}\right) +\frac{2c_0}{\int Q} \left( 1-\frac{\sigma_s}{\lambda}\right)\lambda^{\frac12} \sigma^{-\theta}  .
\end{equation*}
Hence, we deduce from \eqref{BS:param} and \eqref{def:EW} that 
\begin{equation*} 
\left| \lambda^{-\frac12}h_s +\frac 12 \lambda^2g\right| \lesssim |\vec{m}| (1+s^{-1}) ,
\end{equation*}
which implies \eqref{est:h} by choosing $s$ large enough.
\end{proof}

\subsection{Modulation and parameter estimates}

Let $U$ be a solution of \eqref{gkdv} defined on a time interval $\mathcal{I}\subset[t_0,+\infty)$ and set 
\begin{equation} \label{def:v}
v(t,x)=U(t,x)-f(t,x)  ,
\end{equation}
where $f$ is defined in Section~\ref{sec:AS}.
We assume that there exists $(\lambda_\sharp(t),\sigma_\sharp(t)) \in (0,+\infty)^2$ and $\zz \in \mathcal{C}\big(\mathcal{I} : L^2(\mathbb R)\big)$ such that 
\begin{equation} \label{tube.1} 
v(t,x)=\lambda_\sharp^{-\frac12}(t)(Q+r_\sharp R+ \zz)(t,y), \quad y=\frac{x-\sigma_\sharp(t)}{\lambda_\sharp(t)},
\quad 
r_\sharp(t)=\lambda_\sharp^{\frac 12}(t) f(t,\sigma_\sharp(t)),
\end{equation}
with 
\begin{equation} \label{tube.2} 
\|\zz(t)\|_{L^2} + \lambda_\sharp (t) \sigma_\sharp^{-1}(t)+ \sigma_\sharp^{-1}(t)\le \astar  ,
\end{equation}
for all $t \in \mathcal{I}$ and 
where $\astar$ is small positive universal constant. 
For future use, remark that \eqref{tube.2} implies (using $\frac 12<\theta<1$)
\begin{equation}\label{tube.3}
\frac{\lambda_\sharp^{\frac 32}}{\sigma_\sharp^{\theta+1}}
= \left( \frac{\lambda_\sharp}{\sigma_\sharp}\right)^{\frac 32} \sigma_\sharp^{\frac 32-\theta-1}
\lesssim (\astar)^{\frac 32}\quad\mbox{and}\quad
\frac{\lambda_\sharp^{\frac 12}}{\sigma_\sharp^{\theta}}
= \left( \frac{\lambda_\sharp}{\sigma_\sharp}\right)^{\frac 12} \sigma_\sharp^{\frac 12-\theta}
\lesssim (\astar)^{\frac 12}.
\end{equation}
We collect in the next lemma the standard preliminary estimates on this decomposition related to the choice of suitable orthogonality conditions for the remainder term.

\begin{lemma} \label{lemma:decomp}
Assume \eqref{tube.1}-\eqref{tube.2} for $\astar>0$ small enough.
Then, there exist unique continuous functions $(\lambda,\sigma,b):\mathcal{I} \to (0,\infty)\times \mathbb R^2$ such that
\begin{equation} \label{decomp:v}
\lambda^{\frac 12}(t) v(t,\lambda(t) y +\sigma(t)) = Q_{b(t)}(y) + r(t) R(y)+\varepsilon(t,y)=W(t,y)+\varepsilon(t,y) ,
\end{equation}
\begin{equation}\label{lambdacheck}
\left|\frac{\lambda(t)}{\lambda_\sharp(t)}-1\right|
+|b(t)|+\left|\frac{\sigma(t)-\sigma_\sharp(t)}{\lambda_\sharp(t)}\right|
\lesssim \|\zz(t)\|_{L_\loc^2},
\end{equation}
and where $\varepsilon$ satisfies, for all $t\in \mathcal{I}$,
\begin{equation}\label{ortho}
(\varepsilon(t),\Lambda Q)=(\varepsilon(t),y\Lambda Q)=(\varepsilon(t),Q)=0,
\end{equation}
\begin{equation}\label{1est:eps} 
\|\varepsilon(t)\|_{L_\loc^2} \lesssim \|\zz(t)\|_{L_\loc^2},\quad 
\|\varepsilon(t)\|_{L^2} \lesssim \|\zz(t)\|_{L^2}^{\frac 58} .
\end{equation}
\end{lemma}

\begin{proof} First, the decomposition is performed for fixed $t \in \mathcal{I}$. Let us define the map
\begin{equation*} 
\Theta: ( \lamc,\sigc,\bc,v_1 ) \in (0,+\infty)\times \mathbb R^2 \times L^2 \mapsto \left( (\varc,y\Lambda Q),(\varc,\Lambda Q),(\varc, Q)\right) \in \mathbb R^3  ,
\end{equation*} 
where
\begin{align*} 
\varc(y)=\varc_{ ( \lamc,\sigc,\bc,v_1)}(y)&
=\lamc^{\frac12}v_1\left(t,\lamc y+\sigc\right) -Q_{\bc}(y)
-\lamc^{\frac12}\lambda_\sharp^{\frac12}(t)f\left(t,\sigma_\sharp(t)+\lambda_\sharp(t)\sigc\right)R(y)
 . \end{align*}
Let $v_\sharp=Q+r_\sharp R$
and $\theta_0:=(1,0,0,v_\sharp)$.
We see that $\varc_{\theta_0}=0$, so that $\Theta(\theta_0)=0$. Moreover, it follows from explicit computations that
\begin{align*}
&\partial_{\lamc} \varc_{|_{\theta_0}}=\Lambda Q+r_\sharp yR',\\
&\partial_{\sigc} \varc_{|_{\theta_0}}=Q'+r_\sharp R'-\lambda_\sharp^{\frac 32} \partial_x f(t,\sigma_\sharp) R,\\
&\partial_{\bc} \varc_{|_{\theta_0}}=-P .
\end{align*}
In particular, by parity properties, the identity $(Q',y\Lambda Q)=\frac12(Q,\Lambda Q)+(yQ',\Lambda Q)=\| \Lambda Q\|_{L^2}^2$,
and then (from \eqref{decay:q_0.1} and \eqref{tube.3})
\begin{equation*}
|\lambda_\sharp^{\frac 32} \partial_x f(t,\sigma_\sharp)|\lesssim \lambda_\sharp^{\frac 32} \sigma_\sharp^{-\theta-1}
\lesssim (\astar)^{\frac 32},\quad
|r_\sharp|\lesssim \lambda_\sharp^{\frac 12}\sigma_\sharp^{-\theta}
\lesssim (\astar)^{\frac 12},
\end{equation*}
we obtain
\begin{align*} 
\frac{\partial \Theta}{\partial (\lamc,\sigc,\bc)}(\theta_0)
&=
\begin{pmatrix}
(\Lambda Q,y\Lambda Q) & (Q',y\Lambda Q)+r_\sharp (R',y\Lambda Q) & -(P,y\Lambda Q) \\
\|\Lambda Q\|_{L^2}^2 +r_\sharp (yR',\Lambda Q)& (Q',\Lambda Q)-\lambda_\sharp^{\frac 32} \partial_x f(t,\sigma_\sharp)( R,\Lambda Q) & -(P,\Lambda Q)\\
(\Lambda Q,Q) +r_\sharp (yR',Q)& (Q',Q)-\lambda_\sharp^{\frac 32} \partial_x f(t,\sigma_\sharp) (R,Q) & -(P,Q) 
\end{pmatrix}
\\
&=
\begin{pmatrix}
0 & \|\Lambda Q\|_{L^2}^2 & -(P,y\Lambda Q) \\
\|\Lambda Q\|_{L^2}^2 & 0 & -(P,\Lambda Q)\\
0 & 0 & -\frac1{16}\|Q\|_{L^1}^2
\end{pmatrix} + O\left((\astar)^{\frac 12}\right).
\end{align*}
Thus, the Jacobian satisfies for $\astar$ small enough
\begin{equation*} 
\det \left(\frac{\partial \Theta}{\partial (\lamc,\sigc,\bc)}(\theta_0)\right)=
\frac1{16}\|\Lambda Q\|_{L^2}^4 \|Q\|_{L^1}^2+O\left((\astar)^{\frac 12}\right) >\frac1{17}\|\Lambda Q\|_{L^2}^4 \|Q\|_{L^1}>0  . 
\end{equation*}

Therefore, possibly taking a smaller constant $\astar>0$, it follows from the implicit function theorem that for 
any $v_1=Q+r_\sharp R+\zz$ where $z$ satisfies $\|\zz\|_{L^2} <\astar$, there exist unique $(\lamc,\sigc,\bc)=(\lamc,\sigc,\bc)(v_1)$ such that $\Theta(\lamc,\sigc,\bc,v_1)=0$, where $\lamc$ is close to $1$ and $\sigc,\bc$ are small.
Moreover, the map $v_1\mapsto (\lamc,\sigc,\bc)(v_1)$ is continuous.

Now, for a function $v$ satisfying \eqref{tube.1}, we consider
\begin{equation*}
v_1(t,y)=\lambda_\sharp^{\frac 12}(t)v(t,\lambda_\sharp(t) y + \sigma_\sharp(t))=Q(y)+r_\sharp(t) R(y)+ \zz(t,y),
\end{equation*}
and we define
\begin{align*}
&\lambda(t)=\lambda_\sharp(t)\lamc(v_1(t)), \quad \sigma(t)=\lambda_\sharp(t)\sigc(v_1(t))+\sigma_\sharp(t), \quad b(t)=\bc(v_1(t)),
\\
&\varepsilon(t,y)=\varc_{ \left( \lamc(v_1(t)),\sigc(v_1(t)),\bc(v_1(t)),v_1(t)\right)}(y)  .
\end{align*}
In particular, $\varepsilon$ satisfies the orthogonality conditions \eqref{ortho}
and it holds
\begin{equation*} 
\lambda^{\frac 12}(t) v(t,\lambda(t) y +\sigma(t)) 
=Q_{b(t)}(y)+\lambda^{\frac12}(t)f(t,\sigma(t))R(y)+\varepsilon(t,y)  ,
\end{equation*}
which is the desired decomposition for $v$.

Now, we prove \eqref{lambdacheck} and \eqref{1est:eps}.
We omit mentioning the time dependency for simplicity.
Note from the above, the identity
\begin{equation}\label{id:eps}\begin{aligned}
 \varepsilon(y)
&=\lamc^{\frac12} Q(\lamc y+\sigc)
- Q_{\bc}(y)
\\&\quad +\lamc^{\frac 12}\lambda_\sharp^{\frac 12}\left[f(\sigma_\sharp)R(\lamc y +\sigc)- f(\sigma_\sharp+\lambda_\sharp \sigc)R(y)\right]+\lamc^{\frac12}\zz(\lamc y+\sigc).
\end{aligned}\end{equation}
We will project this identity on the three orthogonality directions $y\Lambda Q$, $\Lambda Q$ and $Q$ 
of $\varepsilon$.
First, by direct computations, it holds
\begin{align*}
&(\lamc^{\frac12} Q(\lamc \cdot +\sigc)-Q, y\Lambda Q )
= \sigc\left( \|\Lambda Q\|_{L^2}^2 + O(|\lamc-1|+|\sigc|)\right),\\
&(\lamc^{\frac12} Q(\lamc \cdot +\sigc)-Q, \Lambda Q)
=(\lamc-1) \|\Lambda Q\|_{L^2}^2 + O(|\lamc-1|^2+|\sigc|^2),\\
&(\lamc^{\frac12} Q(\lamc \cdot +\sigc)-Q, Q)
=O(|\lamc-1|^2+|\sigc|^2),\quad
(P_{\bc},Q)=\frac {\bc}{16}\|Q\|_{L^1}^2+O(\bc^{10}).
\end{align*}
Second, by the triangle inequality and \eqref{decay:q_0.1} and \eqref{tube.3},
\begin{align*}
&\lambda_\sharp^{\frac 12}\left\|f(\sigma_\sharp)R(\lamc \cdot +\sigc)- f(\sigma_\sharp+\lambda_\sharp \sigc)R(\cdot)\right\|_{L^2}
\\
&\quad \lesssim \lambda_\sharp^{\frac 12} \left|f(\sigma_\sharp)- f(\sigma_\sharp+\lambda_\sharp \sigc) \right| \left\| R(\lamc \cdot +\sigc)\right\|_{L^2}
+\lambda_\sharp^{\frac 12} |f(\sigma_\sharp)| \left\|R(\lamc \cdot +\sigc)- R(\cdot)\right\|_{L^2}
\\&\quad \lesssim \lambda_\sharp^{\frac 32} \sigma_\sharp^{-\theta-1} |\sigc| + 
\lambda_\sharp^{\frac 12}\sigma_\sharp^{-\theta} \left( |\lamc-1|+|\sigc|\right)
\lesssim (\astar)^{\frac 12} \left(|\lamc-1|+|\sigc|\right).
\end{align*}
Therefore, the projections yield the following estimates
\begin{equation*}
|\lamc-1|+|\sigc|\lesssim \|\zz\|_{L^2_\loc} + |\bc|+(\astar)^{\frac 12} \left(|\lamc-1|+|\sigc|\right),\quad
|\bc|\lesssim \|\zz\|_{L^2_\loc}+(\astar)^{\frac 12} \left(|\lamc-1|+|\sigc|\right). 
\end{equation*}
Combining these estimates, for $\astar$ small enough, we obtain \eqref{lambdacheck}.
Then, \eqref{1est:eps} follows using the above estimates and \eqref{def:Pb} back into \eqref{id:eps}
(note in particular that from \eqref{def:Pb} and $\gamma=\frac 34$, $\|P_{\bc}\|_{L^2}^2 \lesssim |\bc|^{2-\gamma}
\lesssim \|\zz\|_{L^2_\loc}^{5/4}\lesssim \|\zz\|_{L^2}^{5/4}$).
\end{proof}

\begin{remark}
The $\mathcal C^1$ regularity of $t\mapsto (\lambda(t),\sigma(t),b(t))$ and the equation 
\begin{equation}\label{eq:eps} 
\partial_s \varepsilon = \partial_y \left[ -\partial_y^2\varepsilon+\varepsilon-\left((W+F+\varepsilon)^5-(W+F)^5\right)\right]
 - \mathcal E(W) +\vec m \cdot \vec\MM \varepsilon-b\Lambda \varepsilon  ,
 \end{equation}
 (where we have used the notation in \eqref{def:EV} and \eqref{def:mM})
 follow from classical arguments and computations. We refer for example to the proof of Lemma 2.7. in \cite{CoMa2}.
\end{remark}

Next, we derive some estimates for $\varepsilon$ in $H^1$ related to the conservation of mass and energy. 
\begin{lemma}[Mass and energy estimates for $\varepsilon$]\label{le:3.9}
Under the bootstrap assumptions \eqref{BS:param}, it holds
\begin{equation} \label{mass:eps}
\|\varepsilon\|_{L^2}^2 \lesssim \left| \int U_0^2-\int Q^2 \right| +s^{-1}+x_0^{-(2\theta-1)} 
\end{equation}
and 
\begin{equation} \label{ener:eps}
\lambda^{-2}\|\partial_y\varepsilon\|_{L^2}^2 \lesssim |E(u_0)|+|g(s)|+\lambda^{-2}s^{-2}+\lambda^{-2}\int Q\varepsilon^2+x_0^{-(2\theta+1)}  ,
\end{equation}
for $s$ large enough, where $g$ is defined in Lemma \ref{lemma:gh}.
\end{lemma} 

\begin{proof} By using the conservation of the $L^2$ norm for $u$, we obtain that 
\begin{equation*} 
\int U_0^2=\int U^2=\|W+F+\varepsilon\|_{L^2}^2=\int (W+F)^2+2\int (W+F)\varepsilon+\int \varepsilon^2  .
\end{equation*}
We observe by using the third orthogonality condition in \eqref{ortho} and then the Cauchy-Schwarz inequality, \eqref{mass:F} and \eqref{BS:param} that 
\begin{align*}
\left|2\int (W+F)\varepsilon\right| &\le 2|r|\left|\int R \varepsilon\right|+2|b|\left|\int P_b \varepsilon\right|+2\left|\int F \varepsilon \right| \\
& \le \frac12 \int \varepsilon^2+\mathcal{O}(s^{-2+\gamma})+\mathcal{O}(x_0^{-(2\theta-1)})  ,
\end{align*}
which combined with \eqref{mass:W} implies \eqref{mass:eps}.

We turn to the proof of \eqref{ener:eps}. By using the conservation of the energy and the scaling properties, we get that 
\begin{align*}
\lambda^2E(U_0)&=\lambda^2E(U)=E(W+F+\varepsilon)
\\ & =E(W+F)+\int \partial_y(W+F)\partial_y \varepsilon+\frac12\int (\partial_y \varepsilon)^2-\frac16 \int \big( (W+F+\varepsilon)^6-(W+F)^6\big)  .
\end{align*}
Thus, it follows by using the identity $\int (-Q''-Q^5)\varepsilon=\int Q \varepsilon=0$ that 
\begin{align*}
\lambda^2E(u_0)&=E(W+F)+b\int \partial_yP_b\partial_y \varepsilon+r\int \partial_yR\partial_y \varepsilon-\int( \partial_y^2 F+F^5) \varepsilon+\frac12\int (\partial_y \varepsilon)^2
\\ & \quad -\frac16 \int \big( (W+F+\varepsilon)^6-(W+F)^6-6Q^5\varepsilon-6F^5\varepsilon \big)  .
\end{align*}
We get from the Cauchy-Schwarz inequality, \eqref{BS:param} and \eqref{e:r} that 
\begin{equation*} 
|b|\left|\int \partial_yP_b\partial_y \varepsilon\right|+|r|\left|\int \partial_yR\partial_y \varepsilon\right| \le \frac14\int (\partial_y\varepsilon)^2+\mathcal{O}(s^{-2})  .
\end{equation*}
Now, we deal with the term $\int( \partial_y^2 F+F^5) \varepsilon$. We get from the Cauchy-Schwarz inequality and \eqref{bound:df:dx} that 
\begin{equation*} 
\left| \int \partial_y^2 F \varepsilon \right|= \left| \int \partial_y F \partial_y \varepsilon \right| \le \frac18 \| \partial_y \varepsilon\|_{L^2}^2+c \| \partial_yF\|_{L^2}^2 \le \frac18 \| \partial_y \varepsilon\|_{L^2}^2+c\lambda^2 x_0^{-(2\theta+1)}  .
\end{equation*}
Similarly, we get from H\"older's inequality, the Gagliardo-Niremberg inequality \eqref {sharpGN} and then \eqref{1est:eps} and \eqref{bound:df:dx} that
\begin{equation} \label{est:F5eps}
\left| \int F^5 \varepsilon \right| \lesssim \int \varepsilon^6+\int F^6 \lesssim \| \varepsilon \|_{L^2}^4 \| \partial_y \varepsilon\|_{L^2}^2+\| F \|_{L^2}^4 \| \partial_y F\|_{L^2}^2 \le \frac1{16}\| \partial_y \varepsilon\|_{L^2}^2+c\lambda^2 x_0^{-(2\theta+1)}  .
\end{equation}
Moreover, from interpolation, the Gagliardo-Nirenberg inequality \eqref {sharpGN} and \eqref{BS:param}, \eqref{e:r}, \eqref{e:F}, it holds
\begin{align*}
&\left| \int \big( (W+F+\varepsilon)^6-(W+F)^6-6Q^5\varepsilon-6F^5\varepsilon \big)\right| 
\\ & \quad \quad \quad \quad \lesssim \left| \int \big( (W+F)^5-6Q^5-6F^5\big)\varepsilon\right|+\int (W+F)^4\varepsilon^2+\int \varepsilon^6 
\\& \quad \quad \quad \quad \lesssim s^{-2}+\int Q^2\varepsilon^2+\int F^4\varepsilon^2+\|\varepsilon\|_{L^2}^4\int (\partial_y\varepsilon)^2  .
\end{align*}
arguing as in \eqref{est:F5eps}, we estimate the term $\int F^4\varepsilon^2$ as follows 
\begin{equation*}
 \left| \int F^4 \varepsilon^2 \right| \lesssim \int \varepsilon^6+\int F^6 \le \frac1{32}\| \partial_y \varepsilon\|_{L^2}^2+c\lambda^2 x_0^{-(2\theta+1)}  .
\end{equation*} 
We conclude the proof of \eqref{ener:eps} by combining those estimates with \eqref{ener:W} and \eqref{1est:eps}. 
\end{proof}
Next, the equation of the parameters $\lambda$, $\sigma$ and $b$ are deduced from modulation estimates.
\begin{lemma}[Modulation estimates]
Under the bootstrap assumptions \eqref{BS:param}, it holds
\begin{align} 
|\vec m|&\lesssim \| \varepsilon \|_{L^2_\loc} +s^{-2}  ,\label{mod:est:m}
\\
|b_s|&\lesssim \| \varepsilon \|_{L^2_\loc}^2 +s^{-2}  , \label{mod:est:bs}
\\
\lambda^2|g_s|&\lesssim s^{-3}+s^{-1}\| \varepsilon \|_{L^2_\loc}+\| \varepsilon \|_{L^2_\loc}^{2}  , \label{mod:est:g}
\end{align}
for $s$ large enough, where the $L^2_\loc$-norm is defined in \eqref{def:L2B}.
\end{lemma}

\begin{proof} 
First, we differentiate the first orthogonality condition in \eqref{ortho} with respect to $s$, use the equation \eqref{eq:eps}, follow the computations in the proof of Lemma 2.7 in \cite{MaMeRa1} and use the estimate \eqref{est:R:Y} to get that 
\begin{equation} \label{mod:est:m1}
\left| \left(\frac{\lambda_s}{\lambda}+b \right) +\frac{\big( \varepsilon, \mathcal{L}(\Lambda Q)'\big)}{\|\Lambda Q\|_{L^2}^2}\right| \lesssim 
|\vec{m}|\big(s^{-1}+\|\epsilon\|_{L^2_\loc}\big)+|b_s|+s^{-2}+\|\epsilon\|_{L^2_\loc}^2  .
\end{equation}
Now, we derive the second orthogonality condition in \eqref{ortho} with respect to $s$. By combining similar estimates with the identity $(Q',y\Lambda Q)=\| \Lambda Q\|_{L^2}^2$, we also get that 
\begin{equation} \label{mod:est:m2}
\left| \left(\frac{\sigma_s}{\lambda}-1 \right) +\frac{\big( \varepsilon, \mathcal{L}(y\Lambda Q)'\big)}{\|\Lambda Q\|_{L^2}^2}\right| \lesssim 
|\vec{m}|\big(s^{-1}+\|\epsilon\|_{L^2_\loc}\big)+|b_s|+s^{-2}+\|\epsilon\|_{L^2_\loc}^2  .
\end{equation}
Next, we derive the third orthogonality condition in \eqref{ortho} with respect to $s$. It follows that 
\begin{align*}
0=\big(\partial_s\varepsilon,Q \big)&=\big( \partial_y\mathcal{L}\varepsilon,Q\big)-\left(\partial_y\big((W+F+\varepsilon)^5-(W+F)^5-5Q^4\varepsilon \big), Q \right)-\big(\mathcal{E}(W),Q\big)\\ 
&\quad +\left( \frac{\lambda_s}{\lambda}+b\right)\big( \Lambda \varepsilon,Q\big)+
\left( \frac{\sigma_s}{\lambda}-1\right)\big( \partial_y \varepsilon,Q\big)-b\big( \Lambda \varepsilon,Q\big)  .
\end{align*}
We observe the cancellations $\big( \partial_y\mathcal{L}\varepsilon,Q\big)=-(\varepsilon, \mathcal{L}(Q'))=0$ and $\big( \Lambda \varepsilon,Q\big)=-\big( \varepsilon,\Lambda Q\big)=0$. We also get by using \eqref{BS:param}, \eqref{e:r} and \eqref{e:F} that for $s$ large enough 
\begin{equation*} 
\left|\left(\partial_y\big((W+F+\varepsilon)^5-(W+F)^5-5Q^4\varepsilon \big), Q \right)\right| \lesssim s^{-1}\|\varepsilon\|_{L^2_\loc}+\|\varepsilon\|_{L^2_\loc}^2  .
\end{equation*}
Moreover, we have from \eqref{def:EW} that 
\begin{equation*} 
\big(\mathcal{E}(W),Q\big)=-\left( \frac{\lambda_s}{\lambda}+b\right)\big( \Lambda Q,Q\big)-\left( \frac{\sigma_s}{\lambda}-1\right)\big( \partial_y Q,Q\big)+(\mathcal{R},Q)  .
\end{equation*}
Hence, it follows from the cancellations $(\Lambda Q,Q)=(\partial_yQ,Q)=0$ and the estimate \eqref{est:RQ} that 
\begin{equation*}
\left|\big(\mathcal{E}(W),Q\big)- c_1 \left[ b_s+2b^2-4c_0\left(\int Q \right)^{-1}\lambda^{\frac 12}\sigma^{-\theta} \left(\frac 32 \frac {\lambda_s}{\lambda}+\theta \frac{\sigma_s}{\sigma} \right)\right]\right| \lesssim |\vec{m}|s^{-1}+s^{-3}  .
\end{equation*}
We deduce combining those estimates and using \eqref{BS:param} the following rough estimate on $|b_s|$
\begin{equation*}
|b_s| \lesssim s^{-2}+|\vec{m}|\big(s^{-1}+\|\epsilon\|_{L^2_\loc}\big)+s^{-1}\|\varepsilon\|_{L^2_\loc}+\|\varepsilon\|_{L^2_\loc}^2  ,
\end{equation*}
which combined with \eqref{mod:est:m1} and \eqref{mod:est:m2} yields \eqref{mod:est:m} and \eqref{mod:est:bs} by taking $s$ large enough. 

Finally, by combining the previous estimates with \eqref{est:g}, we deduce that 
\begin{equation*} 
\lambda^2 |g_s| \lesssim |\vec{m}| \big( s^{-1}+\|\varepsilon\|_{L^2_\loc}\big)+s^{-3}+s^{-1}\|\varepsilon\|_{L^2_\loc}+\|\varepsilon\|_{L^2_\loc}^2  ,
\end{equation*}
which conclude the proof of \eqref{mod:est:g} by taking $s$ large enough thanks to \eqref{mod:est:m}.
\end{proof}

\subsection{Bootstrap estimates}

Let $\psi \in\Cinfini$ be a nondecreasing function such that
\begin{equation*}
\psi(y) = \left\{ \begin{aligned} e^{2y} \quad & \mbox{for } y < -1, \\ 1 \quad & \mbox{for } y > -\frac 12, \end{aligned} \right.
\end{equation*}
For $B > 100$ large to be chosen later, we define
\begin{equation*}
\psi_B(y) = \psi\left( \frac yB \right) \quad \text{and} \quad \varphi_B(y)=e^{\frac{y}B}  .
\end{equation*}
Note that, directly from the definitions of $\psi$ and $\ph$, we have, for all $y\in\mathbb{R}$ and
\begin{equation} \label{cut:onR}
\left\{
\begin{gathered}
\psi_B(y) +\sqrt{\psi_B(y)}\lesssim \ph_B(y), \quad |y|^4\psi_B(y) \lesssim B^4\varphi_B(y),\\
\psi_B'(y) + B^2|\psi_B'''(y)| + B^2|\ph_B'''(y)| \lesssim \ph_B'(y), \\
\varphi_B(y)+\psi_B(y)+\varphi_B^2(y)e^{-\frac{y}B}+\psi_B^2(y)e^{-\frac{y}B}+|y\psi_B'(y)|\lesssim B \ph_B'(y) ,
\end{gathered}
\right.
\end{equation}
and $\frac 12 e^{\frac{2y}B } \leq \psi_B(y) \leq 3 e^{\frac{2y}B} $, for all $y<0$.

Let $0<\rho\ll 1$ and $B>100$ to be chosen later.
In addition of \eqref{BS:param}, we will work under the following bootstrap assumptions.
\begin{equation}\label{BS:eps}
\mathcal{N}_B(\varepsilon):= \left(\int \varepsilon^ 2 \varphi_B+\int (\partial_y\varepsilon)^ 2 \psi_B\right)^{\frac12} \le |s|^{-\frac 54} 
\end{equation} 
In particular, from the definition of the $L^2_\loc$-norms in \eqref{def:L2B} and $B>100$, it holds
\begin{equation} \label{est:L2loc}
\|\varepsilon\|_{L^2_\loc}+\|\partial_y\varepsilon\|_{L^2_\loc} \lesssim \mathcal{N}_B(\varepsilon) \quad \text{and} \quad \|\varepsilon\|_{L^2_\loc}^2 \lesssim B \int \varepsilon^2 \ph_B'  .
\end{equation}
 
For future reference, we state here some consequences of the bootstrap assumptions.
\begin{lemma} Under the bootstrap assumptions \eqref{BS:param} and \eqref{BS:eps}, it holds
\begin{gather} \label{BS:m}
|\vec{m}|=\left| \frac{\lambda_s}{\lambda}+b\right|+\left| \frac{\sigma_s}{\lambda}-1\right| \lesssim \|\varepsilon \|_{L^2_\loc} +s^{-2} \lesssim s^{-\frac54}  ;
\\ \label{BS:bs}
|b_s| \lesssim \| \varepsilon \|_{L^2_\loc}^2 +s^{-2} \lesssim s^{-2}  ;
\\ \label{BS:gs}
\lambda^2|g_s| \lesssim s^{-3}+s^{-1}\| \varepsilon \|_{L^2_\loc}+\| \varepsilon \|_{L^2_\loc}^{2} \lesssim s^{-\frac94}  ;
\\\label{BS:hs}
\left| \lambda^{-\frac12}h_s +\frac 12 \lambda^2g\right| \lesssim s^{-\frac54}  .
\end{gather}
\end{lemma}
\begin{proof} 
Estimates \eqref{BS:m}-\eqref{BS:hs} follow from \eqref{BS:param}, \eqref{est:h}, \eqref{mod:est:m}, \eqref{mod:est:bs}, \eqref{mod:est:g} and \eqref{BS:eps}.
\end{proof}

\section{Energy estimates}\label{section:energy}
We work  with the notation introduced in Section \ref{section:decomp_sol}. In particular, we assume that $\varepsilon$ satisfies \eqref{decomp:v}-\eqref{eq:eps} and that $(\lambda,\sigma,b,\varepsilon)$ satisfy \eqref{BS:param} and \eqref{BS:eps} on $\mathcal{J}=[s_0,s^{\star}]$ for some $s^{\star} \ge s_0$.

We define the mixed energy-virial functional
\begin{equation} \label{def:F}
\F = \int \left[(\partial_y\varepsilon)^2 \psi_B + \varepsilon^2 \ph_B
- \frac 13 \left( (W +F+ \varepsilon)^6 - (W+F)^6 - 6(W+F)^5\varepsilon \right) \psi_B\right].
\end{equation}
Set
\begin{equation} \label{def:kappa} 
\kappa=\frac{2(2\theta-1)}{1-\theta} \quad \text{so that} \quad \lambda^{\kappa} \sim s^{4}  .
\end{equation}

\begin{proposition}
There exists $\astar>0$, $\mu_0>0$ and $B_0>100$ such that, for all $B \ge B_0$ and for all $s_0$ large enough (possibly depending on $B$), the following hold on $[s_0,s^\star]$.
\begin{enumerate}
\item Time derivative of the energy functional.
\begin{equation}\label{dsF}
\lambda^{-\kappa} \big(\lambda^{\kappa}\F\big)_s + \mu_0 \int \left[(\partial_y \varepsilon)^2+\varepsilon^2\right] \ph_B' 
\lesssim s^{-4}.
\end{equation}
\item Coercivity of $\mathcal F$.
\begin{equation}\label{coF}
 \mathcal{N}_B(\varepsilon)^2
\lesssim \F +s^{-100}  .
\end{equation}
\end{enumerate}
\end{proposition}

\begin{proof}
Let
\begin{equation} \label{def:GBeps}
G_B(\varepsilon)=
-\partial_y(\Psi_B \partial_y\varepsilon) + \varepsilon \ph_B-\psi_B\left((W+F+\varepsilon)^5-(W+F)^5\right).
\end{equation}
We compute using \eqref{eq:eps},
\begin{align*}
\lambda^{-\kappa} \big(\lambda^{\kappa}\F\big)_s 
& = 2 \int \left( \varepsilon_s-\frac{\lambda_s}{\lambda} \Lambda \varepsilon\right) G_B(\varepsilon)
+ \frac{\lambda_s}{\lambda} \left( 2\int \Lambda \varepsilon G_B(\varepsilon) +\kappa\F\right)\\
& \quad - 2\int (W+F)_s \left[(W+F+\varepsilon)^5-(W+F)^5-5(W+F)^4\varepsilon\right]\psi_B \\ 
&= 2\int \partial_y \left[ -\partial_y^2\varepsilon+\varepsilon-\left((W+F+\varepsilon)^5-(W+F)^5\right)\right]
 G_B(\varepsilon) 
\\&\quad - 2 \int \mathcal E(W)G_B(\varepsilon) + \left(\frac{\sigma_s}{\lambda}-1\right) \int \partial_y\varepsilon G_B(\varepsilon)
+ \frac{\lambda_s}{\lambda} \left( 2\int \Lambda \varepsilon G_B(\varepsilon) +\kappa \F\right)\\
& \quad - 2\int (W+F)_s \left[(W+F+\varepsilon)^5-(W+F)^5-5(W+F)^4\varepsilon\right]\psi_B\\
&=:\mathrm{f}_1+\mathrm{f}_2+\mathrm{f}_3+\mathrm{f}_4+\mathrm{f}_5.\end{align*}

\noindent \emph{Estimate of $\mathrm{f}_1$.} We claim, by choosing $\astar$ small enough, $B$ large enough and then $s$ large enough (possibly depending on $B$), that
\begin{equation} \label{est:f1}
\mathrm{f}_1+2\mu_0\int \varphi_B'\varepsilon^2 \lesssim s^{-100}  ,
\end{equation}
where $\mu_0$ is a small positive constant which will be fixed below. 

To prove \eqref{est:f1}, we compute following Step 3 of Proposition 3.1 in \cite{MaMeRa1},
\begin{align*}
\mathrm{f}_1&= -\int \left[3\psi_B'(\partial_y^2 \varepsilon)^2+
(3\ph_B'+\psi_B'-\psi_B''')(\partial_y \varepsilon)^2 + (\varphi_B'-\varphi_B''')\varepsilon^2\right]\\
&\quad -\frac13\int \left[(W+F+\varepsilon)^6-(W+F)^6 - 6(W+F+\varepsilon)^5\varepsilon\right]
(\ph_B'-\psi_B')\\
& \quad + 2\int \left[(W+F+\varepsilon)^5-(W+F)^5-5(W+F)^4\varepsilon\right] \partial_y(W+F) (\psi_B-\ph_B)\\
& \quad + 10 \int \psi_B' \partial_y \varepsilon \left[\partial_y(W+F) \left((W+F+\varepsilon)^4-(W+F)^4\right)
+(W+F+\varepsilon)^4 \partial_y \varepsilon\right]
\\
&\quad -\int \psi_B' \left[\left((-\partial_y^2\varepsilon+\varepsilon-\left((W+F+\varepsilon)^5-(W+F)^5\right)\right)^2-\left(-\partial_y^2\varepsilon+\varepsilon\right)^2\right]
\\ &=\mathrm{f}_1^<+\mathrm{f}_1^>,
\end{align*}
where $\mathrm{f}_1^{<,>}$ correspond respectively to integration on the three regions
$y<-\frac B2$ and $y>-\frac B2$.

\smallskip

\noindent \emph{Estimate of $\mathrm{f}_1^<$.}
In the region $y<-\frac B2$, we use the properties \eqref{cut:onR} and take $B$ large enough to deduce 
\begin{align*}
\mathrm{f}_1^<&+3\int_{y<-\frac{B}2} \psi_B'(\partial_y^2\varepsilon)^2+\frac12\int_{y<-\frac{B}2} \varphi_B'\big((\partial_y\varepsilon)^2+\varepsilon^2 \big) \\ 
& \lesssim \int_{y<-\frac{B}2}\varphi_B' \left( \big(W^4+F^4\big)\varepsilon^2+\varepsilon^6\right)
+B\int_{y<-\frac{B}2}\varphi_B' |\partial_y(W+F)|\left(|W+F|^3|\varepsilon|^2+|\varepsilon|^5 \right)
\\ &\quad +\int_{y<-\frac{B}2} \psi_B'|\partial_y \varepsilon| \left[ |\partial_y(W+F)|\left(|W+F|^3|\varepsilon|+\epsilon^4 \right)+|\partial_y\varepsilon|\left(|W+F|^4+\epsilon^4 \right)\right] 
\\ & \quad +\int_{y<-\frac{B}2} \psi_B' \left(-2(-\partial_y^2\varepsilon+\varepsilon) +(W+F+\varepsilon)^5-(W+F)^5\right)\left((W+F+\varepsilon)^5-(W+F)^5\right) 
\\ &=:\mathrm{f}_{1,1}^<+\mathrm{f}_{1,2}^<+\mathrm{f}_{1,3}^<+\mathrm{f}_{1,4}^<  .
\end{align*}
Observe from \eqref{est:Qb} and \eqref{e:r} that 
\begin{equation*} 
\int_{y<-\frac{B}2} \varphi_B'W^4\varepsilon^2 \lesssim \big(s^{-4}+e^{-\frac{B}2}\big) \int_{y<-\frac{B}2} \varphi_B'\varepsilon^2  . 
\end{equation*}
To deal with $\int F^4\varepsilon^2$, recall from the definition of $F$ that $F(s,y)=\lambda^{\frac12}(s)f(s,\lambda(s)y+\sigma(s))$. By splitting the integration domain into the two cases $\lambda(s)y>-\frac14\sigma(s) \iff \lambda(s)y+\sigma(s)>\frac34 \sigma(s)$ and $\lambda(s)y<-\frac14\sigma(s)$ (from \eqref{BS:param}) and using \eqref{def:f_0}, \eqref{decay:q_0.1} and \eqref{BS:param} we get that, 
\begin{equation} \label{est:Feps} 
\int_{y<-\frac{B}2} \varphi_B'|F|^j\varepsilon^2 \lesssim s^{-j} \int_{y<-\frac{B}2} \varphi_B'\varepsilon^2+B^{-1}\lambda(s)^{\frac{j}2}e^{-\frac{c}{B}s}\int \varepsilon^2 \
\end{equation}
and
\begin{equation} \label{est:Feps_deriv} 
\int_{y<-\frac{B}2} \varphi_B'|\partial_yF|^j\varepsilon^2 \lesssim s^{-2j}\int_{y<-\frac{B}2} \varphi_B'\varepsilon^2+B^{-1}\lambda(s)^{\frac{3j}2} e^{-\frac{c}{B}s}\int \varepsilon^2  ,
\end{equation}
for all $j \in \mathbb N$, $j \ge 1$. 

To control, the purely nonlinear term in $\mathrm{f}_{1,1}$, we recall the following version of the Sobolev embedding (see Lemma 6 of \cite{Mjams} and also \eqref{est:infty}): 
\begin{align*}
\left\|\varepsilon^2\sqrt{\varphi_B'}\right\|_{L^\infty(y<-\frac{B}2)}^2
&\lesssim \|\varepsilon\|_{L^2}^2\left(\int_{y<-\frac{B}2} (\partial_x\varepsilon)^2 \varphi_B'
 +\int_{y<-\frac{B}2} \varepsilon^2\frac{\big(\varphi_B''\big)^2}{\varphi_B'}\right)\\ &\lesssim \delta(\astar)\int_{y<-\frac{B}2} \varphi_B'\left(\varepsilon^2+(\partial_y \varepsilon)^2\right).
 \end{align*}
Thus, it follows that 
\begin{equation*} 
\int_{y<-\frac{B}2} \varphi_B'\varepsilon^6 \le \left\|\varepsilon^2\sqrt{\varphi_B'}\right\|_{L^\infty(y<-\frac{B}2)}^2 \left(\int \varepsilon^2 \right)\lesssim \delta(\astar)\int_{y<-\frac{B}2} \varphi_B'\left(\varepsilon^2+(\partial_y \varepsilon)^2\right).
\end{equation*}
Note also for future reference that the same proof yields
\begin{equation} \label{est:eps6}
 \int \varphi_B'\varepsilon^6 \le \left\|\varepsilon^2\sqrt{\varphi_B'}\right\|_{L^\infty}^2 \left(\int \varepsilon^2 \right)\lesssim \delta(\astar)\int \varphi_B'\left(\varepsilon^2+(\partial_y \varepsilon)^2\right).
 \end{equation}
Hence, we deduce gathering those estimates and choosing $s$ and $B$ large enough and $\astar$ small enough that
\begin{equation*} 
\mathrm{f}_{1,1}^< \le cs^{-100}+\frac18\int_{y<-\frac{B}2} \varphi_B'\left(\varepsilon^2+(\partial_y \varepsilon)^2\right).
\end{equation*}
Observe from Young's inequality that 
\begin{equation*} 
\mathrm{f}_{1,2}^< \lesssim \int_{y<-\frac{B}2} \varphi_B' |\partial_y(W+F)| \left(|W+F)^3+|\partial_y(W+F)|^3 \right)\varepsilon^2+\int_{y<-\frac{B}2} \varphi_B' \varepsilon^6  ,
\end{equation*}
so that it follows arguing as for $\mathrm{f}_{1,1}^<$ that, for $s$ and $B$ large enough and $\astar$ small enough, 
\begin{equation*} 
\mathrm{f}_{1,2}^< \le cs^{-100}+\frac1{2^4}\int_{y<-\frac{B}2} \varphi_B'\left(\varepsilon^2+(\partial_y \varepsilon)^2\right).
\end{equation*}
By using again Young's inequality, we have 
\begin{equation*} 
\mathrm{f}_{1,3}^< \lesssim \int_{y<-\frac{B}2} \psi_B' \left(|W+F|^4+|\partial_y(W+F)|^4 \right) (\varepsilon^2+(\partial_y \varepsilon)^2)+\int_{y<-\frac{B}2} \psi_B'\varepsilon^6+\int_{y<-\frac{B}2} \psi_B' \varepsilon^4(\partial_y\varepsilon)^2
\end{equation*}
The first two terms on the right-hand side of the above inequality are estimated as before. For the third one, we deduce arguing as in \eqref{Sobo:firstderiv} that 
\begin{equation*} 
\int_{y<-\frac{B}2} \psi_B' \varepsilon^4(\partial_y\varepsilon)^2 \lesssim \left\| \varepsilon \partial_y\varepsilon \sqrt{\psi_B'}\right\|_{L^{\infty}(y<-\frac{B}2)}^2 \left(\int \varepsilon^2 \right) \lesssim \delta(\astar)\int_{y<-\frac{B}2} \psi_B' \left((\partial_y\varepsilon)^2+(\partial_y^2\varepsilon)^2 \right).
\end{equation*}
Thus it follows by taking $s$ and $B$ large enough and $\astar$ small enough that 
\begin{equation*} 
\mathrm{f}_{1,3}^< \le cs^{-100}+\frac1{2^5}\int_{y<-\frac{B}2} \varphi_B'(\partial_y \varepsilon)^2+\frac1{2}\int_{y<-\frac{B}2} \psi_B'(\partial_y^2 \varepsilon)^2  .
\end{equation*}
Finally, we get from Young's inequality and \eqref{cut:onR} that 
\begin{equation*} 
\mathrm{f}_{1,4}^< \le \frac1{2^7}\int_{y<-\frac{B}2} \varphi_B'\varepsilon^2+\frac1{8}\int_{y<-\frac{B}2} \psi_B'(\partial_y^2 \varepsilon)^2+c\int_{y<-\frac{B}2} \psi_B' |W+F|^8 \varepsilon^2+c\int_{y<-\frac{B}2} \psi_B'\varepsilon^{10}  .
\end{equation*}
Since $\psi_B'\sim (\varphi_B')^2$ in the region $y<-\frac{B}2$, we have 
\begin{equation*} 
\int_{y<-\frac{B}2} \varphi_B'\varepsilon^{10} \le \left\|\varepsilon^2\sqrt{\varphi_B'}\right\|_{L^\infty(y<-\frac{B}2)}^4 \left(\int \varepsilon^2 \right)\lesssim \delta(\astar)\int_{y<-\frac{B}2} \varphi_B'\left(\varepsilon^2+(\partial_y \varepsilon)^2\right).
\end{equation*}
Hence, we deduce that 
\begin{equation*} 
\mathrm{f}_{1,4}^< \le cs^{-100}+\frac1{2^6}\int_{y<-\frac{B}2} \varphi_B'(\partial_y \varepsilon)^2+\frac1{4}\int_{y<-\frac{B}2} \psi_B'(\partial_y^2 \varepsilon)^2  .
\end{equation*}

Therefore, we conclude gathering all those estimates that 
\begin{equation} \label{est:f1<}
\mathrm{f}_1^<+\int_{y<-\frac{B}2} \psi_B'(\partial_y^2\varepsilon)^2+\frac14\int_{y<-\frac{B}2} \varphi_B'\big((\partial_y\varepsilon)^2+\varepsilon^2 \big) \lesssim s^{-100}  .
\end{equation}

\noindent \emph{Estimate of $\mathrm{f}_1^>$.}
In the region $y>-\frac B2$, one has $\ph_B(y)=e^{\frac{y}B}$ and $\psi_B(y)=1$. Thus,
\begin{align}
\mathrm{f}_1^>&=
-\frac 1B \int_{y>-\frac{B}2} \left[ 3(\partial_y \varepsilon)^2 + (1-\frac1{B^2})\varepsilon^2 \right] e^{\frac{y}B}
\nonumber \\ & \quad -\frac{1}{3B}\int_{y>-\frac{B}2} \left[(W+F+\varepsilon)^6-(W+F)^6- 6(W+F+\varepsilon)^5\varepsilon\right]e^{\frac{y}B}
\nonumber \\
& \quad - 2\int_{y>-\frac{B}2} \left[(W+F+\varepsilon)^5-(W+F)^5-5(W+F)^4\varepsilon\right] \partial_y(W+F) (e^{\frac{y}B}-1) \nonumber \\
&=-\frac 1B\int_{y>-\frac{B}2} \left[3(\partial_y \varepsilon)^2 + \varepsilon^2 -5Q^4\varepsilon^2+20y Q'Q^3\varepsilon^2\right]e^{\frac{y}B}
+\cR_{\rm Vir}(\varepsilon)  ,\label{decomp:f1>}
\end{align}
where 
\begin{equation*}
\cR_{\rm Vir}=\cR_{\rm Vir,1}+ \cR_{\rm Vir,2}+ \cR_{\rm Vir,3}+ \cR_{\rm Vir,4}+\cR_{\rm Vir,5} 
\end{equation*} 
and 
\begin{align*} 
\cR_{\rm Vir,1}&=\frac1{B^3} \int_{y>-\frac{B}2} \varepsilon^2 e^{\frac{y}B} ; \\
\cR_{\rm Vir,2}&=- \frac1{3B}\int_{y>-\frac{B}2} \left[ (W+F+\varepsilon)^6-(W+F)^6 - 6 (W+F+\varepsilon)^5\varepsilon+15Q^4\varepsilon^2 \right]e^{\frac{y}B} ;\\
\cR_{\rm Vir,3}&= -2\int_{y>-\frac{B}2} \left[(W+F+\varepsilon)^5-(W+F)^5-5(W+F)^4\varepsilon-10(W+F)^3 \varepsilon^2\right] \\ & \quad \quad \quad \quad \quad \quad \times \partial_y(W+F)(e^{\frac{y}B}-1) ; \\
\cR_{\rm Vir,4}&=- 20 \int_{y>-\frac{B}2} \left[(W+F)^3\partial_y(W+F)-Q^3Q'\right](e^{\frac{y}B}-1)\varepsilon^2  ; \\ 
\cR_{\rm Vir,5}&=- 20 \int_{y>-\frac{B}2}Q^3Q'\left( e^{\frac{y}B}-1-\frac{y}B \right) \varepsilon^2 .
\end{align*}

To handle the first term on the right-hand side of \eqref{decomp:f1>}, we rely on the following coercivity property of the virial quadratic form (under the orthogonality conditions \eqref{ortho} ) proved in Lemma 3.5 in \cite{CoMa1} and which is a variant of Lemma 3.4 in \cite{MaMeRa1} based on Proposition 4 in \cite{MaMejmpa}. 
\begin{lemma}[Localized virial estimate]
There exist $B_{0}>100$ and $\mu_1>0$ such that
for all $B\geq B_{0}$, 
\begin{equation} \label{est:viriel}
\int_{y>-\frac{B}2} \left[ 3 (\partial_y\varepsilon)^2 + \varepsilon^2 - 5 Q^4 \varepsilon^2 + 20 y Q' Q^3 \varepsilon^2\right] \geq \mu_1 \int_{y>-\frac{B}2} \left[(\partial_y\varepsilon)^2 + \varepsilon^2 \right] e^{\frac{y}B}
- \frac 1 B \int \varepsilon^2 e^{-\frac{|y|}{2}}  .
\end{equation}
\end{lemma}

Now, we turn our attention to $ \cR_{\rm Vir}$. We begin by explaining how to control $\left|\cR_{\rm Vir,1}\right|$ and $\left|\cR_{\rm Vir,5}\right|$. We rely on the calculus inequality 
\begin{equation} \label{est:expo}
\left| e^{\frac{y}B}-1-\frac{y}B \right| \lesssim \frac{|y|^2}{B^2}e^{\frac{|y|}B}  .
\end{equation}
It follows that 
\begin{equation*} 
B\left|\cR_{\rm Vir,1}\right| +B\left|\cR_{\rm Vir,5} \right| \lesssim \frac1{B} \int_{y>-\frac{B}2}\varepsilon^2 e^{\frac{y}B}  .
\end{equation*}
Hence, $\left|\cR_{\rm Vir,1}\right|$ and $\left|\cR_{\rm Vir,5}\right|$ will be controlled by using the contribution coming from the first term on the right-hand side of \eqref{est:viriel} and by taking $B$ large enough.

To estimate $\cR_{\rm Vir,2}$, we write
\begin{align*}
&(W+F+\varepsilon)^6-(W+F)^6 - 6 (W+F+\varepsilon)^5\varepsilon +15 Q^4 \varepsilon^2
\\&\quad = \left[(W+F+\varepsilon)^6-(W+F)^6 - 6 (W+F)^5\varepsilon-15 (W+F)^4 \varepsilon^2\right] \\
&\qquad-6\left[(W+F+\varepsilon)^5-(W+F)^5 -5(W+F)^4\varepsilon\right] \varepsilon-15\left[(W+F)^4-Q^4\right]\varepsilon^2
\end{align*}
so that
\begin{align*}
&\left| \left[(W+F+\varepsilon)^6-(W+F)^6 - 6 (W+F+\varepsilon)^5\varepsilon\right] +15 Q^4 \varepsilon^2\right|
\\& \qquad \lesssim |W+F|^3|\varepsilon|^3+|\varepsilon|^6+|W+F-Q|(|W|+|F|+Q)^3\varepsilon^2.
\end{align*}
Hence, we get 
\begin{align*}
B\left|\cR_{\rm Vir,2} \right| &\lesssim \|\varepsilon \|_{L^{\infty}(y>-\frac{B}2)} \int_{y>-\frac{B}2} \left(1+|F|^3\right)\varepsilon^2e^{\frac{y}B} + \|\varepsilon \|_{L^{\infty}(y>-\frac{B}2)}^4 \int_{y>-\frac{B}2} \varepsilon^2 e^{\frac{y}B} \\ 
& \quad +\int_{y>-\frac{B}2}\left(|Q_b-Q|+|rR|+|F|\right) \left(1+|F|^3\right)\varepsilon^2e^{\frac{y}B}  .
\end{align*}
On the one hand, observe from the Sobolev embedding and the bootstrap assumption \eqref{BS:eps} that 
\begin{equation} \label{est:eps_>}
\|\varepsilon \|_{L^{\infty}(y>-\frac{B}2)} \lesssim \mathcal{N}_B(\varepsilon) \lesssim |s|^{-\frac54}  .
\end{equation} 
On the other hand, recall that $F(s,y)=\lambda^{\frac12}(s)f(s,\lambda(s)y+\sigma(s))$. For $s>2B$, we have $\lambda(s)y+\sigma(s)>\frac34 \sigma(s)$ in the region $y>-\frac{B}2$. Hence, it follows from \eqref{def:f_0}, \eqref{decay:q_0.1} and \eqref{BS:param} that 
\begin{equation} \label{est:F>}
\|F\|_{L^{\infty}(y>-\frac{B}2)} \lesssim s^{-1} \quad \text{and} \quad \|\partial_yF\|_{L^{\infty}(y>-\frac{B}2)} \lesssim s^{-2}  .
\end{equation}
Thus, we deduce by using \eqref{def:Qb}, \eqref{def:Pb}, \eqref{e:r}, \eqref{est:eps_>} and \eqref{est:F>} and by taking $|s|$ large enough that 
\begin{equation*} 
B\left|\cR_{\rm Vir,2} \right| \le \frac{\mu_1}4 \int_{y>-\frac{B}2}\varepsilon^2 e^{\frac{y}B}  .
\end{equation*}

To deal with $\cR_{\rm Vir,3}$, we observe 
\begin{align*} 
\left|\left[(W+F+\varepsilon)^5-(W+F)^5-5(W+F)^4 \varepsilon-10(W+F)^3\varepsilon^2\right] \partial_y(W+F)(e^{\frac{y}B}-1)\right|\\ 
\lesssim \left(|W+F|^2|\varepsilon|^3+|\varepsilon|^5 \right) \left(|\partial_yW|+|\partial_yF|\right)e^{\frac{y}B}  ,
\end{align*}
so that 
\begin{align*}
\left|\cR_{\rm Vir,3} \right| &\lesssim \|\varepsilon \|_{L^{\infty}(y>-\frac{B}2)} \int_{y>-\frac{B}2} \left(1+|F|^2\right)\left(1+|\partial_yF|\right)\varepsilon^2e^{\frac{y}B} \\ & \quad + \|\varepsilon \|_{L^{\infty}(y>-\frac{B}2)}^3 \int_{y>-\frac{B}2} \left(1+|\partial_yF|\right)\varepsilon^2 e^{\frac{y}B}  .
\end{align*}
Hence, we deduce from \eqref{est:eps_>} and \eqref{est:F>} by taking $|s|$ large enough (possibly depending on $B$) that
\begin{equation*} 
B\left|\cR_{\rm Vir,3} \right| \le \frac{\mu_1}8 \int_{y>-\frac{B}2}\varepsilon^2 e^{\frac{y}B}  .
\end{equation*}
To deal with $\cR_{\rm Vir,4}$, we write
\begin{equation*} 
(W+F)^3\partial_y(W+F)-Q^3Q' =\left( (W+F)^3-Q^3\right) \partial_y(W+F)+Q^3\partial_y(W-Q+F)  ,
\end{equation*}
so that 
\begin{align*}
\left|(W+F)^3\partial_y(W+F)-Q^3Q'\right| &\lesssim \left( |Q_b-Q|+|F|+|r||R|\right)\left(|W|+|F|+Q\right)^2\left( |\partial_yW|+|\partial_yF|\right)\\ & \quad +Q^3\left(|\partial_y(Q_b-Q)|+|r||R'|+|\partial_yF|\right),
\end{align*}
Thus, we deduce by using \eqref{def:Qb}, \eqref{def:Pb}, \eqref{e:r} and \eqref{est:F>} and by taking $|s|$ large enough (depending possibly on $B$) that 
\begin{equation*} 
B\left|\cR_{\rm Vir,4} \right| \le \frac{\mu_1}{2^4} \int_{y>-\frac{B}2}\varepsilon^2 e^{\frac{y}B}  .
\end{equation*}

Then, we deduce gathering all those estimates that 
\begin{equation} \label{est:f1>}
\mathrm{f}_1^>+\frac{\mu_1}{2B}\int_{y>-\frac{B}2} e^{\frac{y}B}\varepsilon^2 \lesssim \frac1{B^2} \int \varepsilon^2 e^{-\frac{|y|}{2}}  .
\end{equation}

The proof of \eqref{est:f1} follows by combining \eqref{est:f1<}, \eqref{est:f1>} and  choosing $\mu_0=2^{-5}\min\{1,\mu_1 \}$.

\smallskip

\noindent \emph{Estimate of $\mathrm{f}_2$.} We claim that 
\begin{equation} \label{est:f2}
\big| \mathrm{f}_{2}\big| \le \frac{\mu_0}2 \int \varphi_B'\left(\varepsilon^2+(\partial_y \varepsilon)^2\right) +cB^2s^ {-4}  .
\end{equation}

By using the decomposition in \eqref{decomp:EW}, we have 
\begin{equation*}
\mathrm{f}_2=2\int (\vec{m} \cdot \vec{\MM} Q ) G_B(\varepsilon)-2\int \mathcal{R} \, G_B(\varepsilon)=: \mathrm{f}_{2,1}+\mathrm{f}_{2,2}  .
\end{equation*}

We first deal with $\mathrm{f}_{2,1}$. By using the definition of $G_B(\varepsilon)$ in $\eqref{def:GBeps}$ and integration by parts, we compute 
\begin{align*}
\int \Lambda Q \, G_B(\varepsilon)&=\int \psi_B \mathcal{L}(\Lambda Q)\varepsilon -\int \psi_B'\partial_y(\Lambda Q) \varepsilon +\int (\varphi_B-\psi_B)\Lambda Q \, \varepsilon\\ & \quad -\int \Lambda Q \psi_B \left[ (W+F+\varepsilon)^5-(W+F)^5-5(W+F)^4\varepsilon\right] \\ &
\quad -5\int \Lambda Q \psi_B \left[ (W+F)^4-Q^4\right]\varepsilon  .
\end{align*}
and (also using $\mathcal{L}(\partial_yQ)=0$)
\begin{align*}
\int \partial_y Q \, G_B(\varepsilon)&= -\int \psi_B'\partial_y^2 Q \varepsilon +\int (\varphi_B-\psi_B)\partial_y Q \, \varepsilon-5\int \partial_y Q \psi_B \left[ (W+F)^4-Q^4\right]\varepsilon\\ & \quad -\int \partial_yQ \psi_B \left[ (W+F+\varepsilon)^5-(W+F)^5-5(W+F)^4\varepsilon\right] 
\end{align*}

We estimate each of these terms separately. By using the identity $\mathcal{L}\Lambda Q=-2Q$, the second and third orthogonality identities in \eqref{ortho}, the localisation properties of $\psi_B$, $\psi_B'$ and $\varphi_B$, \eqref{est:expo} and the decay properties of $Q$ and $\Lambda Q$, it follows that
\begin{align*}
&\left|\int \psi_B \mathcal{L}(\Lambda Q)\varepsilon \right|=2\left| \int (\psi_B-1) Q \varepsilon\right| \lesssim \int_{y<-\frac{B}2} e^{-\frac{|y|}2} |\varepsilon| \lesssim e^{-\frac{B}4} \|\varepsilon\|_{L^2_\loc}  ; \\
&\left|\int \psi_B'\partial_y^2 Q \varepsilon \right|+\left|\int \psi_B'\partial_y(\Lambda Q) \varepsilon \right| \lesssim \int_{y<-\frac{B}2} e^{-\frac{|y|}2} |\varepsilon| \lesssim e^{-\frac{B}4} \|\varepsilon\|_{L^2_\loc}  ;
\end{align*}
and
\begin{align*}
&\left|\int (\varphi_B-\psi_B)\Lambda Q \, \varepsilon \right|=2\left| \int(\varphi_B-\psi_B-\frac{y}B)\Lambda Q \, \varepsilon\right| \\ & \quad \quad \quad \lesssim \int_{y<-\frac{B}2} e^{-\frac{|y|}2}|\varepsilon|+\frac1{B^2}\int_{y>-\frac{B}2}e^{-\frac{|y|}2}|\varepsilon| \lesssim \left(e^{-\frac{B}4}+\frac1{B^2}\right) \|\varepsilon\|_{L^2_\loc}  ; \\ 
&\left|\int (\varphi_B-\psi_B)\partial_y Q \varepsilon \right|=\left|\int (\varphi_B-\psi_B-\frac{y}B)\partial_y Q \varepsilon \right|
\lesssim \int_{|y|>\frac{B}2} e^{-\frac{|y|}2}|\varepsilon| \\ & \quad \quad \quad \lesssim \int_{y<-\frac{B}2} e^{-\frac{|y|}2}|\varepsilon|+\frac1{B^2}\int_{y>-\frac{B}2}e^{-\frac{|y|}2}|\varepsilon| \lesssim \left(e^{-\frac{B}4}+\frac1{B^2}\right) \|\varepsilon\|_{L^2_\loc}  ; 
\end{align*}
where in the last line, we have also used the orthogonality condition  (from \eqref{ortho})
\begin{equation*} 
\int y\partial_yQ \varepsilon=\int \big( \Lambda Q-\frac12Q\big) \varepsilon=0  ,
\end{equation*}

Moreover, it follows from \eqref{def:Qb}, \eqref{def:Pb}, \eqref{e:r} and \eqref{e:F} that
\begin{align*}
\int \psi_B \left( |\Lambda Q|+\right. & \left. |\partial_yQ| \right) \left| (W+F)^4-Q^4\right| |\varepsilon| 
\\ &\lesssim \int e^ {-\frac{3|y|}4}|W+F-Q| \left(|W+F|^3+Q^3 \right) |\varepsilon|
 \lesssim s^ {-1}\|\varepsilon\|_{L^ 2_\loc}  .
\end{align*}
To deal with the nonlinear term, we recall the Sobolev bound 
\begin{equation} \label{Sob:eps_varphi}
\|\varepsilon^2 \sqrt{\psi_B}\|_{L^{\infty}}^2 \lesssim \left(\int \varepsilon^2 \right)\int \psi_B\, (\varepsilon^2+(\partial_y\varepsilon)^2)\lesssim \delta(\astar) \mathcal{N}_B(\varepsilon)^2  .
\end{equation}
Hence, we deduce from \eqref{e:r}, \eqref{e:F}, \eqref{cut:onR} and \eqref{BS:eps} that
\begin{align*}
\int \psi_B \left( |\Lambda Q|+|\partial_yQ| \right) &\left| (W+F+\varepsilon)^5-(W+F)^5-5(W+F)^4\varepsilon\right| \\ &
\lesssim \int \psi_B e^ {-\frac{3|y|}4}\left( |W+F|^3 \varepsilon^2+|\varepsilon|^5 \right)
\\ & \lesssim B\int \varphi_B'\varepsilon^2+\|\varepsilon^ 2\sqrt{\psi_B}\|_{L^{\infty}}^2 \int e^{-\frac{3|y|}4}|\varepsilon| 
\lesssim B\int \varphi_B'\varepsilon^2+s^{-\frac52} \|\varepsilon\|_{L^2_\loc} .
\end{align*}

Therefore, we deduce combining those estimates with \eqref{est:L2loc} and \eqref{BS:m}, and choosing $s$ and $B>100$ large enough that 
\begin{equation} \label{est:f21}
\big| \mathrm{f}_{2,1}\big| \le \frac{\mu_0}4 \int \varphi_B'\left(\varepsilon^2+(\partial_y \varepsilon)^2\right) +cs^ {-4}  .
\end{equation}

Now, we turn to $\mathrm{f}_{2,2}$. We compute from the definition of $G_B(\varepsilon)$ in \eqref{def:GBeps} 
\begin{equation*} 
\mathrm{f}_{2,2}=2\int \psi_B\partial_y\mathcal{R}\, \partial_y \varepsilon+2\int \varphi_B \mathcal{R} \, \varepsilon
-2\int \psi_B \mathcal{R} \left[(W+F+\varepsilon)^5-(W+F)^5 \right] .
\end{equation*}
By  the Cauchy-Schwarz inequality and the properties of $\varphi_B$ and $\psi_B$ in \eqref{cut:onR}, it holds
\begin{align*} 
&\left|\int \varphi_B \mathcal{R} \varepsilon \right| \le\left(\int \varphi_B^2e^{-\frac{y}B} \varepsilon^2 \right)^{\frac12}
\left(\int e^{\frac{y}B} \mathcal{R}^2 \right)^{\frac12} \le \frac{\mu_0}{2^{4}} \int \varepsilon^2 \varphi_B' +c B^2\| \mathcal{R}\|_{L^2_B}^2  ; \\ 
&\left|\int \psi_B \partial_y\mathcal{R} \, \partial_y\varepsilon \right| \le\left(\int \psi_B^2e^{-\frac{y}B} (\partial_y\varepsilon)^2 \right)^{\frac12}
\left(\int e^{\frac{y}B} (\partial_y\mathcal{R})^2 \right)^{\frac12} \le \frac{\mu_0}{2^{5}} \int (\partial_y\varepsilon)^2 \varphi_B' +c B^2\| \partial_y\mathcal{R}\|_{L^2_B}^2  .
\end{align*} 
To treat the nonlinear term, we observe that 
\begin{equation*} 
\left|\int \psi_B \mathcal{R} \left[(W+F+\varepsilon)^5-(W+F)^5 \right] \right| \lesssim \int \psi_B |\mathcal{R}|\left(|W|^4+|F|^4 \right)|\varepsilon|
+ \int \psi_B |\mathcal{R}||\varepsilon|^5  .
\end{equation*}
On the one hand, we deduce from \eqref{est:F>}, and then \eqref{cut:onR} and \eqref{est:Feps} that 
\begin{align*}
 \int \psi_B |\mathcal{R}|\left(|W|^4+|F|^4 \right)|\varepsilon| &\le \left(\int \psi_B^2e^{-\frac{y}B}(1+|F|^8\mathbf{1}_{\{y<-\frac{B}2\}}) \varepsilon^2 \right)^{\frac12}
\left(\int e^{\frac{y}B} \mathcal{R}^2 \right)^{\frac12}
\\ & \lesssim \frac{\mu_0}{2^{6}} \int \varepsilon^2 \varphi_B' +c B^2\| \mathcal{R}\|_{L^2_B}^2  .
\end{align*}
On the other hand, \eqref{cut:onR}, the Sobolev bound \eqref{Sob:eps_varphi}, and then \eqref{BS:eps}, \eqref{est:eps6} yield
\begin{align*}
 \int \psi_B |\mathcal{R}||\varepsilon|^5 &\lesssim \left\|\varepsilon^2\sqrt{\psi_B}\right\|_{L^\infty}
 \left(\int \varphi_B^2e^{-\frac{y}B} \varepsilon^6 \right)^{\frac12}
\left(\int e^{\frac{y}B} \mathcal{R}^2 \right)^{\frac12} \lesssim s^{-\frac54} \left( \int \varepsilon^2 \varphi_B' +c B^2\| \mathcal{R}\|_{L^2_B}^2\right).
\end{align*}

Then, we deduce combining those estimates with \eqref{est:R:L2B}, \eqref{est:L2loc}, \eqref{BS:m} and \eqref{BS:bs} that
\begin{equation} \label{est:f22}
\big| \mathrm{f}_{2,2}\big| \le \frac{\mu_0}{8} \int \varepsilon^2 \varphi_B'+c B^2s^{-4}  .
\end{equation}

Finally, we conclude the proof of \eqref{est:f2} gathering \eqref{est:f21} and \eqref{est:f22}.

\smallskip

\noindent \emph{Estimate of $\mathrm{f}_3$.} We claim that 
\begin{equation} \label{est:f3}
\left| \mathrm{f}_{3}\right| \le \frac{\mu_0}{4} \int \varphi_B'\left(\varepsilon^2+(\partial_y \varepsilon)^2\right) +cs^ {-100}  .
\end{equation}

From the definition of $G_B(\varepsilon)$ in \eqref{def:GBeps}, we have
\begin{align*}
\mathrm{f}_{3}&= \left(\frac{\sigma_s}{\lambda}-1\right) \int \partial_y\varepsilon \left[ -\partial_y(\psi_B\partial_y\varepsilon)+\varphi_B\varepsilon-\psi_B \left((W+F+\varepsilon)^5-(W+F)^5 \right) \right] \\ & 
 =: \mathrm{f}_{3,1}+\mathrm{f}_{3,2}+\mathrm{f}_{3,3}  .
\end{align*}
By using the identities 
\begin{equation*} 
\int \partial_y\varepsilon\, \partial_y(\psi_B\partial_y\varepsilon)=\frac12 \int \psi_B' (\partial_y \varepsilon)^2 \quad \text{and} \quad \int \partial_y\varepsilon \varphi_B \varepsilon=-\frac12 \int \varphi_B' \varepsilon^2  ,
\end{equation*}
we deduce from \eqref{BS:m} and \eqref{cut:onR} that 
\begin{equation*} 
\left| \mathrm{f}_{3,1}\right|+\left| \mathrm{f}_{3,2}\right| \lesssim s^{-\frac54}\int \varphi_B' \left(\varepsilon^2+(\partial_y\varepsilon)^2 \right).
\end{equation*}

To deal with $\mathrm{f}_{3,3}$, we compute
\begin{align*} 
 \int \partial_y\varepsilon\psi_B &\left((W+F+\varepsilon)^5-(W+F)^5 \right)\\&=-\frac16\int \psi_B' \left((W+F+\varepsilon)^6-(W+F)^6-6(W+F)^5\varepsilon \right) \\ 
 & \quad -\int \psi_B \partial_y(W+F) \left((W+F+\varepsilon)^5-(W+F)^5-5(W+F)^4\varepsilon \right),
\end{align*}
so that it follows from \eqref{BS:m}, \eqref{cut:onR}, and then \eqref{est:Feps}, \eqref{est:Feps_deriv}, \eqref{est:eps6}, \eqref{est:F>}, that 
\begin{equation*} 
\left| \mathrm{f}_{3,2}\right| \lesssim Bs^{-\frac54} \int \varphi_B' \left[\left(|W+F|^4+|\partial_y(W+F)|^4\right)\varepsilon^2+\varepsilon^6 \right]
 \lesssim Bs^{-\frac54}\int \varphi_B'\left((\partial_y \varepsilon)^2+ \varepsilon^2 \right)+s^{-100} .
\end{equation*}

Therefore, we deduce the proof of \eqref{est:f3} gathering those estimates and taking $s$ large enough (possibly depending on $B$).

\smallskip

\noindent \emph{Estimate of $\mathrm{f}_4$.} We claim that 
\begin{equation} \label{est:f4}
 \mathrm{f}_{4} \le \frac{\mu_0}{8} \int \varphi_B'\left(\varepsilon^2+(\partial_y \varepsilon)^2\right) +cs^ {-100}  .
\end{equation}

Recall from the definition of $\mathcal{F}$ in \eqref{def:F} that 
\begin{align*}
\mathrm{f}_4&=2\frac{\lambda_s}{\lambda}\int \Lambda \varepsilon \left[ -\partial_y(\psi_B\partial_y \varepsilon)+\varphi_B \varepsilon-\psi_B \left( (W+F+\varepsilon)^5-(W+F)^5\right)\right] \\ 
& \quad +\kappa\frac{\lambda_s}{\lambda}\int \left[\psi_B (\partial_y \varepsilon)^2+\varphi_B \varepsilon^2 -\frac13\psi_B\left( (W+F+\varepsilon)^6-(W+F)^6-6(W+F)^5\varepsilon\right)\right]  .
\end{align*}
We compute integrating by parts (see also page 97 in \cite{MaMeRa1})
\begin{align*} 
& \int (\Lambda \varepsilon) \partial_y(\psi_B\partial_y\varepsilon)=-\int \psi_B(\partial_y\varepsilon)^2+\frac12\int y\psi_B'(\partial_y\varepsilon)^2  ; \\ 
& \int (\Lambda \varepsilon) \varepsilon \varphi_B=-\frac12 \int y \varphi_B' \varepsilon^2  ;\\
& \int (\Lambda \varepsilon) \psi_B \left( (W+F+\varepsilon)^5-(W+F)^5\right)\\ & \quad \quad =\frac16 \int (2\psi_B-y\psi_B')\left( (W+F+\varepsilon)^6-(W+F)^6-6(W+F)^5\varepsilon \right)\\ &\quad \quad \quad -\int \psi_B \Lambda (W+F) \left( (W+F+\varepsilon)^5-(W+F)^5-5(W+F)^4\varepsilon \right).
\end{align*}
Hence, we deduce gathering those identities that 
\begin{align*}
\mathrm{f}_4&=\frac{\lambda_s}{\lambda}\left((2+\kappa)\int \psi_B (\partial_y \varepsilon)^2 -\int y\psi_B'(\partial_y\varepsilon)^2 +\kappa\int \varphi_B \varepsilon^2-\int y\varphi_B' \varepsilon^2\right)\\ & \quad -\frac13 \frac{\lambda_s}{\lambda}\int \left( (k+2)\psi_B-y\psi_B'\right)\left( (W+F+\varepsilon)^6-(W+F)^6-6(W+F)^5\varepsilon \right) \\ & \quad +2 \frac{\lambda_s}{\lambda}\int \psi_B \Lambda (W+F) \left( (W+F+\varepsilon)^5-(W+F)^5-5(W+F)^4\varepsilon \right)\\
&=:\mathrm{f}_{4,1}+\mathrm{f}_{4,2}+\mathrm{f}_{4,3}  .
\end{align*}
We will control each of these terms separately. Observe that $\frac{\lambda_s}{\lambda} >0$ since we are in a defocusing regime (see \eqref{BS:param} and \eqref{BS:m}). Thus 
\begin{equation*} 
\frac{\lambda_s}{\lambda}\left( -\int_{y>0} y\psi_B'(\partial_y\varepsilon)^2 -\int_{y>0} y\varphi_B' \varepsilon^2\right) \le 0  .
\end{equation*}
Moreover, we get by using \eqref{BS:param}, \eqref{BS:m} and \eqref{cut:onR}
\begin{equation*} 
\frac{\lambda_s}{\lambda} \int_{y<0}|y\psi_B'|(\partial_y\varepsilon)^2 \lesssim s^{-1}\int \varphi_B'(\partial_y\varepsilon)^2 
\end{equation*}
and, by using H\"older and Young inequalities, 
\begin{align*}
\frac{\lambda_s}{\lambda} \int_{y<0}|y\varphi_B'|\varepsilon^2 &\lesssim s^{-1}\left(\int |y|^{100}e^{\frac{y}B} \varepsilon^2\right)^{\frac1{100}}\left(\int_{y<0} e^{\frac{y}B} \varepsilon^2\right)^{\frac{99}{100}} 
\le \frac{\mu_0}{2^4} \int \varphi_B'\varepsilon^2+s^{-100}\|\varepsilon\|_{L^2}^2  .
\end{align*}

On the other hand, the terms $(2+\kappa)\int \psi_B (\partial_y \varepsilon)^2$ and $\kappa\int \varphi_B \varepsilon^2$ are positive as well as their product with $\frac{\lambda_s}{\lambda}$. However, we can estimate them as above. It follows from \eqref{BS:param} and \eqref{BS:m} that 
\begin{equation*} 
\left| \frac{\lambda_s}{\lambda} \right|\left( \int \psi_B (\partial_y \varepsilon)^2+\int \varphi_B \varepsilon^2\right) \lesssim s^{-1}B \int \varphi_B'\left((\partial_y \varepsilon)^2+ \varepsilon^2 \right).
\end{equation*}

Now, we deal with the nonlinear terms. By using \eqref{BS:param}, \eqref{BS:m} and \eqref{cut:onR}, and then \eqref{est:Feps}, \eqref{est:eps6}, \eqref{est:F>}, we get that 
\begin{equation*} 
\left|\mathrm{f}_{4,2}\right| \lesssim Bs^{-1} \int \varphi_B' \left(|W+F|^4\varepsilon^2+\varepsilon^6 \right)
 \lesssim Bs^{-1}\int \varphi_B'\left((\partial_y \varepsilon)^2+ \varepsilon^2 \right)+s^{-100}  .
\end{equation*}
By definition $\Lambda (W+F)=\frac12 (W+F)+y\partial_y(W+F)$. Moreover, we use that, for $k=0$ or $1$, 
\begin{align*}
&\left| \partial_y^k(W+F)\right| \left| (W+F+\varepsilon)^5-(W+F)^5-5(W+F)^4\varepsilon \right| \\ 
& \lesssim \left| \partial_y^k(W+F)\right| \left(|W+F|^3\varepsilon^2+\varepsilon^5 \right)
 \lesssim \left(\left| \partial_y^k(W+F)\right||W+F|^3+ \left| \partial_y^k(W+F)\right|^4\right) \varepsilon^2+\varepsilon^6  .
\end{align*}
Thus, we deduce from \eqref{BS:param}, \eqref{BS:m} and \eqref{cut:onR}, and then \eqref{est:Feps}, \eqref{est:Feps_deriv}, \eqref{est:eps6}, \eqref{est:F>}, that 
\begin{equation*} 
\left|\mathrm{f}_{4,2}\right| \lesssim Bs^{-1} \int \varphi_B' \left[\left(|W+F|^4+|\partial_y(W+F)|^4\right)\varepsilon^2+\varepsilon^6 \right]
 \lesssim Bs^{-1}\int \varphi_B'\left((\partial_y \varepsilon)^2+ \varepsilon^2 \right)+s^{-100} .
\end{equation*}

Therefore, we conclude the proof of \eqref{est:f4} gathering these estimates. 

\smallskip

\noindent \emph{Estimate of $\mathrm{f}_5$.} We claim that 
\begin{equation} \label{est:f5}
\left| \mathrm{f}_{5}\right| \le \frac{\mu_0}{2^4} \int \varphi_B'\left(\varepsilon^2+(\partial_y \varepsilon)^2\right) +cs^ {-100}  .
\end{equation}

We decompose, from the definition of $\mathrm{f}_{5}$,
\begin{align*}
\mathrm{f}_{5}&= - 2b_s\int \frac{\partial Q_b}{\partial b} \left[(W+F+\varepsilon)^5-(W+F)^5-5(W+F)^4\varepsilon\right]\psi_B \\ 
&\quad - 2\int F_s \left[(W+F+\varepsilon)^5-(W+F)^5-5(W+F)^4\varepsilon\right]\psi_B \\ 
&=: \mathrm{f}_{5,1}+ \mathrm{f}_{5,2}  ,
\end{align*}
and estimate these two terms separately. 

To deal with $\mathrm{f}_{5,1}$, we use that 
\begin{align*}
&\left| \frac{\partial Q_b}{\partial b} \left[(W+F+\varepsilon)^5-(W+F)^5-5(W+F)^4\varepsilon\right] \right|\\
&\quad \quad \lesssim \left|\frac{\partial Q_b}{\partial b} \right|\left( |W+F|^3\varepsilon^2+|\varepsilon|^5\right)
\lesssim \left(|W|^4+|F|^4+ \left|\frac{\partial Q_b}{\partial b} \right|^4\right)\varepsilon^2+\varepsilon^6  .
\end{align*}
Moreover, we observe by using \eqref{eq:dQb:db} and \eqref{cut:onR} that 
\begin{equation*} 
\left|\frac{\partial Q_b}{\partial b}(y) \right|^4\psi_B(y) \lesssim B^4\varphi_B(y)\lesssim B^5\varphi_B'(y)  .
\end{equation*}
Then, it follows from \eqref{BS:bs}, \eqref{est:Feps}, \eqref{est:eps6} and \eqref{est:F>} that
\begin{equation*} 
\left| \mathrm{f}_{5,1}\right| \lesssim B^5s^{-2}\int \varphi_B'\left( 1+|W|^4+|F|^4\right) \varepsilon^2+Bs^{-2}\int \varphi_B' \varepsilon^6
\lesssim B^5s^{-2}\int \varphi_B' \left( \varepsilon^2+(\partial_y\varepsilon)^2\right).
\end{equation*}

To handle $\mathrm{f}_{5,2}$, we first observe, arguing as above, 
\begin{equation*} 
\left| \mathrm{f}_{5,2}\right| \lesssim B \int \psi_B |F_s|\left( |W|^3+|F|^3+|F_s|^3\right) \varepsilon^2+\int \varphi_B' \varepsilon^6  .
\end{equation*}
The second term on the right-hand side of the above inequality will be dealt by using \eqref{est:eps6} and taking $\astar$ small enough. Recalling that $F(s,y)=\lambda^{\frac12}(s)f(s,\lambda(s)y+\sigma(s))$, we compute 
\begin{equation*} 
F_s(s,y)=\frac12\frac{\lambda_s}{\lambda} \lambda^{\frac12} f(s,\lambda(s)y+\sigma(s))+\left[\frac{\lambda_s}{\lambda}y+\frac{\sigma_s}{\lambda} \right]\lambda^{\frac32}\partial_y f(s,\lambda(s)y+\sigma(s))+\lambda^{\frac72}f_t(s,\lambda(s)y+\sigma(s))  .
\end{equation*}
Now, we argue as in the proof of \eqref{est:Feps} and split the integration domain into the two regions $\lambda(s)y>-\frac14\sigma(s) \iff \lambda(s)y+\sigma(s)>\frac34 \sigma(s)$ and $\lambda(s)y<-\frac14\sigma(s)$ (from \eqref{BS:param}). Thus, we deduce from \eqref{def:f_0}, \eqref{decay:q_0.1}, \eqref{decay:q_0.2}, \eqref{BS:param} and \eqref{BS:m} that 
\begin{equation*} 
\left| \mathrm{f}_{5,2}\right| \le cB^2s^{-2} \int \varepsilon^2 \varphi_B'+\frac{\mu_0}{2^5}\int \varphi_B' \left(\varepsilon^2+(\partial_y\varepsilon)^2 \right)+cs^{-100}  .
\end{equation*}

Therefore, we conclude the proof of \eqref{est:f4} gathering those estimates and taking $s$ large enough (possibly depending on $B$). 

\smallskip 

Finally, we conclude the proof of \eqref{dsF} gathering \eqref{est:f1}, \eqref{est:f2}, \eqref{est:f3}, \eqref{est:f4} and \eqref{est:f5}.

\smallskip

Now, we turn to the proof of \eqref{coF}. We decompose $\mathcal{F}$ as follows:
\begin{align*}
\F &= \int \left[(\partial_y\varepsilon)^2 \psi_B + \varepsilon^2 \ph_B-5Q^4\varepsilon^2 \psi_B \right] \\
& \quad - \frac 13\int \left[ (W +F+ \varepsilon)^6 - (W+F)^6 - 6(W+F)^5\varepsilon-15Q^4\varepsilon^2 \right] \psi_B \\ 
&=:\F_1+\F_2  .
\end{align*}
To bound by below $\F_1$, we rely on the coercivity of the linearized energy \eqref{coercivity.2} with the choice of the orthogonality conditions \eqref{ortho} and standard localisation arguments. Proceeding for instance as in the Appendix A of \cite{MaMe} or as in the proof of Lemma 3.5. in \cite{CoMa1}, we deduce that there exists $\tilde{\nu}_0>0$ such that, for $B$ large enough,
\begin{equation*} 
\F_1 \ge \tilde{\nu}_0 \, \mathcal{N}_{B}(\varepsilon)^2  .
\end{equation*}
To estimate $\F_2$, we compute
\begin{align*}
&(W+F+\varepsilon)^6-(W+F)^6 - 6 (W+F)^5\varepsilon -15 Q^4 \varepsilon^2
\\&= \left[(W+F+\varepsilon)^6-(W+F)^6 - 6 (W+F)^5\varepsilon-15 (W+F)^4 \varepsilon^2\right] -15\left[(W+F)^4-Q^4\right]\varepsilon^2,
\end{align*}
so that
\begin{align*}
\left| \F_2\right| \lesssim \int \psi_B \left[\left(|W-Q|^4+|F|^4\right)\varepsilon^2+Q^3|\varepsilon|^3+\varepsilon^6+|W+F-Q|(|W|+|F|+Q)^3\varepsilon^2\right] .
\end{align*}
We will control each term on the right-hand side separately. First observe from \eqref{def:Qb} and \eqref{e:r} and arguing as for \eqref{est:Feps} (but without the restriction $y<-\frac{B}2$) that 
\begin{equation*}
\int \psi_B\left(|W-Q|^4+|F|^4\right)\varepsilon^2 \lesssim s^{-4} \mathcal{N}_B(\varepsilon)^2+s^{-100} 
\end{equation*}
and
\begin{equation*}
\int \psi_B|W+F-Q|(|W|+|F|+Q)^3\varepsilon^2 \lesssim s^{-1} \mathcal{N}_B(\varepsilon)^2+s^{-100}  .
\end{equation*}
Moreover, we deduce from \eqref{Sob:eps_varphi} that 
\begin{equation*}
\int \psi_B Q^3 |\varepsilon|^3 \lesssim \left\| \varepsilon^2\sqrt{\psi_B} \right\|_{L^{\infty}} \left(\int Q^3|\varepsilon|\right)\lesssim \delta(\astar) \mathcal{N}_B(\varepsilon)^2 
\end{equation*}
and 
\begin{equation*}
\int \psi_B \varepsilon^6 \lesssim \left\| \varepsilon^2\sqrt{\psi_B} \right\|^2_{L^{\infty}} \left(\int \varepsilon^2 \right) \lesssim \delta(\astar)\mathcal{N}_B(\varepsilon)^2 .
\end{equation*}
Therefore, we conclude the proof of \eqref{coF} gathering those estimates.
\end{proof}

\section{Construction of flattening solitons}

\subsection{End of the construction in rescaled variables}
In this subsection, we still work with the notation introduced in Sections~\ref{section:decomp_sol} and~\ref{section:energy}.
We prove that for well adjusted initial data,
 the decomposition of the solution introduced in \eqref{decomp:v} and the bootstrap estimates \eqref{BS:param} and \eqref{BS:eps} hold true in the whole time interval $[s_0,+\infty)$ for $s_0$ large enough (equivalently, $x_0$ large enough). The result is summarized in the next proposition. 
\begin{proposition} \label{prop:BS}
For $s_0>0$ large enough, let 
$\sigma_0$ and $\lambda_0$ be such that
\begin{equation} \label{def:sigma_lamba:ini} 
\left| \lambda_0-s_0^{\frac{2(1-\theta)}{2\theta-1}} \right| \leq s_0^{\frac{2(1-\theta)}{2\theta-1}-2\rho}
 \quad \text{and} \quad 
\left|\sigma_0 - (2\theta-1)s_0^{\frac1{2\theta-1}}\right|\leq s_0^{\frac1{2\theta-1}-2\rho}  
\end{equation}
where $\rho$ satisfies \eqref{def:rho}.
Let $\varepsilon_0 \in H^1(\mathbb R)$ be such that 
\begin{equation} \label{def:eps:ini}
 \|\varepsilon_0\|_{H^1}^2+\int\varepsilon_0^2e^{\frac{y}{10}} \le s_0^{-10} 
\quad \text{and} \quad
(\varepsilon_0,\Lambda Q)=(\varepsilon_0,y\Lambda Q)=(\varepsilon_0,Q)=0  .
\end{equation}
Then there exists $b_0 \in \mathbb R$ satisfying
\begin{equation} \label{b:ini} 
b_0 \in \mathcal{D}_0:=\left[-\frac{2(1-\theta)}{2\theta-1}s_0^{-1}-s_0^{-1-3\rho},-\frac{2(1-\theta)}{2\theta-1}s_0^{-1}+s_0^{-1-3\rho}\right]
\end{equation}
such that the solution of \eqref{gkdv} evolving from 
\begin{equation*} 
U_0(x):=\lambda_0^{-\frac12}\left(Q_{b_0}+\lambda_0^{\frac12}f(s_0,\sigma_0)R+\varepsilon_0 \right)\left(\frac{x-\sigma_0}{\lambda_0} \right)+f(s_0,x) 
\end{equation*}
has a decomposition $\big(\lambda(s),\sigma(s),b(s),\varepsilon(s)\big)$ as in Lemma \ref{lemma:decomp} satisfying the bootstrap conditions~\eqref{BS:param}, \eqref{BS:eps} and
\begin{align} 
|h(s)| &\le s^{\frac{1-\theta}{2\theta-1}-2\rho}\label{BS:h}  ,
\\ 
|g(s)| &\le s^{-\frac{3-2\theta}{2\theta-1}-3\rho}  ,\label{BS:g} 
\end{align}
on $[s_0,+\infty)$, where $h$ and $g$ are defined in Lemma \ref{lemma:gh}.
\end{proposition}

\begin{proof} We argue by contradiction and assume that for all $b_0$ satisfying \eqref{b:ini}, 
\begin{equation*} 
s^{\star}=s^{\star}(b_0):=\sup \big\{ s \ge s_0 : \eqref{BS:param}, \eqref{BS:eps}, \eqref{BS:h}, \eqref{BS:g} \ \text{hold} \ [s_0,s]  \big\}< +\infty  .
\end{equation*}
We will first show that we can strictly improve \eqref{BS:param}, \eqref{BS:eps} and \eqref{BS:h} on $[s_0,s^{\star}]$, and then find a contraction for \eqref{BS:g} by using a topological argument (see similar argument in \cite{CoMaMe}).

\smallskip 

\noindent \textit{Closing \eqref{BS:eps}.} We deduce integrating \eqref{dsF} on $[s_0,s]$ with $s \le s^{\star}$ and using \eqref{def:kappa}, \eqref{coF} and \eqref{def:eps:ini} that 
 \begin{equation*} 
 \mathcal{N}_B(\varepsilon)^2(s) \lesssim \mathcal{F}(s) +s^{-100}\lesssim s^{-3}+\left(\lambda_0^{\kappa}\mathcal{F}(s_0)-s_0\right)s^{-4} \le \frac14 s^{-\frac52}  ,
 \end{equation*}
 if $s_0$ is chosen large enough, which strictly improves \eqref{BS:eps}.
 
 \smallskip 
 
 \noindent \textit{Closing \eqref{BS:param}.} First, by using \eqref{def:gh}, and then \eqref{BS:param} and \eqref{BS:h}, we see that 
 \begin{equation} \label{BS:est:lambda_sigma}
 \left|\lambda(s)-\left(\frac{2}{\int Q}\frac{c_0}{1-\theta}\right)^2 \sigma^{2-2\theta}(s) \right| 
 \lesssim \left(\lambda^{\frac12}(s)+\sigma^{1-\theta}(s)\right)|h(s)| \lesssim s^{\frac{2(1-\theta)}{2\theta-1}-2\rho}  .
 \end{equation}
 Thus, we deduce from \eqref{BS:param} and \eqref{BS:m} that 
 \begin{equation*} 
 \left|\sigma_s(s)-\left(\frac{2}{\int Q}\frac{c_0}{1-\theta}\right)^2 \sigma^{2-2\theta}(s) \right| 
 \lesssim \lambda(s) \left| \frac{\sigma_s}{\lambda}+1\right|+\lambda^{\frac12}(s) |h(s)| \lesssim s^{\frac{2(1-\theta)}{2\theta-1}-2\rho}  ,
 \end{equation*}
since $2\rho<\frac54$, which yields, by using \eqref{BS:param},
\begin{equation} \label{BS:est:sigma_s}
 \left|\left(\sigma^{2\theta-1}\right)_s(s)-(2\theta-1)\left(\frac{2}{\int Q}\frac{c_0}{1-\theta}\right)^2 \right| 
 \lesssim \sigma^{2\theta-2}(s)s^{\frac{2(1-\theta)}{2\theta-1}-2\rho} \lesssim s^{-2\rho} .
 \end{equation}
 Observe from the definition of $c_0$ in \eqref{def:c0bis} that 
 \begin{equation*} 
 (2\theta-1)\left(\frac{2}{\int Q}\frac{c_0}{1-\theta}\right)^2=(2\theta-1)^{2\theta-1}  .
 \end{equation*}
 This implies integrating \eqref{BS:est:sigma_s} over $[s_0,s]$, for $s \le s^{\star}$, and using the condition on $\sigma_0$ in \eqref{def:sigma_lamba:ini} that
 \begin{equation*} 
 \left|\sigma^{2\theta-1}(s)-(2\theta-1)^{2\theta-1}s \right| 
 \lesssim s^{1-2\rho} \iff \left|\left(\frac{\sigma(s)}{(2\theta-1)s^{\frac1{2\theta-1}}}\right)^{2\theta-1}-1 \right| 
 \lesssim s^{-2\rho} .
 \end{equation*}
 Hence, we deduce by applying the mean value theorem to the function $r(\tau)=\tau^{\frac1{2\theta-1}}$ that
 \begin{equation} \label{BS:est:sigma}
 \left|\frac{\sigma(s)}{(2\theta-1)s^{\frac1{2\theta-1}}}-1 \right| 
 \lesssim s^{-2\rho} \iff \left|\sigma(s)-(2\theta-1)s^{\frac1{2\theta-1}} \right| 
 \lesssim s^{\frac1{2\theta-1}-2\rho} ,
 \end{equation}
which, for $s_0$ large enough, strictly improves the estimate for $\sigma$ in \eqref{BS:param}. 

Next, we get inserting \eqref{BS:est:sigma} in \eqref{BS:est:lambda_sigma} that 
\begin{equation} \label{BS:est:lambda}
\left|\frac{\lambda(s)}{s^{\frac{2(1-\theta)}{2\theta-1}}}-1 \right| 
\lesssim s^{-2\rho} \iff \left|\lambda(s)-s^{\frac{2(\theta-1)}{2\theta-1}} \right| 
\lesssim s^{\frac{2(1-\theta)}{2\theta-1}-2\rho} ,
\end{equation}
for all $s \in [s_0,s^{\star}]$, which, for $s_0$ large enough, strictly improves the estimate for $\lambda$ in \eqref{BS:param}.
 
Finally, we deduce from \eqref{BS:g}, \eqref{BS:est:sigma} and \eqref{BS:est:lambda} that 
\begin{align*}
\left|\frac{2\theta-1}{2(1-\theta)}sb(s)+1 \right| &=\frac{2\theta-1}{2(1-\theta)} \lambda^2(s)s\left|\frac{b(s)}{\lambda^2(s)}+\frac{2(1-\theta)}{2\theta-1}s^{-1}\lambda^{-2}(s) \right|\\
&\lesssim \lambda^2(s)s\left( |g(s)|+\left|\frac{4c_0}{\int Q}\lambda^{-\frac32}(s)\sigma^{-\theta}(s)-\frac{2(1-\theta)}{2\theta-1}s^{-1}\lambda^{-2}(s)\right|\right)\lesssim s^{-2\rho}  .
\end{align*}
 Thus, 
 \begin{equation*}
\left|b(s)+\frac{2(1-\theta)}{2\theta-1}s^{-1} \right| \lesssim s^{-1-2\rho}
 \end{equation*}
 for all $s \in [s_0,s^{\star}]$, which, for $s_0$ large enough, strictly improves the estimate for $b$ in \eqref{BS:param}.
 
 \smallskip 
 
 \noindent \textit{Closing \eqref{BS:h}.} By using \eqref{BS:hs} and \eqref{BS:g}, we get for any $s \in [s_0,s^{\star}]$, 
 \begin{equation} \label{BS:est:hs}
 \left| h_s(s) \right| \lesssim \lambda^{\frac52}(s)|g(s)|+\lambda^{\frac12}(s)s^{-\frac54}\lesssim s^{-1+\frac{1-\theta}{2\theta-1}-3\rho}  ,
 \end{equation}
 for $s_0$ large enough, since $3\rho < \frac54$. Moreover, observe by the definition of $c_0$ in \eqref{def:c0bis}, the definition of $h$ in \eqref{def:gh} and the choice of $\sigma_0$ and $\lambda_0$ in \eqref{def:sigma_lamba:ini} that $h(s_0)=0$. Hence, it follows integrating \eqref{BS:est:hs} over $[s_0,s^{\star}]$ and using the condition $3\rho <\frac{1-\theta}{2\theta-1}$ that
 \begin{equation*} 
 |h(s)| \lesssim s^{\frac{1-\theta}{2\theta-1}-3\rho}  ,
 \end{equation*}
 for all $s \in [s_0,s^{\star}]$, which, for $s_0$ large enough, strictly improves the estimate for $h$ in \eqref{BS:h}.
 
 \smallbreak
 
 \noindent \textit{Contradiction through a topological argument.} To simplify the notation, for any $b_0$ satisfying \eqref{b:ini}, we introduce 
 \begin{equation}\label{def:mu}
\mu : b_0 \in \mathcal{D}_0 \mapsto \mu_0=\mu(b_0) :=\left(b_0+\frac{2(1-\theta)}{2\theta-1}s_0^{-1}\right) s_0^{1+3\rho} \in [-1,1]  . \end{equation} 
Then, by using the definition of $c_0$ in \eqref{def:c0bis}, the definition of $g$ in \eqref{def:gh} and the choice of $\lambda_0$ and $\sigma_0$ in \eqref{def:sigma_lamba:ini}, we compute 
 \begin{equation*} 
 g(s_0)=\mu_0s_0^{-\frac{3-2\theta}{2\theta-1}-3\rho} \in \left[-s_0^{-\frac{3-2\theta}{2\theta-1}-3\rho},s_0^{-\frac{3-2\theta}{2\theta-1}-3\rho}\right]  .
 \end{equation*}
 
 We have assumed that for all $b_0 \in \mathcal{D}_0$, $s^{\star}=s^{\star}(b_0)<+\infty$. Since we have strictly improved \eqref{BS:param}, \eqref{BS:eps} and \eqref{BS:h}, then \eqref{BS:g} must be saturated in $s^{\star}$, which means that 
 \begin{equation*} 
 |g(s^{\star})|=\left(s^{\star}\right)^{-\frac{3-2\theta}{2\theta-1}-3\rho}  .
 \end{equation*}
 Define the function 
 \begin{equation*} 
 \Phi : \mu_0 \in [-1,1] \mapsto g(s^{\star})\left(s^{\star}\right)^{\frac{3-2\theta}{2\theta-1}+3\rho} \in \left\{-1,1\right\}  ,
 \end{equation*}
 where $s^{\star}=s^{\star}(b_0)$ and $b_0$ is given by the correspondence \eqref{def:mu}. Since \eqref{BS:g} is saturated in $s^{\star}$, it is clear that for $\mu_0=-1$, respectively $\mu_0=1$, then $s^{\star}=s_0$ and $\Phi(-1)=-1$, respectively $\Phi(1)=1$. Now, we will prove that $\Phi$ is a continuous function, which will lead to a contradiction and conclude the proof of Proposition \ref{prop:BS}.
 
 We set 
 \begin{equation*} 
 G(s)=\left(g(s)s^{\frac{3-2\theta}{2\theta-1}+3\rho} \right)^2  .
 \end{equation*}
 It is clear that $G(s^{\star})=1$. Moreover we claim the following transversality property for $G$: let $s_1 \in [s_0,s^{\star}]$ such that $G(s_1)=1$; then there exists $c_0>$ such that 
 \begin{equation} \label{transversality}
G_s(s_1) \ge c_0\left(s_1\right)^{-1}  .
\end{equation}
Indeed, we compute 
\begin{equation*} 
G_s(s)=2\left(\frac{3-2\theta}{2\theta-1}+3\rho\right)G(s)s^{-1}+2g(s)g_s(s)s^{2\left(\frac{3-2\theta}{2\theta-1}+3\rho\right)}  ,
\end{equation*}
which yields \eqref{transversality} by choosing $s_0$ large enough and using \eqref{BS:gs} and $G(s_1)=1$, since $3\rho \le \frac14$.

Finally, it remains to show that $:\mu_0 \mapsto s^{\star}$ is continuous, which will then imply easily that $\Phi$ is continuous. Assume first that $\mu_0 \in (-1,1)$, so that $s^{\star}>s_0$, and let $0<\epsilon<s^{\star}-s_0$. Then, the transversality condition \eqref{transversality} and a continuity argument imply that there exists $\delta>0$ such that for $\epsilon$ small enough\footnote{Observe by continuity that the decomposition \eqref{decomp:v} holds on some time interval after $s^{\star}$ so that $g$ is still well defined on $[s^{\star},s^{\star}+\varepsilon_0]$, for $\varepsilon_0>0$ small enough.}
\begin{equation*} 
G(s^{\star}+\epsilon) > 1+\delta \quad \text{and} \quad G(s)<1-\delta, \ \text{for all} \ s \in [s_0,s^{\star}-\epsilon]  .
\end{equation*}
Now, by continuity of the flow associated to \eqref{gkdv}, there exists $\eta>0$ such that for all $\tilde{\mu}_0 \in (-1,1)$ such that $|\mu_0-\tilde{\mu}_0|<\eta$, the corresponding function $\tilde{G}$ satisfies $ | \tilde{G}(s)-G(s) | < \delta/2$ on $[s_0,s^{\star}+\epsilon]$. Then, denoting $\tilde s^{\star}=s^{\star}(\tilde{\mu}_0)$, we deduce that 
\begin{align*}
\tilde{G}(s)<1-\frac{\delta}2, \ \forall s \in [s_0,s^{\star}-\epsilon] & \implies \tilde s^{\star} \ge s^{\star}-\epsilon  ;\\
\tilde{G}(s^{\star}+\epsilon)>1+\frac{\delta}2 & \implies \tilde s^{\star} \le s^{\star}+\epsilon  .
\end{align*}
This proves the continuity of the map $\mu_0 \mapsto s^{\star}$ at any $\mu_0 \in (-1,1)$. In the case where $\mu_0=-1$ or $\mu_0=1$, then $s^{\star}=s_0$, $G(s_0)=1$ and $G_s(s_0)>0$ (from \eqref{transversality}). Then, we conclude by using a similar argument that $:\mu_0 \mapsto s^{\star}$ is also continuous at $\mu_0 =-1$ and $\mu_0=1$.
 
This concludes the proof of Proposition \ref{prop:BS}.
\end{proof}

\subsection{Main result}\label{S:6}
We are now in a position to state the main result of this paper, in its full generality.
Define the constants
\begin{equation}\label{def:cts}
c_\lambda=\left(\frac {3\beta-1}2\right)^{-\frac {1-\beta}2},\quad
c_\sigma=\frac 1\beta \left(\frac {3\beta-1}2\right)^{1-\beta}.
\end{equation}

\begin{theorem}\label{th:2}
There exists $\rho_1>0$ such that for any $x_0$ large enough, the following holds.
Let $t_0=(2x_0)^{\frac 1\beta}$ and let $\sigma_0$ and $\lambda_0$ be such that
\begin{equation*} 
\left| \lambda_0-c_\lambda t_0^{\frac {1-\beta}2} \right| \leq t_0^{\frac {1-\beta}2(1-\rho_1)}
 \quad \text{and} \quad 
 \left|\sigma_0 - c_\sigma t_0^{\beta}\right|\leq t_0^{\beta(1-\rho_1)}
  .
\end{equation*}
Let $\varepsilon_0 \in H^1(\mathbb R)$ be such that 
\begin{equation*}
 \|\varepsilon_0\|_{H^1}^2+\int\varepsilon_0^2e^{\frac{y}{10}} \leq t_0^{-\frac {3\beta-1}{20}-\rho_1} 
\quad \text{and} \quad
(\varepsilon_0,\Lambda Q)=(\varepsilon_0,y\Lambda Q)=(\varepsilon_0,Q)=0  .
\end{equation*}
Then there exists $b_0=b_0(\lambda_0,\sigma_0,\varepsilon_0)$ with
$|b_0|\lesssim t_0^{-\frac 12(3\beta-1)}$
such that the solution $U(t)$ of \eqref{gkdv} corresponding to the following initial data at $t=t_0$
\begin{equation*} 
U(t_0,x)
= \frac 1{\lambda_0^{\frac 12}} \left(Q_{b_0}+\lambda_0^{\frac12}f(t_0,\sigma_0)R+\varepsilon_0\right)
\left(\frac{x-\sigma_0}{\lambda_0}\right) 
+ f(t_0,x) 
\end{equation*}
decomposes as 
\begin{equation}\label{for:U}
U(t,x)=\mm(t,x)+ \eta(t,x),\quad
\mm(t,x)=\frac 1{\lambda^{\frac 12}(t)} Q\left( \frac{x-\sigma(t)}{\lambda(t)}\right),
\end{equation}
where the functions $\lambda(t)$, $\sigma(t)$ and $\eta(t)$ satisfy
\begin{align}
&\left|\lambda-c_\lambda t^{\frac {1-\beta}{2}}\right|
\lesssim t^{ \frac {1-\beta}{2}(1-\rho_1)},\quad
\left|\sigma-c_\sigma t^{\beta}\right|
\lesssim t^{\beta(1-\rho_1)},\nonumber\\
&\|\eta\|_{L^2}\lesssim t_0^{-\frac14(3\beta-1)},\quad
\|\partial_x\eta\|_{L^2}\lesssim t_0^{-\frac{1+\beta}4},\nonumber\\
&\|\eta\|_{L^2(x>\frac12\sigma)}+\|\partial_x \eta\|_{L^2(x>\frac12\sigma)}\lesssim t^{-\frac14(3\beta-1)}  .
\label{est:eta:right_ext}
\end{align}
\end{theorem}
To prove Theorem~\ref{th:2} from Proposition~\ref{prop:BS}, it is sufficient to return to the 
original variables $(t,x)$ (see \S\ref{S.5.3}) and to prove the additional estimate \eqref{est:eta:right_ext} which improves the region where
the residue $\eta$ converges strongly to $0$ (see \S\ref{S.5.4}).

\subsection{Returning to original variables}\label{S.5.3}
In the context of Proposition~\ref{prop:BS} and Theorem~\ref{th:2}, we prove in this subsection the following set of estimates:
\begin{align}
& \left|\lambda-c_\lambda t^{\frac {1-\beta}{2}}\right|
\lesssim t^{ \frac {1-\beta}{2}(1-\rho_1)},
\quad \left|\frac{\lambda_t}{\lambda} - \frac {1-\beta}{2} t^{-1}\right| \lesssim t^{- 1-\frac{1-\beta}2 \rho_1},
\label{sur:lat}\\
& \left|\sigma-c_\sigma t^{\beta}\right|
\lesssim t^{\beta(1-\rho_1)},
\quad \left|\frac{\sigma_t}{\lambda} - \frac 1{\lambda^3}\right| \lesssim t^{- 1-\frac{3\beta-1}8},\label{sur:sit}
\end{align}
and
\begin{align}
& \|\lambda^{\frac 1{20}}\mm ^\frac1{10}\eta\|_{L^2}\lesssim t^{-\frac12(3\beta-1)},\quad
\|\lambda^{\frac 1{20}}\mm ^\frac1{10}\partial_x \eta\|_{L^2}\lesssim t^{-\beta},\quad
\|\lambda^{\frac 1{20}}\mm ^\frac1{10}\eta\|_{L^\infty}\lesssim t^{-\frac14(5\beta-1)},\label{eq:eta:loc}\\
& \|\eta\|_{L^2(x>\sigma)}\lesssim t^{-\frac12(3\beta-1)},\quad
\|\partial_x \eta\|_{L^2(x>\sigma)}\lesssim t^{-\beta},
\label{eq:eta:right}\\
& \|\eta\|_{L^2}\lesssim t_0^{-\frac14(3\beta-1)},\quad
\|\partial_x\eta\|_{L^2}\lesssim t_0^{-\frac{1+\beta}4},\label{eq:eta}\end{align}
where $\rho_1$ is a small positive number.

\begin{proof}[Proof of \eqref{sur:lat}-\eqref{eq:eta}]
First, we relate $t$ and $s$ from \eqref{eq:sbis}. We claim that for any $t\geq t_0$, $s\geq s_0$,
\begin{equation}\label{sur:t:s}
\left|t-\frac {2\theta-1}{5-4\theta} s^{\frac {5-4\theta}{2\theta-1}}\right|\lesssim \left(1-\left(\frac {s_0}s\right)^{\frac {5-4\theta}{2\theta-1}-\rho}\right) s^{\frac {5-4\theta}{2\theta-1}-\rho}.
\end{equation}
Indeed, from \eqref{BS:param} and $dt=\lambda^3(s) ds$, one has 
($\rho$ is defined in \eqref{def:rho})
\begin{align*}
t-t_0=\int_{s_0}^s \lambda^3(s) ds 
&\geq \int_{s_0}^s (s')^{\frac{6(1-\theta)}{2\theta-1}} ds'
-c \int_{s_0}^s (s')^{\frac{6(1-\theta)}{2\theta-1}-\rho} ds'\\
&\geq \frac{2\theta-1}{5-4\theta}\left( s^{\frac {5-4\theta}{2\theta-1}}-s_0^{\frac {5-4\theta}{2\theta-1}}\right)
-c \left( s^{\frac {5-4\theta}{2\theta-1}-\rho}-s_0^{\frac {5-4\theta}{2\theta-1}-\rho}\right),
\end{align*}
and thus using~\eqref{eq:sbis},
\begin{equation*}
t-\frac{2\theta-1}{5-4\theta} s^{\frac {5-4\theta}{2\theta-1}}
\geq -c \left( 1- \frac {s_0}{s}\right) s^{\frac {5-4\theta}{2\theta-1}-\rho}.
\end{equation*}
This proves the lower bound in \eqref{sur:t:s}.
The corresponding upper bound is proved similarly.
Note that
\begin{equation}\label{t:s:beta}
\frac {2\theta-1}{5-4\theta}=\frac{3\beta-1}{2} \quad\mbox{so that}\quad
\left|t-\frac{3\beta-1}{2} s^{\frac2{3\beta-1}}\right|\lesssim \left(1-\left(\frac {s_0}s\right)^{\frac {5-4\theta}{2\theta-1}-\rho}\right) s^{\frac2{3\beta-1}-\rho}.
\end{equation}

Observing that
\begin{equation*}
c_\lambda = \left( \frac {5-4\theta}{2\theta-1}\right)^{\frac {2(1-\theta)}{5-4\theta}},
\end{equation*}
it follows from \eqref{BS:param} and \eqref{sur:t:s} that 
\begin{equation*}
\left|\lambda(t)-c_\lambda t^{\frac {2(1-\theta)}{5-4\theta}}\right|
\lesssim t^{\frac {2(1-\theta)}{5-4\theta}(1-\rho_1)}\end{equation*}
holds with $0<\rho_1:=\frac {2\theta-1}{2(1-\theta)} \rho<\frac 16$.
Moreover, we check that \eqref{BS:param} and \eqref{BS:m} imply
\begin{equation*}
\left| \frac {\lambda_t (t)}{\lambda(t)} - \frac {2(1-\theta)}{5-4\theta} \frac 1{t} \right|
\lesssim t^{-(1+\frac{2(1-\theta)}{5-4\theta}\rho_1)}.
\end{equation*}
Using $\beta=\frac 1{5-4\theta}$, one finds \eqref{sur:lat}.

Observing that
\begin{equation*}
c_\sigma = (5-4\theta)^{\frac 1{5-4\theta}} (2\theta-1)^{ \frac{4(1-\theta)}{5-4\theta}},
\end{equation*}
it follows from \eqref{BS:param} and \eqref{sur:t:s} that
\begin{equation*}
\left|\sigma(t)-c_\sigma t^{\frac {1}{5-4\theta}}\right|
\lesssim t^{\frac {1}{5-4\theta}(1-\rho_1)}.
\end{equation*}
Moreover, by \eqref{BS:param} and \eqref{BS:m}, it holds
\begin{equation*}
\left|\frac{\sigma_t}{\lambda}-\frac{1}{\lambda^3}\right|=
\left|\frac 1{\lambda^3} \left( \frac{\sigma_s}{\lambda}-1\right)\right|
\lesssim t^{-\frac{19-14\theta}{4(5-4\theta)}}.
\end{equation*}
Thus, \eqref{sur:sit} holds recalling that $\beta=\frac1{5-4\theta}$.

Last, we control $\eta$. Note that from \eqref{def:v} and \eqref{decomp:v}
\begin{equation*}
\lambda^{\frac 12} \eta\left(t,\lambda(t) y +\sigma\right)=
F(t,y) + b(t) P_{b(t)}(y) + r(t) R(y)+ \varepsilon(t, y).
\end{equation*}
Thus, the following estimates hold
\begin{align*}
 \|\eta(t)\|_{L^2}&\lesssim \|f(t)\|_{L^2} + |b(t)| \|P_{b(t)}\|_{L^2}
+|r(t)| + \|\varepsilon(t)\|_{L^2},\\
\|\partial_x \eta(t)\|_{L^2}&\lesssim \|\partial_x f(t)\|_{L^2}+\lambda^{-1}(t) \left[ 
|b(t)| \|\partial_y P_{b(t)}\|_{L^2}
+|r(t)| + \|\partial_y \varepsilon(t)\|_{L^2}\right],
\\
 \lambda^{\frac 1{20}}\|\mm ^\frac1{10}(t)\eta(t)\|_{L^2}&\lesssim \|Q^\frac1{10} F(t)\|_{L^2} + |b(t)| \|Q^\frac1{10}P_{b(t)}\|_{L^2}
+|r(t)| + \|Q^\frac1{10}\varepsilon(t)\|_{L^2},\\
\lambda^{\frac 1{20}}\|\mm ^\frac1{10}(t)\partial_x \eta(t)\|_{L^2}&\lesssim \lambda^{-1}(t) \left[ \|Q^\frac1{10}\partial_y F(t)\|_{L^2}+
|b(t)| \|Q^\frac1{10}\partial_y P_{b(t)}\|_{L^2}
\right.\\&\quad \left.+|r(t)| + \|Q^\frac1{10}\partial_y \varepsilon(t)\|_{L^2}\right].
\end{align*}
From \eqref{est:f_0}, \eqref{bound:df:dx} and $x_0= (2t_0)^{\beta}$, 
\begin{align*}
&\sup_{t\in \mathbb R}\|f(t)\|_{L^2}=\|f_0\|_{L^2}\lesssim x_0^{-\frac 12 (2\theta-1)}
\lesssim t_0^{-\frac{3\beta-1}4},\\
&\sup_{t\in \mathbb R}\|\partial_x f(t)\|_{L^2}\lesssim x_0^{-\frac 12 (2\theta+1)}
\lesssim t_0^{-\frac{7\beta-1}4}.
\end{align*}
From \eqref{e:F} and \eqref{t:s:beta},
\begin{equation*}
\|Q^\frac1{10} F(t)\|_{L^2}\lesssim t^{-\frac{3\beta-1}2},
\quad
\|Q^\frac1{10}\partial_x F(t)\|_{L^2}\lesssim t^{-(3\beta-1)}.
\end{equation*}
From \eqref{def:Pb}, \eqref{BS:param} and \eqref{t:s:beta}, it holds
\begin{align*}
& \|b(t)P_{b(t)}\|_{L^2}
\lesssim |b(t)|^{\frac 58}\lesssim t^{-\frac 5{16}(3\beta-1)},
\\&
 \|b(t)\partial_y P_{b(t)}\|_{L^2}+|b(t)| \|Q^\frac1{10}P_{b(t)}\|_{L^2}+ |b(t)| \|Q^\frac1{10}\partial_y P_{b(t)}\|_{L^2}\lesssim |b(t)|\lesssim t^{-\frac {3\beta-1}2}.
\end{align*}
From \eqref{e:r} and \eqref{t:s:beta}, it holds $|r(t)|\lesssim t^{-\frac 12(3\beta-1)}$.
Last, from \eqref{BS:eps}, 
\begin{equation*}
\|Q^\frac1{10}\varepsilon(t)\|_{L^2}+\|Q^\frac1{10}\partial_y \varepsilon(t)\|_{L^2}\lesssim 
t^{-\frac 5{8}(3\beta-1)}.
\end{equation*}
This implies the first two estimates in \eqref{eq:eta:loc}. For the last estimate in \eqref{eq:eta:loc}, we use
\begin{align*}
\|\lambda^{\frac 1{20}}\mm ^\frac1{10}\eta\|_{L^\infty}^2&
\lesssim\|\lambda^{\frac 1{20}}\mm ^\frac1{10}\eta\|_{L^2}\|\lambda^{\frac 1{20}}\partial_x[\mm ^\frac1{10}\eta]\|_{L^2}\\
&\lesssim\|\lambda^{\frac 1{20}}\mm ^\frac1{10}\eta\|_{L^2}\left(\|\lambda^{\frac 1{20}}\mm ^\frac1{10}\partial_x\eta\|_{L^2}
+\|\lambda^{\frac 1{20}}\partial_x[\mm ^\frac1{10}]\eta\|_{L^2}\right)\\
&\lesssim\|\lambda^{\frac 1{20}}\mm ^\frac1{10}\eta\|_{L^2}\left(\|\lambda^{\frac 1{20}}\mm ^\frac1{10}\partial_x\eta\|_{L^2}
+\lambda^{-1}\|\lambda^{\frac 1{20}}\mm ^\frac1{10}\eta\|_{L^2}\right)
\lesssim t^{-\frac12(5\beta-1)}.
\end{align*}
Similarly, \eqref{eq:eta:right} follows from \eqref{e:FLinfty}, the properties of $P_b$ and \eqref{BS:eps}.
 
Now, we estimate $\|\varepsilon\|_{L^2}$ and $\|\partial_y \varepsilon\|_{L^2}$.
For this, the local estimates~\eqref{BS:eps} involved in the bootstrap are not sufficient, and we have to use global mass and energy estimates from Lemma~\ref{le:3.9}. First, we compute $\int u_0^2$. Using the computations of the proof of Lemma~\ref{le:3.9}
and~\eqref{mass:W},
\begin{equation*}
\int u_0^2 -\int Q^2 
= 2 b_0 \int PQ + \frac 12 r_0 \int Q +\mathcal O(\|\varepsilon_0\|_{L^2}^2)
+\mathcal O(s_0^{-\frac 54}) + \mathcal O(x_0^{-(2\theta-1)}).
\end{equation*}
By the above estimates taken at $t=t_0$ and $\|\varepsilon_0\|_{L^2}\lesssim s_0^{-10} \lesssim t_0^{-5(3\beta-1)}$
(see \eqref{def:eps:ini})
we find
\begin{equation*}
\left|\int u_0^2-\int Q^2\right| 
\lesssim t_0^{-\frac 12 (3\beta-1)}.
\end{equation*}
Thus, by \eqref{mass:eps}, we find, for all $t\geq t_0$, 
$\|\varepsilon\|_{L^2} \lesssim t_0^{-\frac 14 (3\beta-1)}$
and the above estimates imply $\|\eta\|_{L^2} \lesssim t_0^{-\frac14 (3\beta-1)}$.

Second, we compute $E(u_0)$. Using the computations of the proof of Lemma~\ref{le:3.9}
and \eqref{ener:W},
\begin{equation*}
E(u_0) = \mathcal O(\lambda^{-2}(t_0) \|\varepsilon_0\|_{H^1}^2) + \mathcal O(\lambda^{-2}(t_0) s_0^{-2})
+\mathcal O(x_0^{-(2\theta+1)})+\mathcal O(|g(s_0)|).
\end{equation*}
Note that by \eqref{BS:g} and \eqref{t:s:beta}, it holds $|g(s_0)|\lesssim t_0^{-\frac 12 (1+\beta)}$.
We deduce from \eqref{def:eps:ini} and the previous estimates that
$|E(u_0)|\lesssim t_0^{-\frac 12 (1+\beta)}$.
Thus, by \eqref{ener:eps}, we find, for all $t\geq t_0$,
$\lambda^{-1} \|\partial_y\varepsilon\|_{L^2} \lesssim t_0^{-\frac 14 (1+\beta)}$, and the above estimates imply 
$\|\partial_x \eta\|_{L^2} \lesssim t_0^{-\frac 14 (1+\beta)}$.
\end{proof}
 
\subsection{Additional monotonicity argument.}\label{S.5.4}
To complete the proof of Theorem~\ref{th:2}, we prove \eqref{est:eta:right_ext} by extending the local estimates \eqref{eq:eta:right} for $\eta$ on the right of the soliton to the larger region $x>\frac12\sigma$. 
Write the equation for $\eta$ as follows
\begin{equation}\label{equ:eta}
\partial_t \eta+\partial_x^3 \eta
 =-\partial_xN_1 + N_2
\end{equation}
where
\begin{align*} 
N_1&=(\mm +\eta)^5-\mm ^5,\\
N_2&=\frac{\lambda_t}{\lambda}\frac 1{\lambda^{\frac 12}} \Lambda Q\left( \frac{\cdot-\sigma}{\lambda}\right)
+\left(\frac{\sigma_t}\lambda-\frac 1{\lambda^3}\right) \frac 1{\lambda^{\frac 12}} Q'\left( \frac{\cdot-\sigma}{\lambda}\right).
\end{align*}
Fix $t>t_0$ and for any $\tau\in [t_0,t]$, let
\begin{align*}
J(\tau)&=\int \eta^2(\tau,x)\xi(\tau,x) dx  , \\ 
K(\tau)&=\int \left[ (\partial_x\eta)^2-\frac13\left((\mm +\eta)^6-\mm ^6-6\mm ^5\eta\right)\right](\tau,x)\xi(\tau,x) dx , 
\end{align*}
where $ \xi(\tau,x)=\chi\left(\frac{4x-\sigma(t)}{\sigma(\tau)}-2\right)$ and $\chi$ is defined in \eqref{def:chi}.

\begin{lemma}
For $t_0$ large enough and for all $\tau\in [t_0,t]$, it holds
\begin{equation}\label{on:Js}
\frac{dJ}{d\tau}(\tau)\lesssim \tau^{-\frac 12(3\beta+1)},
\end{equation}
and
\begin{equation}\label{on:Ks}
\frac{dK}{d\tau}(\tau)\lesssim \tau^{-\frac 12(3\beta+1)}.
\end{equation}
\end{lemma}
\begin{proof}
We begin with the proof of \eqref{on:Js}. Using \eqref{equ:eta}, we observe after integration by parts
\begin{align*}
\frac{dJ}{d\tau} &= 2\int (\partial_\tau\eta)\eta\xi+\int \eta^2\partial_\tau\xi 
 \\&=-3 \int (\partial_x \eta)^2 \partial_x \xi+\int \eta^2\partial_\tau\xi +\int \eta^2 \partial_x^3 \xi
-2 \int (\partial_xN_1) \eta \xi+2\int N_2 \eta\xi \\ 
& =:J_1+J_2+J_3+J_4+J_5  .
\end{align*}

We compute
\begin{align*}
\partial_x \xi(\tau,x) &= \frac 4{\sigma(\tau)} \chi'\left(\frac{4x-\sigma(t)}{\sigma(\tau)}-2\right) \geq 0,\\
\partial_x^3 \xi(\tau,x) &=\left(\frac 4{\sigma(\tau)}\right)^3 \chi'''\left(\frac{4x-\sigma(t)}{\sigma(\tau)}-2\right),\\
\partial_\tau \xi(\tau,x) &=-\frac{\sigma_t(\tau)}{\sigma(\tau)}\left(\frac{4x-\sigma(t)}{\sigma(\tau)}\right) \chi'\left(\frac{4x-\sigma(t)}{\sigma(\tau)}-2\right).
\end{align*}
Note that by \eqref{sur:sit},
\begin{equation} \label{est:xi3}
0 \le \partial_x\xi \lesssim \tau^{-\beta}, \quad \text{and} \quad |\partial_x^3 \xi|\lesssim \tau^{-3\beta}\lesssim\tau^{-\frac {3\beta+1}2}  .
\end{equation}
Observe that $J_1 \le 0$. Since $\sigma_t\geq 0$ (see \eqref{sur:sit}), $\chi'\geq 0$ on $\mathbb{R}$ and $\chi'(x)=0$ for $x<-2$, we also have $\partial_\tau \xi\leq 0$, so that $J_2 \le 0$. Moreover, by using \eqref{eq:eta} and \eqref{est:xi3}, we have that 
$\left|J_3\right| \lesssim \tau^{-\frac 12 (3\beta+1)}$.
On the other hand, more integration by parts yield
\begin{align*}
J_4
& = 2 \int \left( (\mm +\eta)^5-\mm ^5\right) (\partial_x\eta\, \xi+\eta \, \partial_x\xi)\\
& = - \frac 13 \int \left[ (\mm +\eta)^6-\mm ^6-6\mm ^5\eta\right]\partial_x \xi
+2 \int \left( (\mm +\eta)^5-\mm ^5\right) \eta \partial_x\xi\\
&\quad -2\int \left[ (\mm +\eta)^5-\mm ^5-5\mm ^4\eta\right] (\partial_x \mm ) \xi\\
&=\int \left(5\mm ^4\eta^2+\frac {40}3 \mm ^3\eta^3+15\mm ^2\eta^4+8\mm \eta^5+\frac 53 \eta^6\right)\partial_x \xi\\
&\quad -2\int \left[ (\mm +\eta)^5-\mm ^5-5\mm ^4\eta\right] (\partial_x \mm ) \xi  ,
\end{align*}
so that
\begin{equation*}
\left|J_4\right|\lesssim \int \eta^6 \partial_x \xi +\int \eta^2 \mm ^4\partial_x \xi+\int \eta^2 |\mm |^3|\partial_x\mm |+\int |\eta|^5 |\partial_x\mm |.
\end{equation*}
For the first term on the right-hand side of the above estimate, we argue as in \eqref{est:infty} to deduce 
\begin{equation*} \left\|\eta^2\sqrt{\partial_x\xi}\right\|_{L^\infty}^2
\lesssim \|\eta\|_{L^2}^2\int (\partial_x\eta)^2 \partial_x\xi
 +\|\eta\|_{L^2}^2\int \eta^2\frac{\big(\partial_x^2\xi\big)^2}{\partial_x\xi}  ,
\end{equation*}
in the support of $\partial_x\xi$. Thus, by using in this region
\begin{equation}\label{A-VOIR}
\frac{ (\partial_x^2\xi )^2}{\partial_x\xi} \lesssim \sigma^{-3} \frac{(\chi'')^2}{\chi'} \lesssim \sigma^{-3},
\end{equation}
we deduce from \eqref{eq:eta}, \eqref{est:xi3} that
\begin{equation} \label{full:nonlin:eta}
\int \eta^6 \partial_x \xi\lesssim 
\left(\int \eta^2\right)^{2} \left[ \int (\partial_x \eta)^2 \partial_x \xi +  \sigma^{-3} \int \eta^2\right] \le -\frac12J_1+c\tau^{-\frac {3\beta+1}2}  , 
\end{equation}
by taking $t_0$ large enough. 
For the second term, using \eqref{eq:eta:loc} and \eqref{est:xi3}, we have
\begin{equation*}
\int \eta^2 \mm ^4\partial_x \xi
\lesssim \sigma^{-1}\lambda^{-2} \|\lambda^{\frac 1{20}} \mm ^{\frac 1{10}}\eta\|_{L^2}^2\lesssim 
\tau^{-3\beta} \lesssim\tau^{-\frac {3\beta+1}2}.
\end{equation*}
Next, using \eqref{eq:eta:loc}
\begin{equation*}
\int \eta^2 |\mm |^3|\partial_x\mm |
\lesssim \lambda^{-3} \|\lambda^{\frac 1{20}} \mm ^{\frac 1{10}}\eta\|_{L^2}^2
\lesssim \tau^{-\frac {3\beta+1}2},
\end{equation*}
and
\begin{equation*}
\int |\eta|^5 |\partial_x\mm |
\lesssim \lambda^{-\frac 32} \|\lambda^{\frac 1{20}} \mm ^{\frac 1{10}}\eta\|_{L^\infty}^3\|\lambda^{\frac 1{20}} \mm ^{\frac 1{10}}\eta\|_{L^2}^2
\lesssim \tau^{1-6\beta}\lesssim \tau^{-\frac {3\beta+1}2}.
\end{equation*}
Gathering all these estimates, we obtain that 
\begin{equation*} J_4 \le -\frac12 J_1+c\tau^{-\frac {3\beta+1}2}  .\end{equation*}
Last, we observe by \eqref{sur:lat}-\eqref{sur:sit} that $|N_2(\tau,x)|\lesssim \tau^{-1} \lambda^{-\frac14}\mm ^{\frac12}(\tau,x)$, and so
\begin{equation*}
\left|J_5\right|\lesssim
\tau^{-1} \lambda^{-\frac14}\int \mm ^{\frac12} |\eta|\lesssim \tau^{-1}\lambda^{-\frac12}\|\lambda^{\frac 1{20}}\mm ^{\frac 1{10}} \eta\|_{L^2}
\lesssim \tau^{-\frac {3\beta+1}2}.
\end{equation*}
Collecting these estimates and taking $t_0$ large enough, we have proved \eqref{on:Js}.

\smallskip

Now, we turn to the proof of \eqref{on:Ks}. We compute using \eqref{equ:eta} and integrating by parts 
\begin{align*}
\frac{dK}{d\tau} &= \int \left[(\partial_x\eta)^2-\frac13\left((\mm +\eta)^6-\mm ^6-6\mm ^5\eta \right) \right]\partial_{\tau}\xi 
+2\int \partial_x\eta \partial_{x\tau}^2\eta \xi\\ & \quad 
-2\int \left[\left((\mm +\eta)^5-\mm ^5\right)\partial_{\tau}(\mm +\eta)-5\mm ^4\partial_{\tau}\mm \eta \right] \xi
\\
&=-2 \int \left[\partial_x^2\eta+(\mm +\eta)^5-\mm ^5 \right]^2 \partial_x \xi-2\int (\partial_x^2\eta)^2 \partial_x\xi+\int (\partial_x\eta)^2\partial_x^3\xi\\ 
& \quad +2\int \partial_xN_1 \partial_x \eta\partial_x\xi +2\int \partial_xN_2 \partial_x\eta \xi-2\int N_2 N_1 \xi\\
& \quad -2\int \left[(\mm +\eta)^5-\mm ^5-5\mm ^4\eta \right] \partial_{\tau}\mm  \xi \\
&\quad + \int \left[(\partial_x\eta)^2-\frac13\left((\mm +\eta)^6-\mm ^6-6\mm ^5\eta \right) \right]\partial_{\tau}\xi \\
& =:K_1+K_2+K_3+K_4+K_5+K_6+K_7+K_8  .
\end{align*}
Observe, since $\partial_x \xi \ge 0$, that $K_1 \le 0$ and $K_2 \le 0$. Moreover, it follows from \eqref{eq:eta} and \eqref{est:xi3} that $K_3 \lesssim \tau^{-3\beta}$. Moreover, we compute 
\begin{equation*} 
 K_4 = 10\int (\mm +\eta)^4(\partial_x\eta)^2 \partial_x\xi+10\int \left[(\mm +\eta)^4-\mm ^4 \right]\partial_x\mm \partial_x\eta\partial_x\xi
\end{equation*}
so that
\begin{align*}
 K_4& \lesssim \int \mm ^4(\partial_x\eta)^2 \partial_x\xi+\int \eta^4(\partial_x\eta)^2 \partial_x\xi+\int \left|\mm ^3\partial_x\mm \eta\partial_x\eta\right|\partial_x\xi
 +\int \eta^4(\partial_x\mm )^2 \partial_x\xi \\ 
 &=:K_{4,1}+K_{4,2}+K_{4,3}+K_{4,4}  .
\end{align*}
By using \eqref{sur:lat}, \eqref{sur:sit}, \eqref{eq:eta:loc} and \eqref{est:xi3}, we deduce that 
\begin{align*} 
|K_{4,1}| &\lesssim \sigma^{-1}\lambda^{-2}\|\lambda^{\frac 1{20}} \mm ^{\frac 1{10}}\partial_x\eta\|_{L^2}^2 \lesssim \tau^{-1-2\beta} \\ 
|K_{4,3}| &\lesssim \sigma^{-1}\lambda^{-3}\|\lambda^{\frac 1{20}} \mm ^{\frac 1{10}}\eta\|_{L^2}\|\lambda^{\frac 1{20}} \mm ^{\frac 1{10}}\partial_x\eta\|_{L^2} \lesssim \tau^{-1-2\beta} \\ 
|K_{4,4}| &\lesssim \sigma^{-1}\lambda^{-3}\|\lambda^{\frac 1{20}} \mm ^{\frac 1{10}}\eta\|_{L^{\infty}}^2\|\lambda^{\frac 1{20}} \mm ^{\frac 1{10}}\eta\|_{L^2}^2\lesssim \tau^{-\frac{7\beta+1}2} \lesssim \tau^{-1-2\beta}  ,
\end{align*}
since $\beta>\frac13$. To handle the purely nonlinear term $K_{4,2}$, we use, arguing as in \eqref{Sobo:firstderiv}, the improved Sobolev estimate
\begin{equation*} \left\|\eta\partial_x\eta\sqrt{\partial_x\xi}\right\|_{L^\infty}^2
\lesssim \|\eta\|_{L^2}^2\int (\partial_x^2\eta)^2 \partial_x\xi
 +\|\eta\|_{L^2}^2\int (\partial_x\eta)^2\frac{\big(\partial_x^2\xi\big)^2}{\partial_x\xi}  ,
\end{equation*}
in the support of $\partial_x\xi$.
Thus, using \eqref{A-VOIR}, we deduce from \eqref{eq:eta}, \eqref{est:xi3} that
\begin{equation*}
K_{4,2}\lesssim 
\left(\int \eta^2\right)^{2} \left[ \int (\partial_x^2 \eta)^2 \partial_x \xi + \sigma^{-3} \int \eta^2\right] \le -\frac12K_2+c\tau^{-3\beta}  , 
\end{equation*}
by taking $t_0$ large enough. 

Observe from \eqref{sur:lat}-\eqref{sur:sit} that $|\partial_xN_2(\tau,x)|\lesssim \tau^{-1}\lambda^{-\frac54} \mm ^{\frac12}(\tau,x)$. Thus, it follows using \eqref{sur:lat} and \eqref{eq:eta:loc} that 
\begin{equation*} 
|K_5| \lesssim \tau^{-1}\lambda^{-\frac32}\|\lambda^{\frac 1{20}} \mm ^{\frac 1{10}}\partial_x\eta\|_{L^2} \lesssim \tau^{-\frac{7+\beta}4} \lesssim \tau^{-\frac{3\beta+1}2}, 
\end{equation*}
since $\beta \le 1$. To estimate $K_6$, we observe 
\begin{equation*} 
|K_6| \lesssim \int \mm ^4|\eta| |N_2|+\int |\eta|^5 |N_2| =:K_{6,1}+K_{6,2}  ,
\end{equation*}
so that it follows from $|N_2(\tau,x)|\lesssim \tau^{-1} \lambda^{-\frac14}\mm ^{\frac12}(\tau,x)$ and \eqref{sur:lat}-\eqref{eq:eta:loc}, 
\begin{align*} 
|K_{6,1}| &\lesssim \tau^{-1}\lambda^{-\frac52}\|\lambda^{\frac 1{20}} \mm ^{\frac 1{10}}\partial_x\eta\|_{L^2} \lesssim \tau^{-\frac{7+\beta}4} \lesssim \tau^{-\frac{3\beta+1}2} , \\
|K_{6,2}| &\lesssim \tau^{-1}\lambda^{-\frac12}\|\lambda^{\frac 1{20}} \mm ^{\frac 1{10}}\partial_x\eta\|_{L^{\infty}}^3\|\lambda^{\frac 1{20}} \mm ^{\frac 1{10}}\partial_x\eta\|_{L^2}^2 \lesssim \tau^{-\frac{1+17\beta}4} \lesssim \tau^{-1-2\beta}  ,
\end{align*}
since $\beta>\frac13$. To control the contribution $K_7$, we observe from \eqref{sur:lat}-\eqref{sur:sit}
\begin{equation*} 
\partial_t\mm =-\frac{\lambda_t}{\lambda} \lambda^{-\frac12} \Lambda Q\left(\frac{\cdot-\sigma}{\lambda}\right)-\frac{\sigma_t}{\lambda}\lambda^{-\frac12}Q'\left(\frac{\cdot-\sigma}{\lambda}\right) \quad \text{so that} \quad |\partial_t\mm (\tau,x)| \lesssim \lambda^{-3}\mm (\tau,x)  .
\end{equation*}
Thus, 
\begin{equation*} 
|K_7| \lesssim \lambda^{-3} \int \mm ^4\eta^2+ \lambda^{-3}\int |\mm ||\eta|^5=: K_{7,1} +K_{7,2}  .
\end{equation*}
Hence, it follows from \eqref{sur:lat}-\eqref{eq:eta:loc} that 
\begin{align*} 
|K_{7,1}| &\lesssim \lambda^{-5}\|\lambda^{\frac 1{20}} \mm ^{\frac 1{10}}\partial_x\eta\|_{L^2}^2 \lesssim \tau^{-\frac{3+\beta}2} \lesssim \tau^{-\frac{3\beta+1}2} , \\
|K_{7,2}| &\lesssim \lambda^{-3}\|\lambda^{\frac 1{20}} \mm ^{\frac 1{10}}\partial_x\eta\|_{L^{\infty}}^3\|\lambda^{\frac 1{20}} \mm ^{\frac 1{10}}\partial_x\eta\|_{L^2}^2 \lesssim \tau^{-\frac{21\beta-1}4} \lesssim \tau^{-\frac{3\beta+1}2} ,
\end{align*}
since $\frac13<\beta\le 1$. Finally, we deal with $K_8$ by writing
\begin{equation*} 
K_8 \le \int (\partial_x\eta)^2\partial_{\tau}\xi +c\int \mm ^4\eta^2\left|\partial_{\tau}\xi\right|+c\int \eta^6 \left|\partial_{\tau}\xi\right| =: K_{8,1}+K_{8,2}+K_{8,3}  .
\end{equation*}
Observe that $K_{8,1} \le 0$, since $\partial_{\tau}\xi \le 0$. Moreover, we have $0 \le \frac{4x-\sigma(t)}{\sigma(\tau)} \le 1$ on the support of $\partial_{\tau}\xi$, so that $|\partial_{\tau}\xi(\tau,x)| \lesssim \tau^{-1}$ (from \eqref{sur:lat}-\eqref{sur:sit}). Then, it follows from \eqref{sur:lat}, \eqref{eq:eta:loc} that 
\begin{equation*} 
|K_{8,2}| \lesssim \tau^{-1}\lambda^{-2}\|\lambda^{\frac 1{20}} \mm ^{\frac 1{10}}\partial_x\eta\|_{L^2}^2 \lesssim \tau^{-1-2\beta}  .
\end{equation*}
To deal with the fully nonlinear term $K_{8,3}$, we get arguing as in \eqref{full:nonlin:eta} and using \eqref{sur:lat}, \eqref{eq:eta} and 
\begin{equation}\label{A-VOIR2}
\frac{\big(\partial_{x\tau}^2\xi\big)^2}{|\partial_{\tau}\xi|} \lesssim \left| \frac{\sigma_t}{\sigma} \right| \sigma^{-2} \left(  \bar{x}^{-1}\chi'(\bar{x}-2)+\bar{x}\frac{(\chi'')^2}{\chi'} \right) \lesssim \tau^{-1}\sigma^{-2}
\end{equation}
on the support of $\partial_{\tau}\xi$ ($0 \le \bar{x}:=\frac{4x-\sigma(t)}{\sigma(\tau)} \le 1$),  
that 
\begin{equation*} 
K_{8,3} \lesssim 
\left(\int \eta^2\right)^{2} \left[ \int (\partial_x \eta)^2 \left|\partial_{\tau} \xi\right| + \tau^{-1}\sigma^{-2} \int \eta^2\right] \le -\frac12K_{8,1}+c\tau^{-1-2\beta}  , 
\end{equation*}
for $t_0$ large enough. 

We complete the proof of \eqref{on:Ks} by combining all those estimates. 
\end{proof}

\begin{proof}[Proof of \eqref{est:eta:right_ext}]
We define $\underline t$ such that $\sigma(\underline t)=\frac 14 \sigma(t)$.
Note that $\underline t\gtrsim t$ from \eqref{sur:sit}.
Next, we integrate \eqref{on:Js} on $[\underline t,t]$ and we also use $\chi=0$ on $(-\infty,-2]$ and $\chi=1$ on $[-1,+\infty)$.
We obtain 
\begin{equation*}
\int_{x\geq \frac{\sigma(t)}2} \eta^2(t,x) dx
\leq J(t)\leq J(\underline t)+c t^{-\frac {3\beta-1}2} 
\lesssim \int_{x\geq \sigma(\underline t)} \eta^2(\underline t,x) dx + t^{-\frac {3\beta-1}2}
\lesssim t^{-\frac {3\beta-1}2}  ,
\end{equation*}
where we used \eqref{eq:eta:right} in the last step. This implies the first estimate in \eqref{est:eta:right_ext}. Arguing similarly for $K$, we deduce that 
\begin{equation*}
\int_{x\geq \frac{\sigma(t)}2} \left[ (\partial_x\eta)^2-\frac13\left((\mm +\eta)^6-\mm ^6-6\mm ^5\eta\right)\right](t,x) dx
\leq K(t)\leq K(\underline t)+c t^{-\frac {3\beta-1}2} 
\lesssim t^{-\frac {3\beta-1}2}  .\end{equation*}
Moreover, we deduce by using the Sobolev embedding, \eqref{eq:eta:right} and \eqref{eq:eta} that
\begin{align*}
\int_{x\geq \frac{\sigma(t)}2} \left((\mm +\eta)^6-\mm ^6-6\mm ^5\eta\right)
&\lesssim \int \mm ^4\eta^2+\left(\int \eta^2\right)^2\int_{x\geq \frac{\sigma(t)}2}(\partial_x\eta)^2 \\ 
&\le \frac12\int_{x\geq \frac{\sigma(t)}2}(\partial_x\eta)^2+ct^{-(3\beta-1)}  ,
\end{align*}
by choosing $t_0$ large enough. Therefore, we complete the proof of the second estimate in~\eqref{est:eta:right_ext} by combining the last two estimates. 
\end{proof}

\subsection{Proof of Theorem~\ref{th:1}}
The statement of Theorem~\ref{th:1} in the Introduction corresponds to a simplification of Theorem~\ref{th:2}
and to a further rescaling and translation to consider initial data at $t=0$ and close to the soliton $Q$.

Take $x_0$, $t_0$, $\lambda_0$, $\sigma_0$, $b_0$, $U_0$ and a solution $U(t)$ of \eqref{gkdv} 
as in Theorem~\ref{th:2}.
Define the following rescaled version of $U_0$
\begin{equation*}
u_0(x)=
\lambda_0^{\frac 12} U_0( \lambda_0 x+ \sigma_0)=
Q_{ b_0}(x)+ \lambda_0^{\frac 12}f(t_0, \sigma_0) R(x)
+ \lambda_0^{\frac 12} f (t_0, \lambda_0 x+ \sigma_0)
 \end{equation*}
and consider $u(t)$ the solution of \eqref{gkdv} with initial data $u(0)= u_0$, so that
\begin{equation*}
u(t,x)=\lambda_0^{\frac 12} U(\lambda_0^3 t+t_0,\lambda_0 x+ \sigma_0).
\end{equation*}
Let (see \eqref{eq:sbis})
\begin{equation*}
T_0= t_0 \lambda_0^{-3}= c_\lambda^{-3} t_0^{\frac 12 (3\beta -1)}
=\frac{3\beta-1}2 s_0.
\end{equation*}
The decomposition \eqref{for:U} rewrites
\begin{equation*}
u(t,x) =\frac 1{\ell^{\frac 12}(t)} Q\left( \frac{x-x(t)}{\ell(t)}\right)+ w(t,x)
\end{equation*}
where
\begin{equation*}
\ell(t)=\frac{\lambda(\lambda_0^3 t+t_0)}{\lambda_0},\quad
x(t)=\frac{\sigma(\lambda_0^3 t+t_0)-\sigma_0}{\lambda_0},
\quad
w(t,x)=\lambda_0^{\frac 12} \eta(\lambda_0^3 t+t_0,\lambda_0 x+ \sigma_0).
\end{equation*}
First, as a consequence of \eqref{sur:sit}, we have
\begin{equation*}
\left|x(t)-c_\sigma \frac{(\lambda_0^3 t+t_0)^{\beta}-t_0^\beta} {\lambda_0} \right|
\lesssim \frac{(\lambda_0^3 t+t_0)^{\beta(1-\rho_1)}}{\lambda_0}
.
\end{equation*}
Since
\begin{equation*}
c_\sigma t_0^\beta \lambda_0^{-1}
=\beta^{-1} c_\lambda^{-3} t_0^{\frac 12 (3\beta-1)}
=\beta^{-1} T_0,
\end{equation*}
we obtain, for $\epsilon=\beta\rho_1>0$,
\begin{equation}\label{on:sig}
\left|x(t)-\frac {T_0 }\beta \left[{\left(\frac t{T_0}+1\right)^{\beta}-1}\right] \right|
\lesssim T_0 \left(\frac t{T_0}+1\right)^{\beta} T_0^{-3\epsilon\frac {1-\beta}{3\beta-1}}(t+T_0)^{-\epsilon}
.
\end{equation}
Similarly, by \eqref{sur:lat}, we have
\begin{equation*}
\left|\ell(t)- \left(\frac{t}{T_0}+1\right)^{\frac {1-\beta}2}\right|
\lesssim t_0^{-\rho_1\frac {1-\beta}{2}} \left(\frac{t}{T_0}+1\right)^{\frac {1-\beta}{2}(1-\rho_1)} 
\lesssim T_0^{-\rho_1\frac {1-\beta}{3\beta-1}}\left(\frac{t}{T_0}+1\right)^{\frac {1-\beta}{2}(1-\rho_1)},
\end{equation*}
and so, possibly choosing a smaller $\epsilon>0$,
\begin{equation}\label{on:lam}
\left| \ell(t) - \left(\frac{t}{T_0}+1\right)^{\frac {1-\beta}2}\right|
\lesssim (t+T_0 )^{-\epsilon}.
\end{equation}
This justifies \eqref{def:ell} for $T_\delta=T_0$.

Moreover, it follows from \eqref{eq:eta} that
\begin{align*}
\|w\|_{L^2}&\lesssim \|\eta\|_{L^2}\lesssim t_0^{-\frac14(3\beta-1)}\lesssim T_0^{-\frac 12},\\
\|\partial_x w\|_{L^2} &\lesssim
 \lambda_0 \|\partial_x\eta\|_{L^2}\lesssim \lambda_0 t_0^{-\frac{1+\beta}4}\lesssim T_0^{-\frac 12}.
\end{align*}
Therefore, for arbitrary small $\delta>0$, it is enough to choose $x_0=x_0(\delta)$ large enough and take $T_{\delta}=T_0$, so that
in particular $T_{\delta}^{-1/2}\ll \delta$, which implies the first estimate in \eqref{def:eta}.

Finally, by the definition of $w(t,x)$ and then \eqref{est:eta:right_ext}, one has
\begin{equation*}
\int_{x\geq \frac12x(t)} w^2(t,x) dx 
=\int_{y\geq \frac 12\sigma(\lambda_0^3 t+t_0)} \eta^2(\lambda_0^3 t+t_0, y) dy
\lesssim (\lambda_0^3t+t_0)^{-\frac {3\beta-1}{2}}
\lesssim T_0 \left(\frac {t}{T_0}+1\right)^{-\frac {3\beta-1}2}
\end{equation*}
and similarly,
\begin{multline*}
\int_{x\geq \frac12x(t)} (\partial_x w)^2(t,x) dx 
 =\lambda_0^2 \int_{y\geq \frac 12\sigma(\lambda_0^3 t+t_0)} (\partial_x\eta)^2(\lambda_0^3 t+t_0, y) dy\\
\lesssim \lambda_0^2 (\lambda_0^3t+t_0)^{-\frac {3\beta-1}{2}}
\lesssim \lambda_0^2 T_0 \left(\frac {t}{T_0}+1\right)^{-\frac {3\beta-1}2}.
\end{multline*}
These estimates complete the proof of \eqref{def:eta}.

\begin{remark}\label{rk:T0}
Estimates \eqref{on:sig}-\eqref{on:lam}
for $T_\delta=T_0\gg 1$ describe the behavior of the parameters both for large times and for intermediate times.
Indeed, by continuous dependence of the solution with respect to the initial data, since 
$\tilde u(t,x)=Q(x-t)$ is a solution, it is clear that $T_\delta\to +\infty$ as $\delta\to 0$, and that for
$0<t\ll T_\delta$, the solution $u(t)$ behaves like $\tilde u(t)$.\end{remark}

\subsection{Non-scattering solutions}\label{S:6.5}
We prove that the solution $U$ constructed in Theorem~\ref{th:2} does not behave in $L^2$ as $t\to +\infty$ like a solution of the linear Airy equation.
For the sake of contradiction, assume that
there exists $v_0\in L^2$ such that
defining $v(t,x)$ the solution of 
\begin{equation*}
\partial_t v + \partial_x^3 v= 0,\quad v(0)=v_0,
\end{equation*}
it holds
\begin{equation}\label{scat}
\lim_{t\to+\infty}\|U(t)-v(t)\|_{L^2} =0.
\end{equation}

We perform a monotonicity argument on $v$, similar to the one in \S\ref{S.5.4}.
Let $t\geq 0$.
For the same function $ \xi(\tau,x)=\chi\left(\frac{4x-\sigma(t)}{\sigma(\tau)}-2\right)$, define
\begin{equation*}
L(t)=\int v^2(\tau,x)\xi(\tau,x) dx
\end{equation*}
Then it follows from simple computations and \eqref{est:xi3} that
\begin{equation*}
\frac {dL}{d\tau }
=-3 \int (\partial_x v)^2 \partial_x \xi+\int v^2\partial_\tau\xi +\int v^2 \partial_x^3 \xi \lesssim \tau^{-3\beta}
\int v_0^2.
\end{equation*}
Let $ t_0>0$ be such that $\sigma(t_0)=\frac 1{8} \sigma(t)$ and $t_0\gtrsim t$.
Integrating on $[ t_0,t]$, and using the properties of $\chi$, we have
\begin{equation*}
\int_{x\geq \frac{\sigma(t)}2} v^2(t,x) dx
\leq L(t)\leq L( t_0)+c t^{- 3\beta+1} 
\lesssim \int_{x\geq 2\sigma(t_0)} v^2(t_0,x) dx + t^{- 3\beta+1}.
\end{equation*}
As $t\to +\infty$, by \eqref{scat} and then \eqref{eq:eta:right}, \eqref{sur:lat}, \eqref{sur:sit},
it holds
\begin{align*}
\int_{x\geq 2\sigma(t_0)} v^2(t_0,x) dx
&\lesssim \int_{x\geq 2\sigma(t_0)} U^2(t_0,x) dx +o(1)
\lesssim \int_{x\geq 2\sigma(t_0)} \mm ^2(t_0,x) dx +o(1)\\
&\lesssim \int_{y\geq \frac{\sigma(t_0)}{\lambda(t_0)}} Q^2(y)dy + o(1)
\lesssim e^{-ct^{\frac{3\beta-1}2}}+o(1)=o(1).
\end{align*}
Thus, 
\begin{equation*}
\lim_{t\to +\infty} \int_{x\geq \frac{\sigma(t)}2} v^2(t,x) dx =0,
\end{equation*}
but this is contradictory with \eqref{scat} since from \eqref{for:U} and \eqref{est:eta:right_ext}
\begin{equation*}
\lim_{t\to +\infty} \int_{x\geq \frac{\sigma(t)}2} U^2(t,x) dx =\int Q^2.
\end{equation*}

\end{document}